\documentclass[11pt,letterpaper]{amsart}

\usepackage{amsmath}
\usepackage{tikz-cd}

\usepackage{mathrsfs}
\usepackage{amsthm}
\usepackage{enumerate}
\usepackage{txfonts}

\usepackage{cite}
\usepackage{multirow}
\usepackage{hyperref}
\usepackage{comment}

\usepackage[utf8]{inputenc}
\usepackage[english]{babel}

\numberwithin{equation}{section}

\newtheorem{theo}{Theorem}[subsection]
\newtheorem{prop}[theo]{Proposition}
\newtheorem{lem}[theo]{Lemma}

\newtheorem{cor}[theo]{Corollary}

\newtheorem{conjecture}[theo]{Conjecture}
 
\theoremstyle{definition}
\newtheorem{exam}[theo]{Example}
\newtheorem{definition}[theo]{Definition}
\newtheorem{hypothesis}[theo]{Hypothesis}

\theoremstyle{remark}
\newtheorem{remark}[theo]{Remark}

\newcommand{\bb}[1]{\mathbb{#1}}
\newcommand{\bbf}[1]{\mathbf{#1}}
\newcommand{\n}[1]{\mathcal{#1}}
\newcommand{\f}[1]{\mathfrak{#1}}
\newcommand{\s}[1]{\mathscr{#1}}
\newcommand{\Hom}{\mbox{\normalfont Hom}}

\newcommand{\Gal}{\mbox{Gal}}

\DeclareMathOperator{\GL}{GL}

\DeclareMathOperator{\Tr}{Tr}
\DeclareMathOperator{\id}{id}

\DeclareMathOperator{\End}{End}
\DeclareMathOperator{\iEnd}{\End}

\DeclareMathOperator{\pr}{pr}

\DeclareMathOperator{\Lie}{Lie}

\DeclareMathOperator{\iHom}{\Hom}

\def \gr {{ \mathrm gr}}

\DeclareMathOperator{\an}{\begin{scriptsize}
an
\end{scriptsize}}
\DeclareMathOperator{\et}{\begin{scriptsize}
\acute{e}t
\end{scriptsize}}
\DeclareMathOperator{\proet}{\begin{scriptsize}
pro\acute{e}t
\end{scriptsize}}
\DeclareMathOperator{\ket}{\begin{scriptsize}
k\acute{e}t
\end{scriptsize}}
\DeclareMathOperator{\proket}{\begin{scriptsize}
prok\acute{e}t
\end{scriptsize}}

\DeclareMathOperator{\sm}{\begin{scriptsize}
sm
\end{scriptsize}}

\DeclareMathOperator{\dR}{\begin{scriptsize}
dR
\end{scriptsize}}

\DeclareMathOperator{\cyc}{\begin{scriptsize}
cyc
\end{scriptsize}}

\newcommand{\OBdr}{\s{O}\!\bb{B}_{\dR,\log}}

\newcommand{\OC}{\s{O}\!\bb{C}_{\log}}

\DeclareMathOperator{\sol}{\blacksquare}
\DeclareMathOperator{\Aut}{Aut}
\DeclareMathOperator{\iAut}{\Aut}
\DeclareMathOperator{\Rep}{Rep}
\DeclareMathOperator{\Sym}{Sym}
\DeclareMathOperator{\Spa}{Spa}

\DeclareMathOperator{\dlog}{dlog}
\DeclareMathOperator{\coker}{coker}

\DeclareMathOperator{\Sen}{Sen}

\DeclareMathOperator{\Solid}{Solid}

\begin{document}

\title{Geometric Sen theory over rigid analytic spaces}

\date{}
\author{Juan Esteban  Rodr\'iguez  Camargo}
\address{Max Planck Institute for Mathematics, Vivatsgasse 7,  53111 Bonn-Germany}
\email{rodriguez@mpim-bonn.mpg.de}


\maketitle

\begin{abstract}
In this work we develop geometric Sen theory  for rigid analytic spaces, generalizing  the previous work  of Pan for curves. We also extend the axiomatic Sen-Tate formalism of Berger-Colmez to a certain class of locally analytic representations. 
\end{abstract}

\tableofcontents

\section{Introduction}

Let $p$ be a prime number, $\Gal_{\bb{Q}_p}$ the absolute Galois group of $\bb{Q}_p$.  In the  work \cite{LuePan}, Lue Pan has introduced a new and powerful tool to compute pro\'etale cohomology of $\widehat{\s{O}}$-modules over rigid analytic spaces based on the axiomatic framework of Sen theory of Berger-Colmez \cite{BC1,BC2}.  The  objective of this paper is to generalize Pan's method from curves to log smooth adic spaces; our motivation to develop such a theory is to study of the Hodge-Tate structure of completed cohomology of Shimura varieties, this is carried out in the author's work \cite{RCLocAnACoho}.

Let $\bb{Q}_p^{\cyc}$ be the completion of the $p$-adic cyclotomic extension of $\bb{Q}_p$ and let $(C,C^{+})$ be a perfectoid field containing $\bb{Q}_p^{\cyc}$.  In the following we shall consider almost mathematics with respect to the ideal $\f{m}_C\subset C^+$ of topologically nilpotent elements, and write ``$=^{ae}$'' for a natural equivalence in the category of almost modules, (resp., ``$\cong^{ae}$'' for a non-natural isomorphism of almost modules).

 Let $X$ be a log smooth fs locally noetherian adic space over $(C,C^{+})$ with reduced normal crossing divisors,  we let $\widehat{\s{O}}^{(+)}_X$ and $\s{O}^{(+)}_X$ denote the structural sheaves of the  pro-Kummer-\'etale  and Kummer-\'etale sites of $X$ respectively, finally let $\Omega^{1}_X(\log)$ be the sheaf of log  differential forms of $X$, we refer to \cite{DiaoLogarithmic2019} for the theory of log geometry on adic spaces. 

  In order to state the main theorem in geometric Sen theory we need the following definition.

\begin{definition}
\label{DefRelLocAnIntro}
 A pro-Kummer-\'etale $\widehat{\s{O}}_X$-module   $\s{F}$  over $X$  is said \textit{relative locally analytic ON Banach}\footnote{Where \textit{ON} comes from orthonormalizable.} if there is a Kummer-\'etale  cover $\{U_j\}_{j\in J}$ of $X$ such that,  for all $j$,  the restriction $\s{F}|_{U_j}$ admits a $p$-adically complete  $\widehat{\s{O}}^+_X$-lattice  $\s{F}^0_j$,  and there is   $\epsilon>0$ (depending on $j$) such that $\s{F}^0_j /p^{\epsilon}{\cong}^{ae} \bigoplus_{I_j} \s{O}^+_X/p^{\epsilon} $, for $I_j$ an index set.  
\end{definition}

\begin{remark}
The relative locally analytic condition can be considered as a ``smallness'' condition on the sheaf $\s{F}$ in the sense of Faltings  \cite{FaltingsSimpsonI}. Examples of relative locally analytic ON $\widehat{\s{O}}_X$-modules are $\widehat{\s{O}}_X$-vector bundles, this can be deduced from $v$-descent of vector bundles on perfectoid spaces \cite[Lemma 17.1.8]{ScholzeWeinspadicgeometry}.  See Example \ref{ExampleRelativeLocAnRep} down below. 
\end{remark}

\begin{theo}[Theorem \ref{TheoGluingCaseGamma}]
\label{TheoSenOperatorIntro}
Let $\s{F}$ be a relative locally analytic ON Banach $\widehat{\s{O}}_X$-sheaf over $X$. Then  there is an $\widehat{\s{O}}_X$-linear Higgs field 
\[
\theta_{\s{F}}:\s{F}\to \s{F}\otimes_{\s{O}_X} \Omega^{1}_X(\log) (-1),
\]
(i.e. $\theta_{\s{F}}\wedge \theta_{\s{F}}=0$) called the geometric Sen operator, satisfying the following properties:
\begin{enumerate}

\item The formation of $\theta_{\s{F}}$ is functorial in $\s{F}$.

\item Let $\nu:X_{\proket}\to X_{\ket}$ be the projection from the pro-Kummer-\'etale site to the Kummer-\'etale site, then there is a natural equivalence
\[
R^{i}\nu_* \s{F}= \nu_* H^{i}(\theta_{\s{F}},\s{F}),
\]
where $H^{i}(\theta_{\s{F}},\s{F})$ is the cohomology of the Higgs complex
\[
0\to \s{F}\xrightarrow{\theta_{\s{F}}} \s{F}\otimes_{\s{O}_X} \Omega_X^{1}(\log)(-1) \to\cdots \to \s{F}\otimes_{\s{O}_X} \Omega^i_X(\log)(-i) \to \cdots \to \s{F}\otimes_{\s{O}_X} \Omega_{X}^{d}(\log)(-d)\to 0.
\]

\item Suppose that $\theta_{\s{F}}=0$, then  $\nu_* \s{F}$ is locally on the Kummer-\'etale topology of $X$ an ON Banach $\s{O}_X$-module\footnote{This means that locally in the Kummer-\'etale topology $\nu_* \s{F}$ admits a Banach basis over $\s{O}_X$.} and  $\s{F}= \widehat{\s{O}}_X \widehat{\otimes}_{\s{O}_{X}}\nu_*\s{F}$. Conversely, for any  ON Banach $\s{O}_X$-module $\s{G}$ the pullback $\widehat{\s{O}}_X\widehat{\otimes}_{\s{O}_X} \s{G}$ has trivial Sen operator.

\item  If $X$ has a form $X'$ over a discretely valued field with perfect residue field $(K,K^+)$, then $\theta_{\s{F}}$ is Galois equivariant. In particular, we recover the natural splitting 
\[
R\nu_* \widehat{\s{O}}_X = \bigoplus_{i=0}^d \Omega_X^{i}(\log)(-i)[-i],
\]
deduced from \cite[Proposition 3.23]{ScholzePerfectoidSurvey} in the non-log case, and from \cite[Theorem 3.2.4]{DiaoLogarithmicHilbert2018} in the log case. 

\item Let $f:Y\to X$ be a morphism of smooth fs log adic spaces over $(C,C^+)$. Then there is a commutative diagram
\[
\begin{tikzcd}
f^{*}\s{F} \ar[r, "f^* \theta_{\s{F}}"] \ar[rd, "\theta_{f^*\s{F}}"']& f^{*}\s{F} \otimes_{\s{O}_Y} f^{*}\Omega_X^{1}(\log)(-1) \ar[d, "\id\otimes f^*"] \\ 
			& f^{*}\s{F}\otimes_{\s{O}_Y} \Omega^{1}_Y(\log)(-1).
\end{tikzcd}
\]

\end{enumerate}
\end{theo}

The property of being a relative locally analytic $\widehat{\s{O}}_X$-module might look a bit mysterious. Nevertheless,  these sheaves arise naturally when studying locally analytic vectors of pro\'etale cohomology.  Let us explain in which context they appear.  Let $G$ be a compact $p$-adic Lie group and $\widetilde{X}\to X$ a pro-Kummer-\'etale $G$-torsor (e.g.  take $X$ a finite level modular curve and $\widetilde{X}$ the perfectoid modular curve).    Let $V$ be a $\bb{Q}_p$-Banach locally analytic representation of $G$,  for example,  we can take $V= \s{O}(\bb{G})$,  for some group affinoid neighbourhood $\bb{G}$ of $G$,  endowed with the left regular action.    Then $V$ defines a pro-Kummer-\'etale sheaf $V_{\ket}$ over $X$  by descending the $G$-representation $V$ along the torsor $\widetilde{X}\to X$.   By   Lemma \ref{LemmaBanachLocAn} down below,  there is a lattice $V^0\subset V$,  $\epsilon>0$, and an open subgroup $G_0\subset G$,  such that $G_0$ stabilizes $V^0$,  and that the action of $G_0$ on $V^0/p^{\epsilon}$ is trivial.   Therefore,  the pro-Kummer-\'etale $\widehat{\s{O}}_X$-module $V_{\ket}\widehat{\otimes}_{\bb{Q}_p} \widehat{\s{O}}_X$ is a relative locally analytic sheaf.  Indeed,  the restriction of $ V_{\ket}\widehat{\otimes}_{\bb{Q}_p} \widehat{\s{O}}_X$ to $X_{G_0}:=\widetilde{X}/G_0$ satisfies the conditions of  Definition \ref{DefRelLocAnIntro}.  In this situation we have a  more refined result.

\begin{theo}[Theorem \ref{TheoSenBundle}]
\label{TheoSenBundleofTorsorIntro}
Let $X$ be an fs log smooth adic space over  $\Spa(C,C^+)$ with log structure given by normal crossing divisors. Let   $G$ a $p$-adic Lie group and $\widetilde{X}\to X$ a  pro-Kummer-\'etale $G$-torsor.     Then there is an $\widehat{\s{O}}_X$-linear map\footnote{Here we see $\Lie G$ as the adjoint representation of $G$.}
\[
\theta_{\widetilde{X}}:  \widehat{\s{O}}_X  \otimes_{\bb{Q}_p}  (\Lie G)^{\vee}_{ \ket} \to \widehat{\s{O}}_X(-1)     \otimes_{\s{O}_{X}} \Omega_X^1(\log) 
\]
such that $\theta_{\widetilde{X}}\bigwedge \theta_{\widetilde{X}}=0$, called the geometric Sen operator of the torsor $\widetilde{X}$. Moreover, for any locally analytic Banach representation $V$ of $G$, we have a commutative diagram
\[
\begin{tikzcd}
V_{\ket}\widehat{\otimes }_{\bb{Q}_p} \widehat{\s{O}}_X \ar[r, "d_V\otimes \id_{\widehat{\s{O}}}"] \ar[rd,"\theta_{V}"'] & (V_{\ket}\widehat{\otimes}_{\bb{Q}_p} \widehat{\s{O}}_X )  \otimes_{\bb{Q}_p} (\Lie G)_{\ket}^{\vee}  \ar[d,"\id_{V}\otimes \theta_{\widetilde{X}}"] \\
& (V_{\ket} \widehat{\otimes}_{\bb{Q}_p} \widehat{\s{O}}_X(-1)) \otimes_{\s{O}_X}  \Omega^{1}_{X}(\log)
\end{tikzcd}
\]
such that $d_V:V\to  V\otimes_{\bb{Q}_p} (\Lie G)^{\vee}$ is the connection  induced by  derivations, and $\theta_V$ is the geometric Sen operator of $V_{\ket}\widehat{\otimes}_{\bb{Q}_p}\widehat{\s{O}}_X$ of Theorem \ref{TheoSenOperatorIntro}.

\begin{remark}
Let us explain the meaning of the condition $\theta_{\widetilde{X}}\wedge \theta_{\widetilde{X}}=0$ in Theorem \ref{TheoSenBundleofTorsorIntro}. The map $\theta_{\widetilde{X}}$ is adjoint to a map 
\[
\theta : \widehat{\s{O}}_X\to (\Lie G)_{\ket}\otimes_{\bb{Q}_p} \widehat{\s{O}}_X(-1)\otimes_{\s{O}_X} \Omega^1_{X}(\log)=:\n{E}.
\]
The condition $\theta_{\widetilde{X}}\wedge \theta_{\widetilde{X}}=0$ means that $\theta\wedge \theta =0$ in $\wedge^2 \n{E}$. 
\end{remark}

Moreover, let  $ H\to G$ be a morphism of $p$-adic Lie groups,    let  $Y$ an fs log smooth adic space over $(C,C^+)$  and  let $\widetilde{Y}\to Y'$ be an $H$-torsor.  Suppose we are given with a commutative diagram compatible with the group actions
\begin{equation*}
\begin{tikzcd}
\widetilde{Y}  \ar[r] \ar[d] & \widetilde{X} \ar[d]  \\ 
Y  \ar[r,"f"]& X
\end{tikzcd}.
\end{equation*}
Then the following square is commutative 
\begin{equation*}
\begin{tikzcd}
 f^*(\Lie G)^{\vee}_{\ket}\otimes_{\bb{Q}_p} \widehat{\s{O}}_Y \ar[d] \ar[r, "f^{*} \theta_{\widetilde{X}}"] & f^* \Omega_X^1(\log) \otimes_{\s{O}_Y} \widehat{\s{O}}_Y(-1)   \ar[d] \\
(\Lie H)^{\vee}_{\ket} \otimes_{\widehat{\bb{Q}}_p} \widehat{\s{O}}_Y \ar[r,"\theta_{\widetilde{Y}}"] &   \Omega_Y^1(\log) \otimes_{\s{O}_Y} \widehat{\s{O}}_Y(-1) 
\end{tikzcd}
\end{equation*}
\end{theo}

Let $\Sen_{\widetilde{X}}: \Omega^{1}_X(\log)^{\vee}\otimes_{\s{O}_X}\widehat{\s{O}}_X(1)\to (\Lie G)_{\ket}\otimes_{\bb{Q}_p} \widehat{\s{O}}_X$ be the dual of the Sen operator $\theta_{\widetilde{X}}$.  As a corollary of the previous theorem we deduce that the locally analytic vectors of the torsor $\widetilde{X}$ satisfy some differential equations.

\begin{cor}[Corollary \ref{corVanishingSenOpaofjaowd}]
\label{CoroTrivialActionSenIntro}
Let $|\widetilde{X}|$ be the underlying topological space of $\widetilde{X}$ and let $\widehat{\s{O}}_{X}|_{|\widetilde{X}|}$ be the restriction of $\widehat{\s{O}}_X$ to  a sheaf on $|\widetilde{X}|$.  Let $\s{O}^{la}_{\widetilde{X}} \subset \widehat{\s{O}}_X|_{|\widetilde{X}|}$ be the subsheaf consisting of locally analytic sections for the action of $G$. Then,  the action of $\s{O}^{la}_{\widetilde{X}}\otimes_{\bb{Q}_p} \Lie G$ on $\s{O}^{la}_{\widetilde{X}}$ by derivations vanishes when restricted to the image of the Sen map $\Sen_{\widetilde{X}}: \Omega^{1}_{X}(\log)^{\vee} \otimes_{\s{O}_X} \s{O}^{la}_{\widetilde{X}}(1)\to \Lie G \otimes_{\bb{Q}_p}  \s{O}^{la}_{\widetilde{X}}$. 
\end{cor}

To motivate the construction of the geometric Sen operator over rigid spaces,  we recall how the arithmetic Sen operator is defined.

Let $\bb{Q}_p(\zeta_{p^{\infty}})$ be the cyclotomic extension and $\bb{Q}^{cyc}_p$ its $p$-adic completion, let $\bb{C}_p$ be the $p$-adic completion of an algebraic closure of $\bb{Q}_p$ containing $\bb{Q}_p(\zeta_{p^{\infty}})$. We  denote $H=\Gal(\overline{\bb{Q}}_p/\bb{Q}_p(\zeta_{p^{\infty}}))$ and $\Gamma=\Gal(\bb{Q}_p(\zeta_{p^{\infty}})/\bb{Q}_p)$. Let $V$ be a finite-dimensional representation  of $\Gal_{\bb{Q}_p}$ over $\bb{Q}_p$, the Sen module $\Sen(V)$ attached to $V$ is the finite-dimensional $\bb{Q}_p(\zeta_p^{\infty})$-vector space consisting of the  finite $\Gamma$-vectors  of $(V\otimes_{\bb{Q}_p} \bb{C}_p)^{H}$. It turns out that the dimension of $\Sen(V)$ over $\bb{Q}_p(\zeta_{p^{\infty}})$ is equal to $\dim_{\bb{Q}_p} V$. Indeed, one has a $\Gal_{\bb{Q}_p}$-equivariant isomorphism 
\[
\Sen(V)\otimes_{\bb{Q}_p(\zeta_p^{\infty})} \bb{C}_p = V\otimes_{\bb{Q}_p} \bb{C}_p.
\] 
As it is explained in \cite{BC2}, an equivalent way to construct $\Sen(V)$ is by taking the $\Gamma$-locally analytic vectors of $ (V\otimes_{\bb{Q}_p} \bb{C}_p)^{H}$, in particular $\Sen(V)$ admits an action of $ \Lie \Gamma$. We let $\theta_{\bb{Q}_p}\in \Lie \Gamma$ be the natural basis arising from the clyclotomic character $\chi_{\cyc}:\Gamma\cong \bb{Z}_p^{\times}$. Since $\Gamma$ acts locally constantly on $\bb{Q}_p(\zeta_{p^{\infty}})$ (i.e. for each $v\in \bb{Q}_p(\zeta_{p^{\infty}})$ the orbit $\Gamma\cdot  v$ is finite), $\theta_{\bb{Q}_p}$ is a $\bb{Q}_p(\zeta_{p^{\infty}})$-linear operator. One defines the Sen operator of $V\otimes_{\bb{Q}_p} \bb{C}_p$ to be the $\bb{C}_p$-extension of scalars of $\theta_{\bb{Q}_p}$.  

Summarizing, one can consider $\Sen(V)$ as a decompletion of the semilinear representation $V\otimes_{\bb{Q}_p} \bb{C}_p$ by taking locally analytic vectors along the ``perfectoid cyclotomic coordinate'' $\bb{Q}_p^{cyc} / \bb{Q}_p$.  The Sen operator is then a ``differential operator'' obtained from the cyclotomic extension.

Let us now sketch the definition of the geometric Sen operator. For simplicity,  we assume that   $X=\bb{T}= \Spa(C\langle  T^{\pm 1} \rangle,  C^+\langle T^{\pm 1} \rangle)$ is a one dimensional torus.   Let $\s{F}$ be a relative locally analytic sheaf over $\bb{T}$,  suppose in addition that there is a  $p$-adically complete lattice $\s{F}^0\subset \s{F}$,  and $\epsilon>0$,  such that $\s{F}^0/p^{\epsilon}=^{ae} \bigoplus_I \s{O}^+_X/p^{\epsilon}$.     We want to compute the (geometric) pro\'etale cohomology $R\Gamma_{\proet}(\bb{T},  \s{F})$.  Let 
  \[
  \bb{T}_{n}= \Spa(C\langle T^{\pm 1/p^n} \rangle, C^+\langle T^{\pm 1/p^n} \rangle),\]  and let $\bb{T}_{\infty}= \varprojlim_n \bb{T}_{n}$ be the perfectoid torus.   The perfectoid torus $\bb{T}_{\infty}$ is a Galois cover of $\bb{T}$ with group $\Gamma= \bb{Z}_p(1)$.   By Scholze's almost acyclicity of $\s{O}^+_X/p$ in affinoid perfectoid spaces \cite[Proposition  7.13]{ScholzePerfectoid2012},  one has that 
\[
R\Gamma_{\proet}(\bb{T},  \s{F})= R\Gamma(\Gamma,  \s{F}(\bb{T}_{\infty})).
\]
By the comparison theorem between continuous and locally analytic cohomology  (Theorem \ref{TheoResumeLocAn}), one has that 
\begin{align*}
R\Gamma_{\proet}(\bb{T}, \s{F}) & = R\Gamma(\Gamma,  \s{F}(\bb{T}_{\infty})^{R\Gamma-la}) \\ 
& = R\Gamma(\Gamma^{sm}, R\Gamma(\Lie \Gamma , \s{F}(\bb{T}_{\infty})^{R\Gamma-la})),
\end{align*}
 where $V^{R\Gamma-la}$ are the derived locally analytic vectors of a $\Gamma$-representation $V$, $R\Gamma(\Lie \Gamma,-)$ is the Lie algebra cohomology (equivalently, the cohomology of a Chevalley-Eilenberg complex), and $R\Gamma(\Gamma^{sm},-)$ is the smooth group cohomology (equivalently, group cohomology for smooth or locally constant cochains),  see \S \ref{SubsecRecollectionsLocAnRep} for the definitions.   In other words,  we have separated the problem of computing pro\'etale cohomology in three steps:  first,  we need to compute the derived locally analytic vectors of $\s{F}(\bb{T}_{\infty})$.  Second,  we take the Lie algebra  cohomology of  $\s{F}(\bb{T}_{\infty})$,  and finally,  we take the $\Gamma$-invariants of a smooth representation.  

Let us focus in the first step which seems to be the more subtle.   For $n\geq m$ there are normalized traces
\[
R_{m}^n:  C\langle T^{\pm 1/p^n} \rangle \to  C \langle T^{\pm 1/p^m} \rangle
\]
where $R^n_m= \frac{1}{p^{n-m}} \sum_{\sigma\in p^m\Gamma/p^n\Gamma} \sigma$.   These extend to Tate traces 
\[
R_m:  C\langle T^{\pm 1/p^\infty} \rangle \to  C \langle  T^{\pm 1/p^{m}}\rangle
\]
such that,  for any $f\in C\langle T^{\pm 1/p^{\infty}}\rangle$,  the sequence $(R_{m}(f))_m$ converges to $f$.  Furthermore,  the tuple $(C\langle T^{\pm 1/p^{\infty}} \rangle,  \Gamma)$ satisfies the Colmez-Sen-Tate axioms of \cite{BC1}. 

Let  $(\zeta_{p^n})_n$ be a compatible system of primitive $p$-th power roots of unity,  and let $\psi:  \bb{Z}_p\xrightarrow{\sim} \Gamma$ be the induced isomorphism.   Using $\psi$ we define the affinoid group $\bb{G}_n$ which is a copy of the  additive group of radius $p^{-n}$ (i.e.  $\bb{G}_n(\bb{Q}_p)= p^n \bb{Z}_p$).    We will keep using the expression ``$p^n\Gamma$-analytic'' instead of $\bb{G}_n$-analytic.    The following theorem is a generalization of \cite[Proposition   3.3.1]{BC1} to relative locally analytic sheaves,  it can be seen as a decompletion theorem \`a la Kedlaya-Liu  \cite{KedlayaLiu2019relative} using locally analytic vectors.  

\begin{theo}[Theorem  \ref{TheoSenFunctor}]
There exists $n \gg 0$ depending on $\epsilon$ such that 
\[
\s{F}(\bb{T}_{\infty})= C\langle T^{\pm 1/p^{\infty}} \rangle \widehat{\otimes}_{C\langle T^{\pm 1/p^{n}} \rangle} \s{F}(\bb{T}_{\infty})^{p^n\Gamma-an}.
\]
Moreover,  we have that 
\[
\s{F}(\bb{T}_{\infty})^{R\Gamma-la}= \s{F}(\bb{T}_{\infty})^{\Gamma-la}= \varinjlim_{m} C\langle T^{\pm 1/p^m} \rangle \otimes_{C\langle T^{\pm 1/p^{n}} \rangle} \s{F}(\bb{T}_{\infty})^{p^n\Gamma-an}. 
\]
\end{theo}

The previous theorem shows that,  under certain conditions on $\s{F}$,  the derived locally analytic vectors of $\s{F}(\bb{T}_{\infty})$ are concentrated in degree $0$,  and that all the relevant information is already encoded in the $p^m\Gamma$-analytic vectors for some $m\gg 0$.   In particular,  we have that 
\begin{equation}
\label{eqComputationCohoSenOperatorsLocAn}
R\Gamma_{\proet}(\bb{T}_{\infty},  \s{F}) = R\Gamma(\Gamma^{sm}, R\Gamma(\Lie \Gamma , \s{F}(\bb{T}_{\infty})^{p^n\Gamma-an})).
\end{equation}

Thus,  the problem of computing pro\'etale cohomology has been reduced to the problem of computing Lie algebra cohomology.   The module $ \s{F}(\bb{T}_{\infty})^{p^n\Gamma-an}$ is not mysterious at all,  it is  the Higgs bundle attached to $\s{F}$ via the local $p$-adic Simpson correspondence of $\bb{T}$.

Now,  the action of $\Lie \Gamma$ is $C\langle T^{\pm  1/p^n}\rangle$-linear and $\Gamma$-equivariant.  By extending scalars, it induces a $\Gamma$-equivariant $C\langle  T^{\pm 1/p^{\infty}} \rangle $-linear action on $\s{F}(\bb{T}_{\infty})$.   Moreover, if $X$ has a form over $(K,K^+)$,   this action is $\Gal_K$-equivariant,  where $\Gal_K$ acts on $\Lie \Gamma$ via the cyclotomic character.   By  Proposition \ref{PropProjectionO}, $\Lie \Gamma \otimes C \langle  T^{\pm 1/p^{\infty}}\rangle= \Omega_{\bb{T}}^{1,\vee}(\bb{T}) \otimes  C \langle  T^{\pm 1/p^{\infty}}\rangle(1)$,  this shows that the action of $\Lie \Gamma$ defines an $\widehat{\s{O}}_X$-linear map of pro\'etale sheaves over $\bb{T}$
\[
\Sen_{\s{F}}:  \Omega^{1,\vee}_{\bb{T}} \otimes_{\s{O}_X} \s{F} (1) \to \s{F},
\]
or equivalently,  it defines  the geometric Sen operator  
\[
\theta_{\s{F}}:  \s{F} \to  \s{F}\otimes_{\s{O}_X} \Omega^1_{\bb{T}}(-1).  
\]

 In general, we shall construct the Sen operators locally on the Kummer-\'etale topology of  $X$ by taking charts to products of tori and closed affinoid discs.  The fact that $\theta_{\s{F}}$ does not depend on the charts and that it is functorial with respect to the sheaf and the space will be proved in \S \ref{ss:GeoSenOpe-Globalization}.

\subsection*{An overview of the paper}

In \S \ref{SubsecRecollectionsLocAnRep} we briefly review some basic facts of the theory of solid locally analytic representations of the author and Rodrigues Jacinto \cite{RRLocallyAnalytic,RRLocAn2}. In \S \ref{s:SenTheoryAxioms}-\ref{s:SenTheoryCohomology}, we extend the abstract formalism of Sen  theory of Berger-Colmez \cite{BC1} that will be used in the main applications of the paper. 

In \S \ref{ss:LogKummerSequence} we extend the computation of $R^1\nu_* \widehat{\s{O}}_X(1)=\Omega_X^1$ of  \cite{ScholzePerfectoidSurvey} from smooth to log smooth adic spaces; this follows formally from Scholze's proof using the formalism of log adic spaces. In  \S \ref{s:HTSrigidSenBundle} we construct the geometric Sen operator in local coordinates, proving local versions of Theorems \ref{TheoSenOperatorIntro} and \ref{TheoSenBundleofTorsorIntro}, then, in \S \ref{ss:GeoSenOpe-Globalization} we prove the global versions of  the theorems. In \S \ref{s:HTSrigidProetalecoho} we study locally analytic vectors of torsors of $p$-adic Lie groups. We finish the section in \S \ref{ss:HTSrigidpAdicSimpson} by discussing  the relation of  geometric Sen theory with the $p$-adic Simpson correspondence of \cite{LiuZhuRiemannHilbert,  DiaoLogarithmicHilbert2018, wang2021padic}.

\subsection*{Acknowledgments}

First and foremost,  I would like to thank my advisor Vincent Pilloni who guided me throughout the development of this project.  I want to thank George Boxer and Joaqu\'in Rodrigues Jacinto for their patience and enlightenment after trespassing to their office several times for discussing different stages of the work.   I am grateful  to the  participants of the study group in Lyon in the Spring of 2021 on Pan's work on locally analytic completed cohomology for the modular curve.    I would like to thank  Laurent Berger for his comments on an earlier version of the Sen theory formalism.   I also want to thank Fabrizio Andreatta, Gabriel Dospinescu, Andrew Graham,  Valentin Hernandez,  Sean Howe,  Ben Heuer,    Damien Junger and Arthur-C\'esar Le Bras for fruitful exchanges. Special thanks to Lue Pan for many conversations that enhanced the content of the document, in particular his suggestion to improve a cohomological computation that lead to  Proposition \ref{PropositionSenTheoryCohomology}. I want to thank the Max Planck Institute for Mathematics for the excellent mathematical exchanges that helped me to give a  better presentation of this paper. Finally, I want to thank the anonymous referee for the many corrections and suggestions that notably improved and simplified the exposition of the document.


\section{Abstract Sen theory after Berger-Colmez}
\label{ch:SenTheory}

Sen theory has shown to be a powerful tool in the Galois theory of $p$-adic fields.   For example,  it  is used to compute Galois cohomology over period rings:

\begin{prop}[{\cite[Proposition  8]{Tatepdivisible}} ]
Let $\bb{C}_p$ be the $p$-adic completion of an algebraic closure of $\bb{Q}_p$,  and let $G_{\bb{Q}_p}$ be the absolute Galois group.  For $i\in \bb{Z}$ we let $\bb{C}_p(i)$ denote the $i$-th Tate twist.  Then 
\[
H^k(G_{\bb{Q}_p},  \bb{C}_p(i))=\begin{cases}  0 & \mbox{if } i\neq 0   \\
\bb{Q}_p & \mbox{if } i=0 \mbox{ and } k=0 \\
\bb{Q}_p \log \chi_{\cyc} & \mbox{if } i=0 \mbox{ and } k=1
\end{cases}
\] 
where $\chi_{\cyc}: G_{\mathbb{Q}_p}\to \mathbb{Z}_p^{\times}$ is the cyclotomic character, and $\log: \mathbb{Z}_p^{\times}\to \mathbb{Z}_p$ is the logarithm map defined as $\log(a)= \log(a/[a])$, with $[a]\in \mu_{p-1}(\mathbb{Z}_p)$ the Teichm\"uller lift of $\overline{a}\in \mathbb{F}_p^{\times}$. 
\end{prop}

In \cite{BC1},  Berger-Colmez defined an axiomatic framework where Sen theory can be applied.    Using this formalism, different constructions  attached to  finite dimensional Galois representations become  formally the same:  the Sen module (relative to $\bb{Q}_p^{\cyc}$),  the overconvergent $(\varphi,\Gamma)$-module  (relative to $\widetilde{\bb{B}}^{\dagger}(\bb{Q}_p^{\cyc})$),  the module $\bbf{D}_{\mathrm{diff}}$ of differential equations (relative to $\bb{B}_{\dR}(\bb{Q}_p^{\cyc})$).   Moreover,  using   Sen theory   Berger-Colmez describe in  \cite{BC2}   the locally analytic vectors of completed Galois extensions of $\bb{Q}_p$ with group  isomorphic to a $p$-adic Lie group.   

The work of Lue Pan \cite{LuePan} is an excellent application of this tool.  Inspired from the work of Berger-Colmez,  Pan describes  the $p$-adic Simpson correspondence of the $\GL_2(\bb{Q}_p)$-equivariant local systems of the modular curves in terms of  sheaves over the flag variety.  Furthermore,  based on the strategy of \cite{BC2},   he manages  to use the axiomatic Sen theory to describe the Hodge-Tate structure of locally analytic  ``interpolations'' of finite rank local systems.   

The main goal  of this section is to provide a more conceptual understanding of the work of Pan and Berger-Colmez via the theory of (solid) locally analytic representations.  More precisely,   we will prove that the construction of the Sen module holds not only for finite rank representations,  but for a larger class of orthonormalizable (also denoted ON in this paper) Banach locally analytic representations.

\subsection{Recollections in locally analytic representations}
\label{SubsecRecollectionsLocAnRep}

In this  section we recall some tools from the theory of locally analytic representations of compact $p$-adic Lie groups. We will be only concerned in continuous representations on (colimits of) Banach spaces, so everything can be stated in the classical language of Schneider-Teitelbaum  \cite{SchTeitGL2,SchTeitDist} and Emerton \cite{Emerton}. 
However, we will use some technology from the theory of solid locally analytic representations, as developed in \cite{RRLocallyAnalytic} and \cite{RRLocAn2}, that will facilitate some technical arguments. For example, we will  consistently use the general comparison theorem between continuous and locally analytic cohomology. 

Let $G$ be a compact $p$-adic  Lie group and let $C^{la}(G,\bb{Q}_{p})$ be the $LB$ space of locally analytic functions of $G$, it admits a description as a colimit $C^{la}(G,\bb{Q}_p)=\varinjlim_h C^h(G,\bb{Q}_p)$ of $|p^{h}|$-analytic (or $h$-analytic for simplicity) functions of $G$, which depends on the choice of an uniform pro-$p$-subgroup $G_0\subset G$, see \cite[Section 2]{RRLocAn2}. Throughout this paper we will take an arbitrary choice of that $G_0$, and we will freely change  $G_0$ by a smaller subgroup when necessary.

 We let $C^{la}(G,\bb{Q}_p)_{\star_{1}}$ denote the left regular action (resp. $\star_2$ for the right regular action).  Let $\bb{Q}_{p}[[G]]=\bb{Z}_p[[G]][\frac{1}{p}]$ denote the rational Iwasawa algebra of $G$, seen as a solid $\bb{Q}_p$-algebra. Let $\Solid(\bb{Q}_p[[G]])$ be the abelian category of solid $\bb{Q}_p[[G]]$-representations, or equivalently the category of $\mathbb{Q}_p$-linear solid $G$-representations, cf. \cite[Definition 4.20]{RRLocallyAnalytic}. We let $\s{D}_{\sol}(\bb{Q}_p[[G]])$ be the derived category of $\Solid(\bb{Q}_p[[G]])$. 

\begin{definition}
\label{DefLocAn}
Let $V\in \s{D}_{\sol}(\bb{Q}_p[[G]])$.
\begin{enumerate}

\item The \textit{solid group cohomology} of $V$ is given by 
\[
R\underline{\Gamma}_{\sol}(G,V):=R\underline{\iHom}_{\bb{Q}_p[[G]]}(\bb{Q}_p,V),
\]
where  $\underline{\iHom}_{\bb{Q}_p[[G]]}$ is the solid Hom space of $\bb{Q}_p[[G]]$-modules\footnote{The object $\mathbb{Q}_p[[G]]$ is an algebra in solid $\mathbb{Q}_p$-vectors spaces, this makes $\Solid(\mathbb{Q}_p[[G]])$ naturally enriched on solid $\mathbb{Q}_p$-vector spaces.}.  We denote by  $R\Gamma_{\sol}(G,V):=R\underline{\Gamma}_{\sol}(G,V)(*)$ the underlying $\bb{Q}_p$-vector space. We also write $\underline{H}^i(G,V)$ for the $i$-th cohomology group of $R\underline{\Gamma}_{\sol}(G,V)$ and $H^i(G,V)=\underline{H}^i(G,V)(*)$ for its underlying solid $\mathbb{Q}_p$-vector space. 

\item The \textit{derived locally analytic vectors} of $V$ is the solid $\mathbb{Q}_p[[G]]$-module  defined as 
\[
V^{RG-la}:= R\underline{\Gamma}_{\sol}(G, (V\otimes^{L}_{\bb{Q}_{p,\sol}} C^{la}(G,\bb{Q}_p))_{\star_{1,3}}),
\]
where $\otimes^{L}_{\bb{Q}_{p,\sol}} $ is the solid tensor product, and the $\star_{1,3}$ is the diagonal action on $V$ and $C^{la}(G,\bb{Q}_p)_{\star_1}$. The $G$-action on $V^{RG-la}$ arises from the right regular action on $C^{la}(G,\mathbb{Q}_p)$.  If $V\in \Solid(\bb{Q}_p[[G]])$ we denote $V^{G-la}:= H^{0}(V^{RG-la})$.  

\item We say that a solid $G$-representation $V$ is \textit{locally analytic} if the natural map $V^{RG-la}\to V$ is a quasi-isomorphism of solid $\mathbb{Q}_p[[G]]$-modules (equivalently of solid $\mathbb{Q}_p$-vector spaces). 

\end{enumerate}
\end{definition}

The following theorem summarizes the main features of the category of solid locally analytic representations. 

\begin{theo}
\label{TheoResumeLocAn}
Let $\Rep^{LC}_{\bb{Q}_p}(G)$ be the category of $\bb{Q}_p$-linear compactly generated  complete locally convex continuous representations of $G$. The following hold:

\begin{enumerate}

\item   There is a fully faithful inclusion 
\[
\underline{(-)}:\Rep^{LC}_{\bb{Q}_p}(G)\hookrightarrow \Solid(\bb{Q}_p[[G]]).
\]

\item Under the inclusion of (1), there is a natural quasi-isomorphism of $\bb{Q}_p$-vector spaces 
\[
R\Gamma_{\sol}(G,\underline{V})= R\Gamma(G,V)
\]
between solid and continuous cohomology. 

\item The  functor $(-)^{RG-la}$ is idempotent. Let $\s{D}_{\sol}(\mathbb{Q}_p)^{G^{la}}\subset \s{D}_{\sol}(\bb{Q}_p[[G]])$ be the full  subcategory of locally analytic representations. Then  $\s{D}_{\sol}(\mathbb{Q}_p)^{G^{la}}$ is stable under colimits and  tensor products. Moreover, the inclusion admits a right adjoint given by   the functor of derived locally analytic vectors $(-)^{RG-la}$. 

\item Let $V\in \s{D}_{\sol}(\bb{Q}_p[[G]])$, then there is a natural equivalence 
\[
R\underline{\Gamma}_{\sol}(G,V)=R\underline{\Gamma}(G^{sm}, R\underline{\Gamma}(\Lie G, V^{RG-la})),
\] 
where $R\underline{\Gamma}(\Lie G,-)$ is (the solid enriched) Lie algebra cohomology, and $R\underline{\Gamma}(G^{sm},-)$ is  (the solid enriched) smooth group cohomology. 

\end{enumerate}
\end{theo}

\begin{proof}
Parts (1) and (2) follow from Proposition 3.7 and Lemma 5.2 of \cite{RRLocallyAnalytic} respectively. Part (3)  follows from  Proposition 3.1.11 and Proposition   3.2.3  of \cite{RRLocAn2}. Finally, part (4) follows from  Theorems 5.3 and 5.5 of \cite{RRLocallyAnalytic}, or  Theorem 6.3.4 of  \cite{RRLocAn2}.
\end{proof}

\begin{remark}\label{x92ejw}
In Theorem \ref{TheoResumeLocAn} (4)  smooth  $G$-cohomology and Lie algebra cohomology are as in \cite[Definition 1.0.4]{RRLocAn2}, namely, the right adjoints of the natural trivial representation map from solid $\mathbb{Q}_p$-vector spaces to the derived categories of  smooth $G$-representations or $U(\Lie G)$-modules respectively.  For objects in the heart of the $t$-structure of $G$-smooth representations,  \cite[Proposition 6.3.3]{RRLocAn2} shows that  smooth group-cohomology  can be computed via smooth cochains. By the standard Koszul resolution of $\mathbb{Q}_p$ as $U(\Lie G)$-module, Lie algebra cohomology can also be computed via the explicit Chevalley-Eilenberg complex. 
\end{remark}

An immediate consequence of the theorem is the following corollary:

\begin{cor}
\label{CoroComparisonVla}
Let $V\in \Solid(\bb{Q}_p[[G]])$ be  solid $G$-representation. Suppose that $V^{RG-la}= V^{G-la}$, i.e., that the higher derived locally analytic vectors vanish. Then there are isomorphisms 
\[
\underline{H}^{i}_{\sol}(G,V)= \underline{H}^{i}(\Lie G, V^{G-la})^{G}
\] 
of solid $\mathbb{Q}_p$-vector spaces.  In particular, if $V$ arises from a compactly generated locally convex $\bb{Q}_p$-vector space, we have isomorphisms
\[
H^{i}(G,V)=H^{i}(\Lie G, V^{G-la})^{G}
\]
of $\mathbb{Q}_p$-vector spaces. 
\end{cor}
\begin{proof}
This follows from parts  (2) and (4) of Theorem \ref{TheoResumeLocAn}. 
\end{proof}

The property of being a locally analytic representation roughly means that the operators $[g]-[1]$ for $g\in G$ have norm $<1$. An example of this phenomenon appears in the following lemma.

\begin{lem}
\label{LemmaBanachLocAn}
Let $V$ be a  $\mathbb{Q}_p$-Banach representaion of $G$. The following hold: 

\begin{enumerate}

\item  If $V$ is locally analytic then for any $\mathbb{Z}_p$-lattice $V_0\subset V$ there is an open compact subgroup $G_{0}$ stabilizing $V_0$  such that $G_0$ acts trivially on $V_0/p$. 

\item Conversely, suppose that there is a finite extension $K/\mathbb{Q}_p$, with uniformizer $\varpi$ such that there is an $\mathcal{O}_K$-lattice of $W_0\subset V\otimes_{\mathbb{Q}_p} K$ and an open compact subgroup $G_0$ stabilizing $W_0$,  with $G_0$ acting trivially on $W_0/\varpi$. Then $V$ is locally analytic.

\end{enumerate}

Moreover,  if  $V$ is locally analytic, then it is analytic in the following sense:  there is an uniform pro-$p$-subgroup $G_0\subset G$ (cf. \cite[Section 4]{SchTeitDist}) with coordinates $G_0\cong \mathbb{Z}_p^d$ such that the orbit map $\mathcal{O}_V:V\to C(G_0,V)$ factors through $C^{\an}(G_0,V)$ where $C^{\an}(G_0,V)=C^{\an}(G_0,\mathbb{Q}_p) \widehat{\otimes}_{\mathbb{Q}_p} V$ and $C^{\an}(G_0,\mathbb{Q}_p)\cong \mathbb{Q}_p\langle T_1,\ldots, T_d \rangle$ is the space of analytic functions of $G_0$ with respect to the coordinates $(T_1,\ldots, T_d)$ of $\mathbb{Z}_p^d$.
\end{lem}
\begin{proof}

Let us prove part (1). Let $V$  be a locally analytic representation of $G$. By shrinking $G$ we can assume without loss of generality that $G$ is an uniform pro-$p$-group and that we have an isomorphism of $p$-adic manifolds $G\cong \mathbb{Z}_p^d$. The powers $G^{p^n}$ are then uniform pro-$p$-groups and  we can write $C^{la}(G,\mathbb{Q}_p)=\varinjlim_{h\to \infty} C^{h}(G,\mathbb{Q}_p)$  where $C^{h}(G,\mathbb{Q}_p)$ is the space of analytic functions of the cosets  $G=\bigsqcup_{g\in G/G^{p^h} } gG^{p^h}$.  Therefore, since $V$ is locally analytic, its orbit map is a map of solid $\mathbb{Q}_p$-vector spaces
\[
\n{O}_V^{la}:V\to \varinjlim_{h} V\widehat{\otimes}_{\mathbb{Q}_p}   C^{h}(G,\mathbb{Q}_p).
\]
By \cite[Lemma 3.32]{RRLocallyAnalytic} there is some $h$ such that $\mathcal{O}_V$ factors through $O_V^{h}:V\to V\widehat{\otimes}_{\mathbb{Q}_p}   C^{h}(G,\mathbb{Q}_p)$.  The map $\n{O}_V^h$ is  $G$-equivariant for the natural action on $V$ and  for the right regular action on the tensor. One can find a suitable $\bb{Z}_p$-lattice $V_0\subset V$ such that $O^{h}_V$ sends $V_0$ to $V_0\widehat{\otimes}_{\bb{Z}_p} C^{h}(G,\mathbb{Z}_p)$  (with $C^h(G, \mathbb{Z}_p)$ the space of $\mathbb{Z}_p$-valued $h$-analytic functions).  But then it is easy to see that for any $k\geq 1$ there is $h'>h$  such that the action of $G^{p^{h'}}$ on $C^h(G,\mathbb{Z}_p)/p^k$ is trivial (eg. since $C^h(G,\mathbb{Z}_p)/p^k$ is a $\mathbb{Z}/p^k$-algebra of finite type and $G$ acts via algebra morphisms, so that it suffices to trivialize the action on algebra  generators).  Pick $k$ such that $\mathcal{O}^h_V(V_0)\cap p^k( V_0\widehat{\otimes}_{\bb{Z}_p} C^{h}(G,\mathbb{Z}_p)) \subset \mathcal{O}^h_V(p V_0)$ (which is guarantee by continuity of the orbit map), then $V_0/p$ is a subquotient of $V_0\widehat{\otimes}_{\bb{Z}_p} C^{h}(G,\mathbb{Z}_p)/p^k$ and so it is a trivial representation of $G^{p^{h'}}$ as wanted.

Part (2) follows  from  \cite[Proposition 3.3.3]{RRLocAn2} but let us give the details how the proposition is used. First, since the category of locally analytic representations is stable under colimits on $\mathbb{Q}_p$-solid $G$-modules (Theorem \ref{TheoResumeLocAn} (3)), it suffices to show that $W=V\otimes_{\mathbb{Q}_p} K$ is locally analytic (as $V$ is a retract of $W$). Then, the same proof of \cite[Proposition 3.3.3]{RRLocAn2} applies for  $p$ replaced by the pseudo-uniformizer $\varpi$ and the $\varpi$-adically complete representation $W_0$.

The last claim about the analyticity of  a suitable uniform compact open subgroup $G_0\subset G$ follows from the proof of (1). 
\end{proof}

Finally, we will need a projection formula for the functor of locally analytic vectors\footnote{This projection formula holds for general $G$-equivariant $A$-semilinear solid representations; we keep the assumptions in the lemma for simplicity in the exposition.}. 

\begin{lem}
\label{LemmaProjectionFormulaLocAnRep}
Let $A$ be a Banach $\bb{Q}_p$-algebra endowed with a locally analytic action of $G$. Let $X$ and $M$ be  $G$-equivariant Banach $A$-modules such that $M$ has an ON basis over $A$, and  such that the action of $G$ on $M$ is locally analytic. Then there is a natural equivalence  of $\mathbb{Q}_p$-solid $G$-representations 
\begin{equation}
\label{eqwjfied}
 \underline{X}^{RG-la}\otimes_{\underline{A},\sol}^{L} \underline{M} \xrightarrow{\sim} (\underline{X\widehat{\otimes}_A M})^{RG-la} 
\end{equation}
\end{lem}
\begin{proof} 

 By the projection formula of locally analytic vectors \cite[Corollary 3.1.15 (3)]{RRLocAn2} we know that, when $A=\mathbb{Q}_p$, the natural map  \eqref{eqwjfied} is an equivalence.  We want to reduce the general case to this situation by applying the Bar construction of the relative tensor product, cf. \cite[Discussion before Proposition 8.6.3]{MR1269324} or \cite[Construction 4.4.2.7]{HigherAlgebra}.

First, since $M\cong \widehat{\bigoplus}_I A$, by \cite[Lemma 3.13]{RRLocallyAnalytic} we have that $\underline{M}=\underline{A}\otimes^{L}_{\bb{Q}_p,\sol} P$ with $P=\widehat{\bigoplus}_I \bb{Q}_p$. Then, \textit{loc. cit.} implies that 
\[
\underline{X\widehat{\otimes}_A M } \cong \underline{X\widehat{\otimes}_{\bb{Q}_p} P} = \underline{X}\otimes^{L}_{\bb{Q}_{p,\sol}} \underline{P} \cong \underline{X} \otimes^{L}_{\underline{A},\sol} \underline{M},
\]
in particular $ \underline{X} \otimes^{L}_{\underline{A},\sol} \underline{M}=\underline{X} \otimes_{\underline{A},\sol} \underline{M}$ sits in degree $0$. Then, since Banach spaces are flat solid $\bb{Q}_p$-vector spaces (see \cite[Lemma 3.21]{RRLocallyAnalytic}), the Bar resolution for the tensor product gives rise a $G$-equivariant long exact sequence  of solid $\mathbb{Q}_p$-vector spaces
\[
\cdots \to   \underline{X}\otimes_{\bb{Q}_p,\sol} \underline{A}^{\otimes n} \otimes_{\bb{Q}_p,\sol} \underline{M} \to  \cdots  \to  \underline{X}\otimes_{\bb{Q}_p,\sol} \underline{M}   \to   \underline{X} \otimes_{\underline{A},\sol} \underline{M}\to 0.
\]
Then, by the projection formula of \cite[Corollary 3.1.15 (3)]{RRLocAn2}, we have that 
\[
  (\underline{X}\otimes_{\bb{Q}_p,\sol} \underline{A}^{\otimes n} \otimes_{\bb{Q}_p,\sol} \underline{M})^{RG-la}=  \underline{X}^{RG-la}\otimes_{\bb{Q}_p,\sol} \underline{A}^{\otimes n} \otimes_{\bb{Q}_p,\sol} \underline{M}.
\]
Since $(-)^{RG-la}$ has finite cohomological dimension (\cite[Lemma 3.2.2]{RRLocAn2}), one deduces that 
\[
\underline{X}^{RG-la} \otimes^{L}_{\underline{A},\sol} \underline{M}\xrightarrow{\sim} (\underline{X} \otimes^{L}_{\underline{A},\sol} \underline{M})^{RG-la}
\]
 is an equivalence as wanted. 
\end{proof}

\begin{remark}\label{RemhminusAnRep}
The key input in the proof of Lemma \ref{LemmaProjectionFormulaLocAnRep}  is the projection formula for locally analytic vectors of \cite[Corollary 3.1.15 (3)]{RRLocAn2}. An analogue statement holds for $h$-analytic vectors for a suitable notion of $h$-analyticity. Indeed, consider $C^{h^-}(G, \bb{Q}_p)=\varinjlim_{h'<h} C^{h'}(G, \bb{Q}_p)$ the space of locally analytic functions of $G$ that has radius of analyticity  strictly bigger than $p^{-h}$.  By the same proof of \cite[Proposition 3.2.10]{RRLocAn2} the functor of $h^{-}$-analytic vectors defined as the group cohomology of $C^{h^-}(G, \bb{Q}_p)\otimes^L_{\bb{Q}_{p,\sol}}-$ is idempotent and so there is a well defined full subcategory of $h^-$-analytic representations of $\s{D}_{\sol}(\bb{Q}[[G]])$. Then the same proof of \cite[Corollary 3.1.15]{RRLocAn2} applies, which in particular yields a projection formula for $h^-$-representations. As consequence, Lemma \ref{LemmaProjectionFormulaLocAnRep} also holds for $h^-$-analytic representations.  
\end{remark}

\subsection{Colmez-Sen-Tate axioms}
\label{s:SenTheoryAxioms}

In this section we introduce the abstract framework of Sen theory that will be needed in this paper. Our discussion follows  \cite{BC1}. 

Let $F/\mathbb{Q}_p$ be an algebraic extension with $p$-adic completion $\widehat{F}$ a perfectoid field.   Let $A$ be a sous-perfectoid $\widehat{F}$-algebra endowed with  the spectral norm $|-|$.  We let $A^0\subset A$ be the $p$-complete open  $\mathbb{Z}_p$-algebra of elements of norm $\leq 1$, equivalently, the open subalgebra of power bounded elements.  For a closed Banach subvector space $W\subset A$ we let $W^0=W\cap A^0$ be the open subspace of $W$ of elements of norm $\leq 1$. Let $I\subset \mathbb{R}_{\geq 0}$ be the positive additive valuation monoid of $\overline{F}$ normalized such that $|p|=p^{-1}$, given $\epsilon \in I\backslash 0$ we let $\varpi^{\epsilon} \in F$ be an algebraic pseudo-uniformizer with absolute value $|\varpi^{\epsilon}|=|p|^{\epsilon}$.    The additive submonoid $I\subset \mathbb{R}_{\geq 0}$ is dense as $\widehat{F}$ is a  perfectoid field, from now on we always take $\epsilon \in I$.  In the following we shall consider almost mathematics with respect to the maximal ideal $\f{m}_{\widehat{F}}\subset \n{O}_{\widehat{F}}$. 

\subsubsection{One dimensional Sen theory}

Let $\Gamma\cong \mathbb{Z}_p$ be a one parameter group with generator $\gamma$. Suppose that $\Gamma$ acts continuously on $A$,  in particular it acts by isometries with respect to the spectral norm. For $k\in \mathbb{N}_{\geq 0}$ a non negative integer, let $\Gamma^{(k)}:=\Gamma^{p^k}\subset \Gamma$.

\begin{definition}[Colmez-Sen-Tate axioms]
\label{DefinitionSenAxioms}

We consider the following axioms on the pair $(A,\Gamma)$:

\begin{itemize}

\item[(CST1)] \textit{Tate's normalized traces}. We are given with the following datum:

\begin{itemize}

\item[(a)] For $n\gg 0$ a closed $\mathbb{Q}_p$-Banach subalgebra $A_{n}\subset A$.

\item[(b)]   For $n\gg 0$ we have $A_{n}$-linear and $\Gamma$-equivariant projection maps $R_{n}: A\to A_n$, i.e. maps such that the composite $A_n\to A\xrightarrow{R_n} A_n$ is the identity morphism\footnote{In particular, all the algebras $A_n$ are sous-perfectoid and $A_n^0=A^0\cap A_n$ is the subring of power-bounded elements. Hence, our notation  $W^0:=W\cap A^0$  for closed Banach subspaces $W\subset A$ is consistent with that of Banach rings.}. We write $X_n:=\ker (R_n)$.

\end{itemize}

The previous datum is subject to the following conditions: 

\begin{enumerate}

\item We have $F\subset \bigcup_{n} A_n$. We let $F_n:= \widehat{F}\cap A_n$ and let  $I_n\subset I$ be its additive valuation submonoid. For $\epsilon\in I_n$ we choose  $\varpi^{\epsilon}\in F_n$.

\item For all $\epsilon>0$ there is  $n_1(\epsilon)\gg 0$ such that for $n\geq n_1(\epsilon)$ and $x\in A^0$  we have $R_{n}(x)  \in {\varpi}^{-\epsilon} A_{n}^0$.   In other words $||R_{n}(x)|| \leq |\varpi^{-\epsilon}| |x|$ for $x\in A$.    

\item Given $x\in A$  the sequence $(R_n(x))_{n\in \mathbb{N}}$ converges to $x$ in $A$.

\item  The action of $\Gamma$ on the Banach algebra $A_{n}$ is locally analytic.   Equivalently, by Lemma \ref{LemmaBanachLocAn}, for any $\epsilon \in I_n$ there is some $k\in \mathbb{N}$ (depending on $n$ and $\epsilon$)    such that the action of $\Gamma^{(k)}\subset \Gamma$ on  $A^0_{n}/\varpi^{\epsilon}$ is trivial.

\end{enumerate}

\item[(CST2)]\textit{Bounds for the vanishing of cohomology.} 

For all $\epsilon>0$ and $k\in \mathbb{N}$ there is $n_2(\epsilon,k) \geq n_1(\epsilon)$ such that if $n\geq n_2(\epsilon,k)$ the map $1-\gamma^{p^k}: X_{n}\to X_{n}$ is invertible with inverse satisfying $||(1-\gamma^{p^k})^{-1}||\leq |\varpi|^{-\epsilon}$.

\end{itemize}

We say that the tuple $(A,\Gamma,(R_n)_{n\in \mathbb{N}})$  is a \textit{one dimensional Sen theory} if the axioms (CST1) and (CST2)   hold. 
\end{definition}

The following  lemma will be used later in the paper. 

\begin{lem}\label{lemUniformCompletion}
Let $(A,\Gamma, (R_n))$ be a one dimensional Sen theory. Then the map $\varinjlim_{n} A_n^{0} \to A_{\infty}^0$ induces an almost isomorphism after $p$-adic completions. In particular, $A$ is the uniform completion of the colimit of Banach rings $\varinjlim_{n} A_n$.
\end{lem}
\begin{proof}
Let $B=\varinjlim_{n} A_n$ and $B^0=\varinjlim_{n} A_n^{0}$.  By Nakayama's lemma it suffices to see that the map 
\begin{equation}
\label{eq01i33rod}
B^{0}/\varpi\to A^{0}/\varpi
\end{equation}
is an almost equivalence. Let us first show that it is almost surjective. Let $a\in A^{0}$ and $\epsilon\in I$, we want to show that there is $b\in B^0$ such that $b-\varpi^{\epsilon} a \in \varpi A^{0}$. By (CST1) (3) there is some $n\gg 0$ such that $b:=R_n(\varpi^{\epsilon} a)$ satisfies $|b-\varpi^{\epsilon} a|\leq |\varpi|$. By (CST1) (2) we can even pick $n$ such that $b\in A^{0}_n$ proving what we wanted. 

Now we need to show that the map \eqref{eq01i33rod} is almost injective, we will even show that it is injective. For that, it suffices to see that $A^{0}_n\cap \varpi A^{0}= \varpi A^{0}_n$. But this is clear as  $A^0_{n}= A_n\cap A^0$. 
\end{proof}

\begin{exam}\label{exampleOneDimSen}

\begin{enumerate}

\item  The most naive example of a Sen theory is  a \textit{trivial Sen theory}. Namely, we let $A$ be a Banach ring endowed with a locally analytic action of $\Gamma$.  Then, the traces $R_n$ are just the identity maps on $A$, and it is straightforward to check that all the axioms of Definition \ref{DefinitionSenAxioms} hold.  One may think that this example is strange, but it actually shows up in practice, see Example \ref{ExamBoundaryTori} (5). 

\item The first non-trivial example of a Sen theory is the classical one presented by Tate \cite{Tatepdivisible} and Sen \cite{SenCotinuous}, given by an abelian totally wildly ramified extension $K_{\infty}$ of a discretely valued field $K$ with perfect residue field with $\Gal(K_{\infty},K)$ a $1$-dimensional $p$-adic Lie group.   In this case one takes $A=\widehat{K}_{\infty}$ to be the completion of the abelian extension, and the maps $R_n$ are constructed from normalized traces, see \cite[Proposition 4.1.1]{BC1}.  The prototypical example is given by the cyclotomic extension $\mathbb{Q}_p(\zeta_{p^{\infty}})$.

\item The next example relevant for us arises from geometry. Let $C/\mathbb{Q}_p$ be a perfectoid field containing $\bb{Q}_p^{\cyc}$. Let $A=  C \langle T^{\pm 1/p^{\infty}} \rangle $  be the ring of functions of the perfectoid torus and $A_{n}=C\langle T^{\pm 1/p^n}\rangle$. Let $\Gamma=\mathbb{Z}_p(1)$ be the Tate group of $p$-th powers root of unit, and let us fix $\gamma=(\zeta_{p^{n}})_n$ a compatible sequence of $p$-th power roots of $1$. We have a natural action of $\Gamma$ on $A$ given by 
\begin{equation}\label{eqkoawdwda}
\gamma^a\cdot T^{1/p^k} = \zeta_{p^k}^a T^{1/p^k} 
\end{equation}
for all $a\in \mathbb{Z}_p$.  The ring $A$ has a standard  Banach $C$-basis given by the elements $T^{r}$ for $r\in \mathbb{Z}[1/p]$.  We have maps $R_n: A\to A_n$ given by taking the projection on the elements  $\{T^r\}_r$ in the basis with $|r|_p \leq |p^{-n}|$. It is straightforward to check that the axioms of Definition \ref{DefinitionSenAxioms} hold in this situation. Indeed,  in this case one even has that $||R_n||\leq 1$ as operators, i.e. that they send power bounded elements on $A$ to power-bounded elements on $A_n$. Letting $X_{n}=\ker R_n$, one sees that $X_n$ is the Banach direct sum of the elements $T^r$ with $|r|>|p^{-n}|$. By \eqref{eqkoawdwda}, for $k \leq n$ one sees that 
\[
 (1-\gamma^{p^k}) T^{r} = (1-\zeta_{p^{s-k}}^a) T^r
\]
where $r=a/p^s$ with $(a,p)=1$. This immediately shows that $(1-\gamma^{p^k})$ is invertible on $X_n$.   Since $|1-\zeta_{p^m}^b|_p= \frac{1}{(p-1)p^{m-1}}$ for $(b,m)=1$, the inverse operator $(1-\gamma^{p^k})^{-1}$ has norm $\leq | 1-\zeta_{p^{n+1-k}}|_p =\frac{1}{(p-1)p^{n-k}}$.

\item We will also need a variant of the previous example. In this case we take $A=C\langle S^{1/p^{\infty}} \rangle$ the ring of functors of a perfectoid affinoid disc of radius $1$. We take $\Gamma=\mathbb{Z}_p(1)$ as before acting in the obvious way, the rings $A_n=C\langle S^{1/p^n}\rangle$ and the maps $R_n: A\to A_n$ given by projecting to the basis elements $\{S^r\}_r$ with $|r|\leq |p^{-n}|$. 

\end{enumerate}

\end{exam}

\begin{remark}\label{remarkaoqkqsdqw}
Strictly speaking, for having a Sen theory in Example \ref{exampleOneDimSen} (2) in the sense of Definition \ref{DefinitionSenAxioms}, one would need to take the restriction to an open subgroup of the Galois group. Since Sen theory is of infinitesimal nature, i.e. it only depends on open neighbourhoods of the identity of the group,  we will  omit this detail now and in the future. 
\end{remark}

\subsubsection{Products of Sen theories}

In the previous section we defined a Sen theory for one dimensional $p$-adic Lie groups. In the following  we extend the previous axioms to the case when $\Gamma\cong \mathbb{Z}_p^d$ as $p$-adic Lie group. We fix a basis $\gamma_i\in \Gamma$ with $i=1,\ldots, d$ and let $\Gamma_i:=\gamma_i^{\mathbb{Z}_p}\subset \Gamma$ the $1$-parameter subgroup induced by $\gamma_i$. Finally,  consider a continuous action of $\Gamma$ on a sous-perfectoid ring $A$.

\begin{definition}\label{defProdSenTheory}
A \textit{$d$-dimensional Sen theory}  $(A,\Gamma, (R_n^i)_{ i=1}^d)$ for $(A,\Gamma)$ is the datum of    Sen theories $(A,\Gamma_i, (R_n^i)_{n\in \bb{N}})$  in the sense of Definition \ref{DefinitionSenAxioms} (with closed subrings $A_{n}^i$) for all  $i=1,\ldots, d$, satisfying the following properties: 

\begin{itemize}

\item[(a)]  For all $i,j\in \{1,\ldots, d\}$ the traces commute, i.e. $R_n^iR_m^j=R_m^jR_n^i$ for all $n,m\geq 0$. 

\item[(b)] The subrings $A_n^i$ are $\Gamma$-stable and the traces $R_n^i$ are $\Gamma$-equivariant. 

\end{itemize}

For all $J\subset \{1,\ldots, d\}$ and all tuple $\underline{\bbf{n}}=(n_j)_{j\in J}\subset \mathbb{N}^J$ we let $A^J_{\underline{\bbf{n}}}\subset A$ be the closed subring consisting of the image of the operator $\prod_{j\in J} R_{n_j}^j:A\to A$. 

\end{definition}

\begin{remark}\label{remUnifCompletionProduct}
 Keep the notation of Definition \ref{defProdSenTheory} and  let $J\subset \{1,\ldots, d\}$.   By an inductive argument, Lemma \ref{lemUniformCompletion} implies that $A$ is the uniform completion of the colimit of Banach rings $\varinjlim_{\underline{n}\in \mathbb{N}^{J}} A_{\underline{n}^J}$, namely, by induction one can show that the map $\varinjlim_{\underline{n}\in \mathbb{N}^J} A_{\underline{n}}^{J,0}\to A^0$ is an almost isomorphism after $p$-completion. 
\end{remark}

\begin{exam}\label{ExampleSenTheoryProduct}
\label{ExamDiscTorus}
\label{ExamBoundaryTori}

\begin{enumerate}

\item A first example of this kind of multivariable Sen theory appears in the work of  Brinon \cite{BrinonSen} where he introduces generalizations of Sen theories in the case of non-archemedian fields with imperfect residue characteristic. Examples of such fields arise from completions of fields on rigid varieties at Zariski generic points in their Berkovich or adic spectrum. Since this theory will not be necessary for us we will not attempt to precisely construct a $d$-dimensional Sen theory as in Definition \ref{defProdSenTheory} in this situation.  

\item The main example for this paper will arise from products of tori. For later convenience,  we will introduce the relevant geometric notation.

Let $C/\bb{Q}_p$ be a perfectoid field containing $\bb{Q}_p^{\cyc}$ and $C^+\subset C$ an open bounded integrally closed subring.  Let  $\bb{T}_{C}:= \Spa(C\langle T^{\pm 1} \rangle,C^+\langle T^{\pm 1} \rangle)$ be the torus of dimension $1$ with coordinate $T$ over $\Spa(C,C^+)$.  For  non-negative integers $k,n\geq 0$ we let $\bb{T}^k_{C,n}$ be the $k$-th dimensional torus with coordinates $T_{1}^{1/p^n}, \cdots, T_{k}^{1/p^n}$. For $n=\infty$ we let $\bb{T}_{C,\infty}^k=\varprojlim_{n} \bb{T}^k_{C,n}$ be the $k$-th dimensional perfectoid torus with variables $T_1^{1/p^{\infty}},\ldots, T_k^{1/p^{\infty}}$.  Let $A$ be the ring of functions of $\bb{T}^k_{C,n}$ and let $\Gamma=\mathbb{Z}_p(1)^k$ acting on $A$ in the natural way. By taking completed tensor products over $C$,  the maps $R_n^i$ of the one-dimensional torus with coordinate $T_i$ of Example \ref{exampleOneDimSen} (3) give rise a Sen theory  $(A,\Gamma, (R_n^i)_n)$ in the sense of Definition \ref{defProdSenTheory}.  Note that in this case, the map $\mathbb{T}_{C,\infty}^k\to \bb{T}_{C}^k$ is a pro-finite-\'etale $\bb{Z}_p(1)^k$-torsor.

\item A variant of the previous example that will be needed on this paper is obtained by discs instead of tori, we  introduce the relevant geometric notation. Let $\bb{D}_{C}:= \Spa(C\langle S \rangle ,  C^+\langle S \rangle)$ be the $1$-dimensional disc  over $C$ with coordinate $S$. For non negative integers $k,n\geq 0$ let $\bb{D}_{C,n}^k$ be the $k$-th dimensional polydisc with coordinates $S_1^{1/p^n}, \ldots, S_k^{1/p^{n}}$. Let $\bb{D}_{C,\infty}^k=\varprojlim_{n} \bb{D}^k_{C,n}$ be the perfectoid polydisc of dimension $k$ and coordinates $S_1^{1/p^{\infty}}, \ldots, S_k^{1/p^{\infty}}$. By letting $A$ be the ring of functions of $\bb{D}_{C,\infty}^k$ and $\Gamma=\mathbb{Z}_p(1)^k$, the one dimensional Sen theories  of Example \ref{exampleOneDimSen} (4) will combine together to a Sen theory $(A,\Gamma, (R_n^i)_{n,i})$ in the sense of Definition \ref{defProdSenTheory}.  Note that in this case the map $\mathbb{D}_{C,\infty}^k\to \bb{D}_{C}^k$ is a pro-finite-Kummer-\'etale $\bb{Z}_p(1)^k$-torsor.

\item By mixing together the previous two examples of Sen theories, we can construct a Sen theory for the product $\mathbb{S}_{C}^{(e,d-e),\infty}:=\mathbb{T}^e_{C,\infty}\times \mathbb{D}_{C,\infty}^{d-e}$ of tori and polydiscs with coordinates $T_1,\ldots, T_e, S_{e+1},\cdots, S_d$. We shall write $\mathbb{S}_{C,n}^{(e,d-e)}:=\mathbb{T}^d_{C,n}\times \mathbb{D}_{C,n}^{d-e}$ for $n\in \mathbb{N}\cup\{\infty\}$. By letting $A$ to be the ring of functions of $\mathbb{S}_{C,\infty}^{(e,d-e)}$ and $\Gamma=\mathbb{Z}_p(1)^d$, one has a natural action of $\Gamma$ on $A$ and the Sen theories of (2) and (4) combine to a Sen theory on $(A,\Gamma)$.

\item We finally discuss an example that combines the mixed case of (4), and the trivial Sen theory of Example \ref{exampleOneDimSen} (1). Let $\bb{S}^{(e,d-e)}_{C}$ be as in (4). We see $\bb{S}^{(e,d-e)}_{C}$ as a log adic space endowed with log structure arising from the normal crossing divisor $S_{e+1}\cdots S_d=0$. For $i\in \{e+1,\ldots,d\}$ let $D_i\subset\bb{S}_{C}^{(e,d-e)}$ be the divisor defined by $S_i=0$, and given $J\subset \{e+1,\ldots, d\}$ we let $D_{J}=\bigcap_{i\in J} D_i$. We endow $D_{J}$ with the log structure $\n{M}$ given by pullback along the map $D_J\to \bb{S}_{C}^{(e,d-e)}$. For $n\in \bb{N}\cup \{\infty\}$ let $D_{J,n}=\bb{S}_{C,n}^{(e,d-e)}\times_{\bb{S}_{C}^{(e,d-e)}} D_J$ where the fiber product is an fs log adic space, see \cite[Proposition 2.3.27]{DiaoLogarithmic2019}. Without loss of generality let us assume that $J=\{i+1,\ldots,d\}$ for some $i\geq e$, the underlying adic space of $D_{J,\infty}$ is the perfectoid space
\[
D_J=\Spa ( C\langle \underline{T}^{\pm \frac{1}{p^{\infty}}}, S_{e+1}^{\frac{1}{p^{\infty}}}, \ldots, S_{i}^{\frac{1}{p^{\infty}}} \rangle, C^+\langle \underline{T}^{\pm \frac{1}{p^{\infty}}}, S_{e+1}^{\frac{1}{p^{\infty}}}, \ldots, S_{i}^{\frac{1}{p^{\infty}}} \rangle)=\bb{S}_{C}^{(e,i-e)}.
\]
On the other hand, the log structure $\n{M}$ of $D_{\infty}$ is induced by the map of monoids
\[
\delta: (\bb{N}[\frac{1}{p}])^{d}\to  A_{J}:= C\langle \underline{T}^{\pm \frac{1}{p^{\infty}}}, S_{e+1}^{\frac{1}{p^{\infty}}}, \ldots, S_{i}^{\frac{1}{p^{\infty}}} \rangle
\]
mapping the $j$-th component to $S_j$ if $j\leq i$, and to $0$ if $j>i$. In particular, $(D_{J,\infty},\n{M}_{\infty}) \to (D_J,\n{M})$ is a $\Gamma=\bb{Z}_p(1)^d$-torsor in the pro-Kummer-\'etale site of $(D_J,\n{M})$.   Note, however, that only the first $e+i$ components of $\Gamma$ act non-trivially on $A_J$. The  natural normalized traces attached to the coordinates  $T_1,\ldots, T_d, S_{e+1},\ldots, S_i$ give rise to normalized traces on $A_J$ and they give rise to a $d$-dimensional Sen theory  for the pair $(A_J,\Gamma)$ where the last $d-i$ components act trivially on $A_J$.  

\end{enumerate}

\end{exam}

\subsubsection{Base change and retracts of Sen theories}

In this last subsection we prove some stability properties of Sen theories that will be necessary for our applications. First, we need a stability under retracts in a suitable sense:

\begin{lem}\label{LemmaSenTheoryProjection}
Let $\Gamma=\mathbb{Z}_p$ and consider a  one dimensional Sen theory  $(A, \Gamma, (R_n))$,  let $B\subset A$ be a closed Banach $\widehat{F}$-subalgebra  stable under $\Gamma$ and $\pi:A\to B$ a $\Gamma$-equivariant  $B$-linear projection.  Suppose that $ \pi R_n=R_n \pi$ for all $n\gg 0$ so that the maps $R_n$  restrict to  maps $R'_n : B\to B$. Then  the tuple $(B, \Gamma, (R'_n)_n)$ is a  one dimensional Sen theory. 
\end{lem}
\begin{proof}
Since $B$ is a retract of $A$ via the map $\pi$, $B$ is sous-perfectoid and we have that $B^0=B\cap A^0$. Thus, the restriction of the spectral norm of $A$ to $B$ agrees with the spectral norm of $B$. The condition $\pi R_n=R_n \pi$ implies that the traces $R_n$ leave $B$ stable, and so their images define closed Banach subalgebras $B_n=B\cap A_n$. The axiom (CST1) of Definition \ref{DefinitionSenAxioms} holds for $(B,\Gamma, R_n' )$. It is left to see that the axiom (CST2 holds), so let $\epsilon\in I$ and let  $n\gg k$ be such that $1-\gamma^{p^k}$ is invertible on $X_n$ with inverse $(1-\gamma^{p^k})^{-1}$  bounded by $\varpi^{-\epsilon}$. Let $X_n'$ be the kernel of $R_{n}'$. We have that $X_{n}'= X_n\cap B$ and so $1-\gamma^{p^k}$ is also invertible on $X_n'$ (since the projection $\pi$ must commute with $(1-\gamma^{p^k})^{-1}$ in $X_n$) and its inverse is bounded by $\varpi^{-\epsilon}$. This proves the lemma. 
\end{proof}

For extending Sen theories along \'etale maps we will need the following approximation of \'etale maps for affinoid perfectoid rings.

\begin{lem}\label{lemCompletionPerfd}
Let  $\{A_i\}_{i\in I}$ be a filtered system of sous-perfectoid $\widehat{F}$-algebras with uniform completed colimit $A=(\varinjlim_i A_i^0)^{\wedge -p }[\frac{1}{p}]$ given by a perfectoid ring. Let $A\to A'$ be an \'etale extension of $A$ that  factors as  a composite of finite \'etale extensions and rational localizations.
\begin{enumerate}

\item  There is some index $i$ and an \'etale extension $A_i\to A_i'$ that factors as a composite of finite \'etale maps and rational localizations such that $A'=A\widehat{\otimes}_{A_i} A_i'$.

\item   Let $i\in I$ and $A_i\to A_i'$ an \'etale map that factors as a composite of rational localizations and finite \'etale maps. For $j\geq i$ let $A_j':= A_j\otimes_{A_i} A_i'$. Then for any $\epsilon>0$ there is some $j\geq i$ such that the map 
\[
A_j'^{0}\widehat{\otimes}_{A_j^0} A^0 \to A'^0
\]
has  cokernel killed by $\varpi^{\epsilon}$.
\end{enumerate}

\end{lem}
\begin{proof}
Part (1) follows from \cite[Lemma 6.13 (ii)]{ScholzePerfectoid2012} for rational localizations and from  \cite[Proposition 5.4.53]{GabberRameroalmost} for finite \'etale maps. The general case is done by an inductive argument.

Part (2) is  \cite[Lemma 4.5 (iii)]{ScholzeHodgeTheory2013}. Note that, in the notation of \textit{loc. cit.} the fact that the $R_i$ are  Tate algebras of finite type is never used in the proof, only that they are uniform Tate algebras and that the \'etale extension $R_i\to S_i$ remains uniform. This holds in our situation since the algebras $A_i$ are sous-perfectoid.
\end{proof}

The following lemma extends actions of profinite groups along \'etale maps.

\begin{lem}\label{LemExtensionOfgrouos}
Let $\Pi$ be a profinite group acting on a sous-perfectoid ring $A$ over $\bb{Q}_p$. Let  $A\to A'$ be an \'etale map. Then, there is an open compact subgroup $\Pi'\subset \Pi$ and a continuous action $\rho$ of $\Pi'$ on $A$ extending that on $A$.  Given two of such extensions $\rho_1$  to $\Pi'_1$ and $\rho_1$ to $\Pi_2'$, there is a compact  open subgroup $\Pi'_3\subset \Pi'_1\cap \Pi_2'$ such that $\rho_1$ and $\rho_2$ agree on $\Pi'_3$.  Furthermore, if $\Pi$ is a $p$-adic Lie group and the action on $A$ is locally analytic, then the action of $\Pi'$ on $A'$ is also locally analytic. 
\end{lem}
\begin{proof}
  One can separate the problem in open immersions and finite \'etale maps. For open immersions $U=\Spa (A',A'^+)\subset \Spa(A',A'^{0})$ it suffices to take the stabilizer $G_U$ of $U$ in $G$. The case of finite \'etale maps is \cite[Lemma]{ScholzeLubinTate} (except that there it is stated only for strongly noetherian Tate rings but the argument works in general).  

By shrinking $\Pi$ we can assume that it acts on both $A$ and $A'$. It is left to show that the action of $\Pi$ on $A'$ is locally analytic if it is on $A$. Since the map $A\to A'$ is topologically of finite type there is a ring of definition $A_0\subset A$ and elements $f_1,\ldots, f_k\in A'^{0}$ such that the $p$-complete algebra $A_0':=A_0\langle f_1,\ldots, f_k\rangle \subset A'$ generated by the elements $f_i$  is a ring of definition of $A'$. By Lemma \ref{LemmaBanachLocAn} it suffices to show that the action of $\Pi$ on $A_0'/p$ factors through a finite quotient. Since the action on $A$ is locally analytic the same lemma implies that the action on $A_0/p$ factors through a finite quotient, by trivializing the action of the elements $f_i$ modulo $p$ (which is possible since the action of $\Pi$ in $A'_0/p$ is smooth) one gets the desired factorization.   
\end{proof}

\begin{lem}\label{SenTracesBanach}
Let $(A, \Gamma ,(R_n)_n)$ be a one dimensional Sen theory.  Let $n\gg 0$ be an integer such that there is a Banach basis $\{e_{i}\}_{i\in I}$ of $A^0$  as $A_n^0$-module (contained in $A^0$) such that

\begin{itemize}

\item  $I=\bigcup_{k\geq n} I_k$ with $I_k$ a sequential system of  finite subsets, 

\item    for $m\geq n$
 \[
 A_m^0=\bigoplus_{i\in I_m} e_{i}A_n^0,
 \]
 
 \item  for $m\geq n$ the traces $R_m: A\to A_m$ are given by the projection along the basis $\{e_i\}_{i\in I_m}$. 
 
 \end{itemize}
 
 \begin{enumerate}
\item   The solid base change along $A_n\to A$ is flat.   

\item  Let $M$ be a Banach $A_n$-module and consider the base change $M_A=M\widehat{\otimes}_{A_n} A$ as well as $M_{m}= M\widehat{\otimes}_{A_n} A_m$. Then $M_m\subset M_A$ is a closed Banach $A$-module, the base change of $R_m: A\to A_m$ gives rise to a retract $R_m: M_A\to M_m$ and for $x\in M_A$ the sequence of elements $(R_m(x))_{m}$ in $M_A$ converges to $x$. 
\end{enumerate}
\end{lem}
\begin{proof}
 
 By hypothesis $A$ is a Banach direct sum of copies of $A_m$ for $m\geq n$ giving rise to an isomorphism $A\cong A_{n}\widehat{\otimes}_{\mathbb{Q}_p} \widehat{\bigoplus}_{I} \mathbb{Q}_p$ of $A_n$-modules.   Since   Banach $\mathbb{Q}_p$-vector spaces are flat for the solid tensor product \cite[Lemma 3.21]{RRLocallyAnalytic}, by base change  the solid tensor product along $A_m\to A$ is flat for all $m\geq n$, and then so is the $p$-complete tensor product of Banach $A_m$-modules. The fact that $M_m\subset M_A$ is a closed Banach $A_m$-module and that we have the retract $M_A\to M_m$ is clear.  This proves (1). 
 
For part (2),  by assumption on $A$ we can write 
\[
M_A= \widehat{\bigoplus}_{i\in I} e_i M
\]
and the projection maps $M_A\to M_n$ are given by projecting onto a suitable subbasis of the $\{e_{i}\}_{i\in I}$. The convergence of the sequence $(R_m(x))_x$ to $x\in M_A$ is then clear.  
\end{proof}

Finally, the following lemma  extends Sen theories along \'etale maps under certain hypothesis that will occur in practice. 

\begin{lem}\label{LemmaEtaleBaseChange}
Let $\Gamma=\mathbb{Z}_p$ and consider a Sen theory  $(A, \Gamma, (R_n)_n)$ as in Definition \ref{DefinitionSenAxioms} with $A$ perfectoid. Let $A\to A'$ be an \'etale map obtained as a composite of finite \'etale maps and rational localizations. 

\begin{enumerate}

\item  There is some $k\gg 0$ and an \'etale map $A_k\to A_k'$ that factors as a composite of rational localizations and finite \'etale maps such that $A'\cong A\otimes_{A_k} A_{k}'$. Furthermore, after taking a smaller open subgroup if necessary,  the action of $\Gamma$ on $A_k$ extends to $A_k'$ and so to $A'$.  For $n\geq k$ let $R'_n:A'\to A'_n$ be the base change of the map $R_n:A\to A_n$ along $A_{k}\to A_{k}'$.

 \item  Suppose that there is  $n\gg 0$  an integer such that there is an $A_n^0$-Banach basis $\{e_{i}\}_{i\in I}$  as  in Lemma \ref{SenTracesBanach}.    Then the tuple $(A',\Gamma, (R'_n)_n) $ is a Sen theory in the sense of Definition \ref{DefinitionSenAxioms}.
 
\end{enumerate}
 
\end{lem}
\begin{proof}

Part (1) follows from Lemma \ref{lemCompletionPerfd} (1) and Lemma \ref{LemExtensionOfgrouos}. Indeed, $A$ is the uniform completion of $\varinjlim_n A_n$ by Lemma \ref{lemUniformCompletion}.

 Let us now suppose that $A_k\to A_k'$ is an \'etale map that factors as a composite of  rational localizations and finite \'etale maps, and suppose that the action of $\Gamma$ on $A_k$ extends to $A_k'$.   For $n\geq k$ let $A_n':=A_{k}'\widehat{\otimes}_{A_k} A_n$ and let  $R_n': A'\to A_n'$ be the $A_{k}'$-linear extension of the map  $R_n:A\to A_n$. We want to see that the axioms (CST1) and (CST2) of Definition \ref{DefinitionSenAxioms} hold for $(A',\Gamma, (R'_n)_n)$.

 We first show the axioms of (CST1). The axiom (CST1) (1) is obvious, the axiom (CST1) (3) follows from Lemma \ref{SenTracesBanach} when $M=A'_k$, and the axiom (4) follows from  Lemma \ref{LemExtensionOfgrouos}.  We need to show that the traces $R'_n$ satisfy the desired bound condition of (CST1) (2).  Let $\epsilon>0$, by Lemma \ref{lemCompletionPerfd} there is some $N\gg 0$ such that for $m\gg N$ the map 
 \[
 A_m'^0\widehat{\otimes}_{A^0_m} A^0\to A'^0
 \]
 has cokernel killed by $\varpi^{\epsilon}$. Note that the map must be injective since $A^0$  has a Banach basis as $A_m^0$-module. In particular,  $ A_m'^0\widehat{\otimes}_{A^0_m} A^0$ is a ring of definition of $A'$.   Similarly,  by (CST1) (2) we can take $N\gg 0$  such that for all $m\gg N$ the trace $R_m$ induces a map  
 \[
 R_m: A^0 \to \varpi^{-\epsilon} A_m^0. 
 \]
 This implies that the base change $R'_m: A'\to A_m'$ induces a map
 \[
 R'_m:  A_m'^0\widehat{\otimes}_{A^0_m} A^0 \to \varpi^{-\epsilon} A_m'^0.
 \]
Hence, $R'_m$ induces a map $R_m': A'_0\to \varpi^{-2\epsilon} A_{m}'^0$. Since $\epsilon$ was arbitrary this proves the axiom.

Next we show that the axiom (CST2) holds.   Let $X_m=\ker{R_m}$ and let $X'_m$ be its base change along $A_k\to A'_k$. Let $\epsilon>0$, by taking $m$ big enough we can assume that the injective map $X^0_{m}\widehat{\otimes}_{A_m^0} A_{m}'^0 \to X_m'^0:= X_m'\cap A'^0$   has cokernel killed by $\varpi^{\epsilon}$. Therefore, $\widetilde{X}_m'^0:=X^0_{m}\widehat{\otimes}_{A_m^0} A_{m}'^0$ is open in $X'^0$ and we can prove the axiom (CST2) using the subspace  $\widetilde{X}_m'^0$ instead of $X_m'^0$. Let $\gamma\in \Gamma$ be a generator.  By (CST2) for $\epsilon>0$ and $s\in \mathbb{N}$ there is $n_{2}(\epsilon, s)\geq n_1(\epsilon)$ such that if $n\geq n_2(\epsilon,s)$ the map $1-\gamma^s: X_n\to X_n$  has an inverse $(1-\gamma^{p^s})^{-1}$ of norm $\leq |\varpi|^{-\epsilon}$. This means that the map $1-\gamma^{p^s}: X_n^0\to X_n^0$ has an inverse up to a $\varpi^{\epsilon}$-factor
\[
(1-\gamma^{p^s})^{-1}: X^0_n\to  \varpi^{-\epsilon } X_n^0. 
\] 
On the other hand, by taking $s\gg 0$, since the action of $\Gamma$ on $A_k$ is locally analytic, by Lemma \ref{LemmaBanachLocAn} we can assume that $\gamma^{p^s}$ acts trivially on $A_k^0/\varpi^{2\epsilon}$ and $A_{k}'^0/\varpi^{2\epsilon}$.  Consider the image $Y^0_n=(1-\gamma^{p^s}) X_n^0$, the space $Y^0_n$ is contained in $X^0_n$, contains $\varpi^{\epsilon} X^0_n$, and is stable under the multiplication by $A_k^0$. 

By assumption (2) the $A_k^0$-module $X_n^0$ has an ON basis, say $\{f_j\}_{j\in J}$. Since the elements $\varpi^{\epsilon}$ are algebraic over $\mathbb{Q}_p$, we can choose $s\gg 0$ such that $\gamma^{p^s}$ is $\varpi^{\epsilon}$-linear. Since $1-\gamma^{p^s}$ is an isomorphism on $X_n^0$, the previous shows that the image of $\{f_j\}_{j\in J}$ by $1-\gamma^{p^s}$ gives rise to an ON basis of $Y^0_{n}$ as $A_k^0$-module. Indeed, one has an $A^0_k$-linear map of torsion free $\varpi$-complete  $A_k^0$-modules
\[
\widehat{\bigoplus}_{j\in J} A_k^0 (1-\gamma^{p^s}) f_j \to Y^0_{\varpi}
\]
which is an equivalence modulo $\varpi^{2\epsilon}$, so by Nakayama's lemma it must be an isomorphism. 

We deduce that the base change $Y_n'^0=Y^0_{n}\otimes_{A_k^0} A_k'^0$ is an ON Banach $A_k'^0$-module and that it gives rise to an open ball in $Y_n'^0[\frac{1}{p}]=X_n'^0$  such that 
\[
\varpi^{\epsilon} \widetilde{X}_n'^0 \subset Y_n'^0 \subset \widetilde{X}_n'^0. 
\] 
Finally, one sees that  the induced map $1-\gamma^{p^s}: X_n'\to X_n'$ sends $\widetilde{X}_n'^0$ to $Y_n'^0$. By reducing modulo $\varpi^{2\epsilon}$ this map is $A_k^{'0}/\varpi^{2\epsilon}$-linear and is the base change along $A_k^0\to A_k'^0$ of the isomorphism 
\[
1-\gamma^{p^s}:X_n^0/\varpi^{2\epsilon} \xrightarrow{\sim} Y_n^0/\varpi^{2\epsilon}. 
\]
By Nakayama's lemma this implies that $1-\gamma^{p^s}: \widetilde{X}_n'^0\xrightarrow{\sim} Y_n'^0$ is also an isomorphism. One deduces that $1-\gamma^{p^s}: X_n'\to X_n'$ has an inverse $(1-\gamma^{p^s})^{-1}$ with norm $\leq |\varpi^{-\epsilon}|$ with respect to the Banach ball $\widetilde{X}'^0_n$. Since $\widetilde{X}'^0_n \subset X_n'^0$ has cokernel killed by $\varpi^{\epsilon}$, one deduces that the norm   of $(1-\gamma^{p^s})^{-1}$ with respect to the ball $X_n'^0$ is $\leq |\varpi^{-2\epsilon}|$. As $\epsilon$ was arbitrary this shows (CST2) and finishes the proof of the lemma. 
\end{proof}

\begin{remark}\label{remLemmaApplies}
Thanks to the Examples \ref{exampleOneDimSen} (2)-(4), Lemma \ref{LemmaEtaleBaseChange} applies for the perfectoid torus,  the perfectoid disc and  the cyclotomic tower. This will cover all the relevant applications in this paper. 
\end{remark}

We conclude with the construction of Sen theories for smooth affinoid spaces. We keep the notation of Example \ref{ExampleSenTheoryProduct}.

\begin{prop}\label{PropSenTheoryAffinoids}
Let $(X,\mathcal{M}_X)$ be a log affinoid  adic space over a perfectoid field $(C,C^+)$ containing $\mathbb{Q}_p^{\cyc}$  with $X=\Spa (A,A^+)$. Suppose that there is an \'etale chart
\[
\psi: X\to \mathbb{S}^{(e,d-e)}_C
\]
written as a composite of finite \'etale maps and rational localizations,  where the log structure of $X$ arises from the log structure of the polydisc term on the right hand side. Let $\mathbb{S}^{(e,d-e)}_{K,\infty}$ be the perfectoid product of tori and discs endowed with the action of $\Gamma= \mathbb{Z}_p(1)^d$, and let $X_{\infty}= X_{\infty} \times_{\mathbb{S}^{(e,d-e)}_{C}} \mathbb{S}^{(e,d-e)}_{C,\infty}$ be its base change to a pro-Kummer-\'etale $\Gamma$-torsor of $X$. Let $A_{\infty}$ be the ring of functions of $X_{\infty}$, and let $R_{n}'^i$ be the base change of the Sen traces of $\mathbb{S}^{(e,d-e)}_{C,\infty}$ to $A_{\infty}$. Then the tuple $(A_{\infty}, \Gamma , (R_n'^i)_{n} )$ is a $d$-dimensional Sen theory. A similar statement holds for the rings of functions $A_{\infty,J}$ obtained by boundary divisors for some subset $J\subset \{e+1,\ldots, d\}$ as in Example \ref{ExamDiscTorus} (5). 
\end{prop}
\begin{proof}
The one dimensional case follows from  \ref{LemmaEtaleBaseChange}, the higher dimensional case follows from induction on the number of coordinates. 
\end{proof}

\subsection{Relative locally analytic representations}
\label{s:SenTheoryRelLocAnRep}

Let $A$ be a sous-perfectoid ring  and let $A^0\subset A$ be the subring of power bounded elements. Let $\Pi$ be a profinite group acting continuously on $A$. 

The goal of this section is to introduce a relative analogue of ON Banach locally analytic representations over $A$ for the semilinear action of $\Pi$. The motivation is  provided by Lemma \ref{LemmaBanachLocAn},   saying that,   a continuous action of a compact $p$-adic Lie group $G$ on a Banach space  $V$ is locally analytic  if and only if there is a  $G$-stable lattice $V^0\subset V$  such that  $G$ acts through a finite quotient on $V^0/p$.  

Moreover, in order to adapt the decompletions of \cite{BC1} in Section \ref{s:SenTheoryFunctor} for relative locally analytic representations, we shall need to study infinite dimensional analogues of $1$-cocycles.  This requires a topological understanding of the automorphism group of an infinite dimensional ON Banach space: 

\begin{lem} \label{lemTopEnd}
For $I$ and $J$  index sets,  we have isomorphisms of $A$ and $A^0$-modules
  \begin{equation}
  \label{eqHomBanachs}
  \begin{aligned}
\Hom_{A}(\widehat{\bigoplus}_I A, \widehat{\bigoplus}_J A)= (\prod_{I} \widehat{\bigoplus_J} A^0)[\frac{1}{p}],  \\ 
\Hom_{A^0}(\widehat{\bigoplus}_I A^0, \widehat{\bigoplus}_J A^0)= \prod_{I} \widehat{\bigoplus_J} A^0 
  \end{aligned}
  \end{equation}
Where the $\Hom$  spaces is as continuous $A$ or $A^0$-modules.  We endow the $\Hom_{A^{0}}$ space with its natural product topology, and the $\Hom_{A}$  space with the locally convex topology making  the $\Hom_{A^{0}}$  subspace  bounded. Equivalently, we endow  the previous $\Hom_{A^0}$ and $\Hom_{A}$ spaces with the compact open topology (these two topologies agree thanks to \cite[Proposition 4.2]{ClausenScholzeCondensed2019}).  
    \proof
    We have that $\iHom_{A}(\widehat{\bigoplus}_I A, \widehat{\bigoplus}_J A)= \iHom_{A^0}(\widehat{\bigoplus}_I A^0, \widehat{\bigoplus}_J A^0) [\frac{1}{p}]$, so it suffices to prove the second equality. One has that 
    \begin{eqnarray*}
\iHom_{A^0}(\widehat{\bigoplus}_I A^0, \widehat{\bigoplus}_J A^0) & = \varprojlim_{n} \iHom_{A^0/p^n}(\bigoplus_I A^0/p^n, \bigoplus_J A^0/p^n) \\ 
& =  \varprojlim_{n} \prod_{I} \bigoplus_J A^0/p^n \\ 
& = \prod_{I} \widehat{\bigoplus}_J A^0. 
    \end{eqnarray*}
    \endproof
\end{lem}

\begin{remark}
Let $V$ be an ON Banach $A$-module with fixed ON-lattice $V^0$ over $A^0$. Classically one would endow $\End_{A}(V)$ with a Banach topology for the operator norm $||-||$ with respect to the norm  of $V$ determined by $V^0$. Concretely, the elements of norm $\leq 1$ in $\End_A(V)$  correspond to $\End_{A^0}(V^0)$. On the other hand, the topology used in Lemma \ref{lemTopEnd} is the compact open topology, which naturally agrees with the condensed topology. One has a natural continuous map 
\[
\End_A'(V)\to \End_A(V)
\]
where the LHS  is endowed with the operator norm topology while the RHS is endowed with the compact open topology. This map is an isomorphism on underlying vector spaces but it is not an isomorphism of condensed $A$-modules. However, the existence of this map shows that a sequence of endomorphisms of $V$ that converges for the operator norm topology will also converge for the compact open topology. 
\end{remark}

\begin{definition}
  Let $V$ be an ON Banach $A$-module. We define the topological group $\iAut_A(V)$ to be the pullback 
\[
\begin{tikzcd}
\Aut_A(V) \ar[r] \ar[d] &  \iEnd_{A}(V) \times \End_{A}(V) \ar[d, "\phi"]  \\ 
\{\id_{V}\} \times \{\id_{V}\} \ar[r] & \iEnd_A(V) \times \End_A(V)
\end{tikzcd}
\]
where $\phi(f,g)=(f\circ g, g\circ f)$.  Equivalently, it is the closed subspace of $\End_{A}(V) \times \End_{A}(V)$ consisting on those pairs $(f,g)$ such that $f\circ g=\id_V = g\circ f$.  If $V^0$ is an ON  $A^0$-lattice on $V$ we define $\Aut_{A^0}(V^0)$ in a similar way. 
\end{definition}

The following lemma will be used  to construct invertible elements in $\iAut_A(V)$.

\begin{lem}
\label{LemInvertibleMatrix}
 Let $V$ be an ON Banach $A$-module with fixed ON-lattice $V^0$ over $A^0$. Let $M \in \End_A(V)$ be an endomorphism whose  operator norm with respect to $V^0$ satisfies $||M||\leq  |\varpi^{\epsilon}|$ for some $\epsilon>0$.   Then $1-M \in \iAut_{A}(V)$ and its inverse is given by the convergent series $(1-M)^{-1}= \sum_{n=0}^\infty M^n$.  
\proof
Write  $V^0=\widehat{\bigoplus}_{I} A^0$  so that $\End_A(V)\cong (\prod_{I}\widehat{\bigoplus}_I A^0) [\frac{1}{p}]$.  It is enough to show  that $\sum_{n=0}^\infty M^n$ converges in $\iEnd_A(V)$, and that the sequence $((1-M)  \sum_{n=0}^{m} M^n)_{m\in \bb{N}}$ converges to $\id_V$.   By hypothesis $M'=\frac{1}{\varpi^\epsilon}M $ is an operator of $V^0$,  thus $\sum_{n=0}^\infty M^n= \sum_{n=0}^\infty \varpi^{\epsilon n} M'^{n}$ converges since $\iEnd_{A^0}(V^0)\cong \prod_{I} \widehat{\bigoplus}_I A^0$ is $p$-adically complete, and both $\varpi^{\epsilon}$ and $p$ are topologically nilpotent units of $A$.    One shows in a similar way that the sequence $((1-M)\sum_{n=0}^m M^{n})_{m\in \bb{N}} $ converges to $\id_{V}$ finishing the proof. 
\endproof
\end{lem}

Let $\Pi$ be a profinite group acting continuously on $A$.   Given an index set $I$, we denote by $e_i$ the standard basis of $\widehat{\bigoplus}_I A$, and    write $\GL_I(A)= \iAut_{A} (\widehat{\bigoplus}_I A)$  (resp. for $\GL_I(A^0)$).  The groups $\GL_I(A)$ and $\GL_I(A^0)$ have a natural continuous action of $\Pi$ on the coefficients.  Then,    ON  Banach $A$-semilinear representations of ``rank $I$'' with fixed ON basis  are equivalent to continuous $1$-cocycles of $\Pi$ on $\GL_I(A)$. 

 More precisely,  consider an $A$-semilinear representation $\rho: \Pi\times V \to V$ on an ON $A$-Banach module and  let $\Upsilon: \widehat{\bigoplus}_{i\in I} A e_i \xrightarrow{\sim } V$  be an $A$-linear isomorphism. Let $\sigma$ be the $A$-semilinear action on $\widehat{\bigoplus}_{i\in I} A e_i $ where $\Pi$ acts trivially on the  $e_i$'s.    One defines the $1$-cocycle attached to $\rho$ as
 \begin{equation}\label{eqConstCocycle}
U_{\rho}: \Pi \to \GL_I(A)
 \end{equation}
 via the formula 
 \[
 U_{\rho}(g)=   \Upsilon^{-1} \circ \rho(g)  \circ \Upsilon \circ  \sigma(g)^{-1} . 
 \]

Conversely, given $U: \Pi\to \GL_I(A)$  a continuous $1$-cocycle,  one has attached the representation $\rho_U$ on $V$ via the formula 
\[
\rho_{U}(g) =\Upsilon \circ  U(g) \circ \sigma(g) \circ \Upsilon^{-1}.
\]
One easily verifies that this produces a bijection between $A$-linear representation of $\Pi$ on $V$ (with fixed ON-basis) and continuous $1$-cocycles of $\Pi$ on $\GL_I(A)$.

We give the following definition of relative locally analytic ON Banach representation.

\begin{definition}
\label{DefRelLABanach}
An ON Banach $A$-semilinear  representation $\rho: \Pi\times V\to V$   is said relative locally analytic if  there exists a basis $\{v_i\}_{i\in I}$  generating an ON lattice $V^0$ over $A^0$ such that: 
\begin{itemize}

\item  There is  $\Pi' \subset \Pi$  an open subgroup stabilizing $V^0$ and  $\epsilon >0$ such that the action of $\Pi'$ on $\{v_i \mod \varpi^{\epsilon}\}_{i\in I}$ is  trivial.  

\end{itemize}
We say that $\{v_i\}_{i\in I}$ is a relative locally analytic basis of $V$.  
\end{definition}

\begin{remark}
The previous definition can be rewritten in terms of $1$-cocycles.  Indeed, let $\rho$ be a continuous $A$-semilinear representation of $\Pi$ on an ON $A$-Banach module $V$. Let $\{v_i\}_{i\in I}$ be a fixed ON basis generating an ON $A^0$-lattice $V^0$ stable under the action of $\Pi$. By the construction of  \eqref{eqConstCocycle}, attached to $\rho$ we have a $1$-cocycle $U_{\rho}$ in $\GL_I(A)$. Since $\Pi$ leaves $V^0$-stable it even lives in $\GL_I(A^0)$. Then $\{v_i\}_{i\in I}$ is a relative locally analytic basis if  there is some $\varpi^{\epsilon}$ such that the composite  $U_{\rho}: \GL_I(A^0)\to \GL_I(A^0/\varpi^{\epsilon})$ factors through a finite quotient $\Pi/\Pi'$, namely, that the $1$-cocycle $U_{\rho}$ is trivial on $\Pi'\subset \Pi$ modulo $\varpi^{\epsilon}$.
\end{remark}

The following lemma says that composition  with matrices in $\iAut_A(V)$ that are close enough to $\id_{V}$ preserves relative locally analytic basis.  

\begin{lem}
\label{LemConjLAbasis}
Let $V$ be a relative locally analytic representation of $\Pi$,  $\bbf{v}=\{v_i\}_{i\in I}$ a relative locally analytic basis, and  $V^0$  the $A^0$-lattice spanned by $\{v_i\}_{i\in I}$.   Let $\psi \in \iEnd_{A}(V)$ be an operator such that  $||1-\psi||\leq  |\varpi^{\epsilon}|$ for the operator norm induced by $V^0$ and  some $\epsilon >0$.  Then $\psi(\bbf{v})=\{\psi(v_i)\}_{i\in I}$ is  a relative locally analytic basis of $V$.  
\end{lem}
\begin{proof}
Let $\Pi'\subset \Pi$ be an open subgroup stabilizing $V^0$, and let $\epsilon'>0$ such that the action of $\Pi'$ on $\{v_i\mod \varpi^{\epsilon'}\}_{i\in I}$ is trivial.   Let $\epsilon''=\min\{\epsilon, \epsilon'\}$,  then  $\psi(v_i)\equiv v_i \mod \varpi^{\epsilon''}$ and $\Pi'$ acts on $\{\psi(v_i)\mod \varpi^{\epsilon''} \}_{i\in I}$ trivially.  This proves the lemma.  
\end{proof}

\begin{exam}
\label{ExampleRelativeLocAnRep}

\begin{enumerate}

\item Let $\Pi=G$ be a compact $p$-adic Lie group and $W$ be a Banach locally analytic representation over $\bb{Q}_p$.  Then,  by Lemma  \ref{LemmaBanachLocAn},   $W\widehat{\otimes}_{\bb{Q}_p} A$ is a relative locally analytic representation of $\Pi$.

\item Slightly more generally,  suppose that $\Pi$ admits  by  quotient  $\Pi\to G$ a compact $p$-adic Lie group.  Let $W$ be a Banach locally analytic representation of $G$ over  $\bb{Q}_p$.  Then $W\widehat{\otimes}_{\bb{Q}_p}A$ is a relative locally analytic representation of $\Pi$.  This is the situation we will face in the applications to rigid spaces.

\item Let $V$ be a finite free semilinear representation of $\Pi$, then $V$ is relative locally analytic. Indeed,  let $V^0\subset V$ be a $\Pi$-stable $A^0$-lattice with basis $v_1,\ldots, v_k$. Then there is some open subgroup $\Pi_0$ acting trivially on all the $v_i$ modulo $pV^0$. 
\end{enumerate}
\end{exam}

\subsection{The Sen functor}
\label{s:SenTheoryFunctor}

Let $A$ be a sous-perfectoid ring and $\Gamma\cong \bb{Z}_p^d$ a torsion free abelian $p$-adic Lie group of dimension $d$. Let  $(A, \Gamma, (R_n^i)_{i=1}^d)$ be a $d$-dimensional Sen theory. In this section we define the \textit{Sen functor} for relative locally analytic ON $A$-Banach representations of $\Gamma$; this is nothing but the decompletion by taking (derived) locally analytic vectors.

\begin{definition}[The Sen functor]
\label{DefinitionSenFunctor}
Let $V$ be a  relative locally analytic ON $A$-Banach representation of $\Gamma$.  We define the \textit{Sen module of $V$} to be  the subspace of locally analytic vectors for the action of $\Gamma$:
\[
S(V):= V^{\Gamma-la}.
\]
We also define the \textit{derived Sen module of $V$} to be  the solid complex of derived $\Gamma$-locally analytic vectors. 
\[
RS(V):=V^{R\Gamma-la}.
\]
\end{definition}
\begin{remark}\label{RemarkSenFunctorSolid}
In Definition \ref{DefinitionSenFunctor} (2) we see $V$ as a solid $\bb{Q}_{p}$-vector space.  Note also that  the  derived Sen module can be constructed for any complex of solid $A$-semilinear $\Gamma$-representation. 
\end{remark}

Before stating the main theorem, let us study the behaviour of passing through analytic vectors for the action of $\Gamma$ on $A$.  Write $C^{la}(\Gamma, \bb{Q}_p)=\varinjlim_{h\to \infty} C^{h}(\Gamma, \bb{Q}_p)$ as a colimit of $h$-analytic functions.  The following lemma says that the $h$-analytic vectors of $A$ are sandwiched with the images of the Sen traces.  

\begin{lem}\label{LemLocAnVectorsRingSenTheory}
Let $(A, \Gamma, R_n)$ be a one dimensional Sen theory and let $\gamma\in \Gamma$ be a generator. The following holds:

\begin{enumerate}
\item  Given $h>0$ there is $n(h) $ such that for $k\geq n (h)$ one has $X_k^{Rh-an}=0$. In particular $A^{Rh_n-an}= A_{n(h)}^{Rh-an}$, and the natural map $A^{Rh-an}\to A$  induces an  inclusion of subalgebras of $A$
\[
A^{h-an}\subset A_{n(h)}.
\]

\item Conversely, given $n\in \bb{N}$ there is $h(n)>0$ such that $A_n$ is $h(n)$-analytic. In particular we have an inclusion of subalgebras of $A$
\[
A_n\subset A^{h(n)-an}.
\]

\item Taking colimits as $n\to \infty$ we get that 
\[
A^{R\Gamma-la}=A^{\Gamma-la}=\varinjlim_{n} A_n.
\]

\end{enumerate}
\end{lem}
\begin{proof}
The second statement is just (CST4) (4). For the first statement, we need to see that given $h>0$ there is $n(h)$ such that for $k\geq n(h)$ the space $X_k=\ker(R_n:A\to A_n)$ has no $h$-analytic vectors.

By Lemma \ref{LemmaBanachLocAn} given  $h>0$ and $\epsilon>0$  there is some $s\in \bb{N}$ such that the action of $1-\gamma^{p^s}$ satisfies $||(1-\gamma^{p^s})||\leq |\varpi^{2\epsilon}|$ as operator on $C^{h}(\Gamma, \bb{Q}_p)$. By (CST2) given $\epsilon>0$ and $s\in \bb{N}$ there is $n(\epsilon,s)$ such that for $k\geq n(\epsilon,s)$ the operator $1-\gamma^{p^s}$ is invertible on $X_k$ with inverse satisfying $|| (1-\gamma^{p^s})^{-1}|| \leq |\varpi|^{-\epsilon}$.   Then, the operator $1-\gamma^{p^s}$ is invertible on $X_k\widehat{\otimes}_{\bb{Q}_p} C^h(\Gamma, \bb{Q}_p)$ with inverse given by the power series
\[
(1-\gamma^{p^s})^{-1}=-\sum_{i=0}^\infty \gamma^{-p^s}(\gamma^{-p^s}-1)^{-(i+1)}\otimes (\gamma^{p^s}-1)^{i}.  
\]
This implies that 
\[
X^{Rh-an}_k =R\Gamma(\Gamma, X_k\widehat{\otimes}_{\bb{Q}_p} C^h(\Gamma, \bb{Q}_p))= 0.
\]
Taking $n(h)=n(s, \epsilon)$ for a fix $\epsilon$ we get the first statement.

For the third statement, it suffices to write $A= X_k\oplus A_k$ and to know that for some $h$ (which can be large as long as $k$ is even larger) one has $X^{Rh-an}_k=0$. Hence 
\[
A^{Rh-an}=A_k^{Rh-an}. 
\] 
Taking colimits as $h,k\to \infty$ one then gets 
\[
A^{R\Gamma-la}=\varinjlim_{k,h} A_k^{Rh-an} = \varinjlim_{k} A_k
\]
as wanted. 
\end{proof}

 Let $(A, \Gamma, (R_{n}^i)_{i=1}^d)$ be a $d$-dimensional Sen theory, for a tuple $\underline{\bbf{n}}\in \bb{N}^d$ we let $A_{\underline{\bbf{n}}}= \prod_{i=1}^d R_{n_i}^i (A)$ be the image of the composition of Sen trances. For $n\in \bb{N}$ we will simply write $A_n = \prod_{i=1}^d R_{n}^i(A)$.   We  state the main theorem of this section regarding the good behaviour of the Sen functor for relative locally analytic ON $A$-Banach representations,   cf.   \cite[Proposition 3.3.1]{BC1}. 

\begin{theo}
\label{TheoSenFunctor}
  Let $V$ be a relative locally analytic ON $A$-Banach representation of $\Gamma$ with fixed ON $A^0$-lattice $V^0$ and $\bbf{v}=\{v_i\}_{i\in I}$ a relative locally analytic $A^0$-basis of $V^0$.  The following hold:
\begin{enumerate}

\item For all $\epsilon>0$ close enough to $zero$  and $n\gg 0$  the $A$-module  $V$ contains a  closed ON Banach $A_{n}$-submodule  $S_{n}(V)$ with ON $A_n^0$-lattice $S_n(V)^0= S_n(V)\cap V^0$ such that:

\begin{itemize}

\item[\normalfont(a)] $S_n(V)^0$ admits an ON basis $\bbf{v}'=\{ v_i'\}_{i\in I}$  with $v_i'\equiv v_i \mod \varpi^{\epsilon} V^0$.  In particular $A\widehat{\otimes}_{A_{n}} S_{n}(V)= V$.

\item[\normalfont(b)] The $A_{n}$-module  $S_{n}(V)$ is stable under the action of $\Gamma$.  Moreover,  $S_{n}(V)$ is a locally analytic representation.

\item[\normalfont(c)] For $m\geq n$ we have $S_{m}(A)=A_m\widehat{\otimes}_{A_n} S_{n}(A)$. 

\end{itemize}
We call $S_m(V)$ the \textit{$m$-th Sen module of $V$}, or the \textit{Sen module of $V$ over $A_m$}.

\item Given $h\gg 0$ depending on $V$, there is $n(h)\in \bb{N}$  with $n(h)\to \infty$ as $h\to \infty$ such that: 

\begin{itemize}

\item[\normalfont(a)] The Banach $A^{h-an}$-module $V^{h-an}$ admits an ON-basis.

\item[\normalfont(b)] We have a map $A^{h-an}\to A_{n(h)}$ as in Lemma \ref{LemLocAnVectorsRingSenTheory} and   
\[
S_{n(h)}(A)= A_{n(h)}\widehat{\otimes}_{A^{h-an}} V^{h-an}.
\]
In particular, for $n\gg 0$ the $A_n$-module $S_n(A)$  of (1) is unique and equal to the (base change of the) $h$-analytic vectors of $V$ for a suitable $h$. 

\end{itemize}

\item We have that 
 \[
 RS(V)=S(V)=\varinjlim_n S_{n}(V).
 \]
 In other words,  the derived Sen functor is concentrated in degree $0$ and equal to the colimit along $n$ of the decompletions $S_n(V)$.   
\end{enumerate}

\end{theo}

  The proof-strategy of Theorem \ref{TheoSenFunctor} consists of  two steps. First, we  generalize the decompletion  of \cite{BC1} from finite rank $A$ modules to ON Banach $A$-modules; this will show part (1) of the theorem.  Second, we use the decompletion of step one and  Lemma \ref{LemLocAnVectorsRingSenTheory} to deduce parts (2) and (3).

\subsubsection{Reduction to the one dimensional case}

As a first technical step, let us reduce Theorem \ref{TheoSenFunctor} to the one dimensional case by a simple inductive argument.

\begin{prop}\label{PropDevisageInductionSenTheory}
Suppose that Theorem \ref{TheoSenFunctor} holds for one dimensional Sen theories, then it holds for Sen theories of arbitrary dimensions. 
\end{prop}
\begin{proof}
Let $\Gamma\cong \bb{Z}_p^d$ and let $(A,\Gamma, (R_n^i)_{i=1}^d)$ be a $d$-dimensional Sen theory. Suppose by induction that  Theorem \ref{TheoSenFunctor} holds for $d-1$-dimensional Sen theories. Write $\Gamma= \Gamma_{1}\times \Gamma_{[2,d]}$ as a product of the first coordinate and the last $d-1$-coordinates. Let $R_n^1: A\to A^1_n$ be the Sen traces attached to the first coordinate. Then $(A, \Gamma_{1}, R_n^1)$ is a one dimensional Sen theory by definition, and by assumption Theorem \ref{TheoSenFunctor} holds for it. Thus, for all $\epsilon>0$ close enough to $0$ and $n\gg 0$ we have a Sen module $S^1_n(V)\subset V$ over $A_n^1$ such that: 

\begin{itemize}

\item[i.] There is an ON basis $\{v_i'\}_{i\in I}$ of $S^1_n(V)^0=S^1_n(V)\cap V^0$ such that $v_i'\equiv v_i \mod \varpi^{\epsilon} V^0$. 

\item[ii.] For $n\gg 0$ we can recover $S_n^1(V)$ by the formula 
\[
S_n^1(V)= A_n^1\otimes_{A^{\Gamma_1, h-an}} V^{\Gamma_1, h-an}
\]
where $W^{\Gamma_1,h-an}$ are the $h$-analytic vectors  of $W$ for the action of $\Gamma_1$ (for a suitable choice of $h$). In particular, the action of $\Gamma$ leaves $S_n^1(V)$-stable. 

\item[iii.] We have $RS^1(V)=S^1(V)=\varinjlim_{n} S_n^1(V)$, i.e. the higher locally analytic vectors for the action of $\Gamma_1$-vanish. 

\end{itemize}

Let us now fix $n_1\gg 0$  such that the previous properties hold for all $m\geq n_1$. Then, the property (i) above implies that $\{v_i'\}_{i\in I}$ is a relative locally analytic basis of $S_n^1(V)$ as $A_{n}^1$-module for the action of $\Gamma_{[2,d]}$. By Lemma \ref{LemmaSenTheoryProjection} the tuple $(A_n^1, \Gamma_{[2,d]}, (R_m^i)_{i=2}^d)$ is a $(d-1)$-dimensional Sen theory. By the inductive hypothesis, for all $m\gg n_1$ we have a Sen module $S_m(V)$ over $A_m = \prod_{i=2}^d R^i_{m} (A^1_m)$ satisfying the analogue properties (i)-(iii) above for the group $\Gamma_{[2,d]}$. 

We now deduce that the properties (i)-(iii) above hold for the whole group $\Gamma$ (after increasing $m$ if necessary). The property (i) does not depend on the group. For the property (ii), since the action of $\Gamma_{1}$ was already locally analytic on $S_m^1(V)$ and then so is its action on $S_m(V)$, then by  \cite[Proposition 3.4.2]{RRLocAn2} the action of $\Gamma$ on $S_m^1(V)$ is also locally analytic. Since $S_m^1(V)$ is a Banach space, Lemma \ref{LemmaBanachLocAn} implies that it is $h$-analytic as $\Gamma$-representation for some $h$. By increasing $m$ we can then find  $h$ such that 
\[
S_m(V)= A_m\otimes_{A^{h-an}} V^{h-an}
\]
obtaining (ii). Note that, by taking $h$ and $m$ such that $A_m\subset A^{h-an}$, and that $S_m(V)$ is $h$-analytic, one has that 
\begin{equation}\label{eqkoaemfoawq}
V^{h-an}= (A\widehat{\otimes}_{A_m} S_m(V))^{h-an}= A^{h-an}\widehat{\otimes}_{A_m} S_m(V),
\end{equation}
in particular $V^{h-an}$ is also an ON $A^{h-an}$-Banach representation of $\Gamma$. 

Let us briefly justify the projection formula of $h$-analytic vectors of \eqref{eqkoaemfoawq}: by using the argument of Lemma \ref{LemmaProjectionFormulaLocAnRep} one reduces to the projection formula $(V\widehat{\otimes}_{\bb{Q}_p} W)^{Rh-an} = V^{Rh-an}\otimes^L_{\bb{Q}_{p,\sol}} W$ for $W,V$ Banach representations of $\Gamma$ with $V$ being $h$-analytic. The derived locally analytic vectors of the tensor $V\widehat{\otimes}_{\bb{Q}_p} W$ is group cohomology of the representation $C^h(\Gamma, \bb{Q}_p)\widehat{\otimes}_{\bb{Q}_p} W \widehat{\otimes}_{\bb{Q}_p} V$ where $\Gamma$ acts on $W$ and $V$ as usual, and in the $h$-analytic functions via the left regular action. Since the action of $V$ is $h$-analytic, there is a $\Gamma$-equivariant isomorphism 
\[
\psi:C^h(\Gamma,\bb{Q}_p)\widehat{\otimes}_{\bb{Q}_p} V \cong C^h(\Gamma,\bb{Q}_p)\widehat{\otimes}_{\bb{Q}_p} V_0
\]
 where the action of $\Gamma$ in the RHS is trivial on $V_0$ and the left regular action of the $h$-analytic functions; this map is given by sending a function $f:\Gamma\to V$ to the function $\widetilde{f}(g)=g^{-1}f(g)$. Thus, we have that 
 \[
 \begin{aligned}
 (V\widehat{\otimes}_{\bb{Q}_p} W)^{Rh-an} & = R\Gamma(G, C^h(\Gamma, \bb{Q}_p)\widehat{\otimes}_{\bb{Q}_p} W \widehat{\otimes}_{\bb{Q}_p} V) \\
 & \cong  R\Gamma(G, C^h(\Gamma, \bb{Q}_p)\widehat{\otimes}_{\bb{Q}_p} W \widehat{\otimes}_{\bb{Q}_p} V_0) \\ 
 & = R\Gamma(G, C^h(\Gamma, \bb{Q}_p)\widehat{\otimes}_{\bb{Q}_p} W ) \otimes_{\bb{Q}_{p,\sol}}^{L} V\\ 
 & = W^{Rh-an} \otimes_{\bb{Q}_{p,\sol}}^{L} V
 \end{aligned} 
 \]
 where in the  second equivalence we use the isomorphism $\psi$, in the third equivalence we use the projection formula for group cohomology, and  the last equivalence is the definition of $h$-analytic vectors.

 For (iii), fix $m$ such that the Sen module $S_m(V)$ is defined. Then $V=A\widehat{\otimes}_{A_m} S_m(V)$. We pick $h$ such that $A_m$ and $S_m(V)$ are $h$-analytic. Then, by Lemma \ref{LemmaProjectionFormulaLocAnRep} one has 
\[
V^{R\Gamma-la}=(A\widehat{\otimes}_{A_m} S_m(V))^{R\Gamma-la}=A^{R\Gamma-la}\otimes^L_{A,\square} S_m(V). 
\]
For $k\gg m$ we can write $A=X_k^1\oplus A_m^1$. Lemma \ref{LemLocAnVectorsRingSenTheory} implies that $\varinjlim_{k\to \infty} (X_k^1)^{R\Gamma_1-la}=0$. And so that $A^{R\Gamma_1-la}=\varinjlim_{n} A_{n}^1$. By an inductive argument one deduces that $A^{R\Gamma-la}=\varinjlim_{n} A_n$ which yields
\[
RS(V)=V^{R\Gamma-la}=A^{R\Gamma-la}\otimes^L_{A,\sol} S_m(V) = \varinjlim_{k\geq m} A_k \otimes^L_{A,\sol} S_m(V)  = \varinjlim_{k\geq m} S_k(V)= S(V)
\]
sits in degree $0$ and is the colimit of the Sen modules $S_m(V)$ over $A_m$. This finishes the proof of condition (iii) and so the proof of the  proposition. 
\end{proof}

The following lemma will reduce the proof of Theorem \ref{TheoSenFunctor} to the construction of the Sen modules $S_n(V)$. 

\begin{lem}\label{LemReductiontoDevisage}
Keep the notation of Theorem \ref{TheoSenFunctor} and suppose that $\Gamma$ is a one dimensional $p$-adic Lie group. Suppose that we can find Sen  modules $S_m(V)$ over $A_m$ for $m\gg 0$ satisfying parts (1) and (2) of Theorem  \ref{TheoSenFunctor}. Then part (3) of the theorem holds, namely, 
\[
RS(V)=S(V)=\varinjlim_m S_m(V).
\]
\end{lem}
\begin{proof}
Let us take $h>0$ and $m\gg 0$ such that $A^{h-an}\subset A_m$, $V^{h-an}$ is an ON-Banach module over $A^{h-an}$ and $S_m(V)= A_m\widehat{\otimes}_{A^{h-an}} V^{h-an}$. Then $S_m(V)$ is a locally analytic representation of $\Gamma$ and by Lemma \ref{LemmaProjectionFormulaLocAnRep} we have that 
\[
RS(V)=V^{R\Gamma-la}= (A\widehat{\otimes}_{A_m} S_m(V))^{R\Gamma-la}=A^{R\Gamma-la}\otimes_{A_m,\sol}^L S_m(V).
\]
By Lemma \ref{LemLocAnVectorsRingSenTheory} we get that $A^{R\Gamma-la}=\varinjlim_k A_k$ which  implies what we want. \end{proof}

\subsubsection{Decompletion after \cite{BC1}}

Let $A$ be a sous-perfectoid ring. Let  $(A,\Gamma, R_n)$ be a one dimensional Sen theory with $\gamma\in \Gamma$ a generator of $\Gamma$. Given $n\gg 0$ let $c_1,c_2>0$ be such that $||R_{n}||\leq |\varpi|^{-c_1}$ as a map $A\to A_n$ and $||(1-\gamma)^{-1}|| \leq |\varpi|^{-c_2}$ on $X_n=\ker(R_n)$. Recall that we can take $c_1,c_2\to 0$ as $n\to \infty$. We let $V$ be an ON $A$-Banach representation of $\Gamma$  with $A^0$-lattice $V^0$. Suppose that $V$ is relative locally analytic and that we are given with a relative locally analytic basis $\{v_i\}_{i\in I}$ of $V^0$ as in Definition \ref{DefRelLABanach}. Let $U_{V}:\Gamma\to \GL_I(A)$ be the $1$-cocycle constructed in \eqref{eqConstCocycle} depending on the basis $\{v_i\}_{i\in I}$, we will simply write $U=U_{V}(\gamma)$.

\begin{lem}[{ \cite[L\`emme 3.2.3]{BC1} }]
\label{LemmaDevisage2}

Let $\delta,  a, b \in \bb{R}_{>0}$ be  such that $a\geq c_1+c_2+\delta$ and $b\geq \sup \{a+c_1,  2c_1+2c_2+\delta \}$.  Suppose that the following hold:
\begin{itemize}
\item[i.]  $U=1+U_{1}+U_{2}$ where $U_{1}\in \prod_{I} \widehat{\bigoplus}_I A^0_{n} $ and $U_{2}\in \prod_{I} \widehat{\bigoplus}_I A^{0}$.

\item[ii.] $U_{1}\equiv 0 \mod \varpi^a$ and $U_{2}\equiv 0 \mod \varpi^b$. 
\end{itemize} 
Then there exists $M\in \GL_I(A^{0})$ with $M\equiv 1 \mod \varpi^{b-c_2-c_3}$ such that 
\begin{itemize}
\item[i.] $M^{-1} U \gamma(M)= 1+V_{1}+ V_{2}$ with $V_{1}\in \prod_{I} \widehat{\bigoplus}_I A^{0}_{n}$ and $V_{2}\in \prod_{I}\widehat{\bigoplus}_I A^{0}$.

\item[ii.] We have $V_{1}\equiv 0 \mod \varpi^a$ and $V_{2}\equiv 0 \mod \varpi^{b+\delta}$.  

\end{itemize}
\end{lem}

\begin{proof} The proof is exactly the same as the one of  \cite[L\`emme 3.2.3]{BC1}, we give the details for completeness. 
Let $R_{n}:  A\to A_{n}$ be the Sen trace  and let $X_{n}$ be its kernel.    Since we have the decomposition $A= A_{n}\oplus X_{n}$,  the following space decomposes via $R_{n}$:
\[
(\prod_{ I} \widehat{\bigoplus}_{I}  A^0)[\frac{1}{p}] = (\prod_{ I} \widehat{\bigoplus}_{I}  A^{0}_n)[\frac{1}{p}]\oplus (\prod_{ I} \widehat{\bigoplus}_{I}  X_{n})[\frac{1}{p}]. 
\]
Then,  using the bound of (CST1) (2),  we can write $U_{2}= R_{n}(U_{2})+  (1-R_n)(U_2)$ with:
\begin{itemize}
\item  $R_{n}(U_{2}) \in \prod_{i} \widehat{\bigoplus} A^0_{n}$ and $R_{n}(U_{2})\equiv 0 \mod \varpi^{b-c_2}$.

\item  $(1-R_n)(U_2)\in \prod_{i} \widehat{\bigoplus} X^0_{n}$ and $(1-R_n)(U_2)\equiv 0 \mod \varpi^{b-c_2}$.
\end{itemize}  Using (CST2) we can write $(1-R_n)(U_2)= (1-\gamma) V$ with $V\in \prod_{i} \widehat{\bigoplus} X^0_{n}$ such that $V\equiv 0\mod \varpi^{b-c_2-c_1}$. We now modify the cocycle $U$ by the matrix $M=1+V$, this amounts to the computation 
\[
(1+V)^{-1}U \gamma(1+V). 
\]
By the conditions on $a$, $b$ and $\delta$ we have that  $V^2$,  $U_1\gamma(V)$, $U_2\gamma(V)$, $VU_1$ and $VU_2$   are equivalent  to $0 \mod   \varpi^{b+\delta}$.  Reducing modulo $\varpi^{b+\delta}$ we find that 
\[
\begin{aligned}
(1+V)^{-1}U \gamma(1+V) & \equiv (1-V) U (1+\gamma(V))  \mod \varpi^{b+\delta} \\ 
& \equiv  1-V+U_1+U_2+\gamma(V) \mod \varpi^{b+\delta}\\ 
& \equiv 1+U_1+ R_n(U_2) \mod \varpi^{b+\delta}. 
\end{aligned}
\]
This shows that $M^{-1}U\gamma(M)= 1+V_1+V_2$ with $V_1=U_1+R_n(U_2)$ and satisfying the requirements of the lemma. 
\end{proof}

\begin{cor}
\label{CoroDevisageGammas}
Let $\delta>0$ and $b\geq 2 c_2+2 c_3+\delta$.  Let $U\in \GL_I(A^{0})$ be a matrix such that  $U\equiv 1 \mod \varpi^b$.   Then there exists $M\in \GL_I(A^{0})$ with $M\equiv 1 \mod \varpi^{b-c_3-c_2}$ such that  $M^{-1} U \gamma(M)\in \GL_I(A^{0}_{n})$. 
\proof
By Lemma \ref{LemmaDevisage2} there exists $M^{(1)}\in \GL_I(A^{0})$ with $M^{(1)}\equiv 1 \mod \varpi^{b-c_2-c_3}$ such that 
\[
U^{(1)}:=M^{(1),-1} U \gamma(M^{(1)})\in \GL_I(A^0_n) \mod \varpi^{b+\delta}.
\]
Let $k\in \bb{N}_{\geq 1}$,  by induction we can find matrices $M^{(k)}\in \GL_I(A^{0})$ with $M^{(k)}\equiv 1 \mod \varpi^{b+\delta(k-1)-c_2-c_3}$ with 
\[
U^{(k)}:= M^{(k),-1} U^{(k-1)}\gamma(M^{(k)}) \in \GL_I(A_{n}^0)\mod \varpi^{b+\delta k}. 
\]
Taking $k\to \infty$,  and $M:=M^{(1)}M^{(2)}\cdots$ one sees that the matrix  $U'=M^{-1} U \gamma(M)$ takes values in $\GL_I(A^0_{n})$, and that $U'\equiv 1 \mod \varpi^{b-c_2-c_3}$. 
\endproof
\end{cor}

\begin{lem}[{ \cite[L\`emme 3.2.5]{BC1} }]
\label{LemmaAnaliticityMatrixB}
Let  $B\in \GL_I(A)$.  Suppose that we are given with $V_{1}, V_{2}\in \GL_I(A^0_{n})$ such that  $V_{1} \equiv V_{2}\equiv 1 \mod \varpi^{c_2+\epsilon}$ for some $\epsilon >0$,  and that $\gamma(B)= V_{1}BV_{2}$.  Then $B\in \GL_I(A_{n})$. 
\end{lem}
\begin{proof}
Consider $C=B-R_{n}(B)$,  then $\gamma(C)= V_{1} C V_{2}$.  We have 
\[
\gamma(C)-C=(V_{1}-1)C V_{2} + V_{1} C(V_{2}-1)-(V_{1}-1)C(V_{2}-1).
\]
Then $C\in \varpi^r  \prod_I \widehat{\bigoplus}_I X^{0}_n $ implies $\gamma(C)-C\in \varpi^{r+c_2+\epsilon} \prod_I\widehat{\bigoplus}_I X^{0}_n$.  On the other hand,  (CST2) provides an isomorphism $1-\gamma:X_n\cong X_n$ whose inverse has norm $\leq |\varpi^{-c_2}|$. We deduce that $C\in \varpi^{r+\epsilon} \prod_{I} \widehat{\bigoplus}_I X^{0}_n$.  Since $r$ was arbitrary, one gets that $C\in \varpi^{r+s\epsilon} \prod_{I} \widehat{\bigoplus}_I X^{0}_n$ for all $s\geq 1$ and so that $C=0$, or equivalently, that $B=R_n(B)$ proving what we wanted. 
\end{proof}

We finally prove parts (1) and (2) of Theorem \ref{TheoSenFunctor}.

\begin{prop}
Let $(A,\Gamma, R_n)$ be a one dimensional Sen theory. Then  parts (1) and (2) of Theorem \ref{TheoSenFunctor}  hold. 
\end{prop}

\begin{proof}
Let $V$ be an ON $A$-Banach relative locally analytic representation of $\Gamma$ with ON $A^0$-Banach lattice $V^0\subset V$. Let $\{v_i\}_{i\in I}$ be a relative locally analytic basis of $V^0$ and let $U_{V}: \Gamma\to \GL_I(A^0)$ be the $1$-cocycle attached to $V$  and the fixed basis. Given $\gamma\in \Gamma$ we will write $U_{\gamma}$ for $U_V(\gamma)$. 

Let us write $\Gamma_k:=\Gamma^{p^k}$ for a basis of open neighbourhoods of $1$ in $\Gamma$, we let $\gamma_k\in \Gamma_n$ be a generator.  Since $V$ is relative locally analytic, there is some $\epsilon>0$ and $k\in \bb{N}$ such that $U_{\gamma}\equiv 1 \mod \varpi^{\epsilon}$ for all $\gamma\in \Gamma_n$. Let us fix a $k\in \bb{N}$ satisfying this condition. 

Now, by (CST1), for $n\gg k$ the trace map $R_n:A\to A_n$ has norm $\leq |\varpi^{-c_1}|$ with $c_1\to 0$ as $n\to \infty$ . By (CST2), for  $k$ which is fixed, we can find $m_0\gg k$  and a constant $c_2>0$ such that for all $m\geq m_0$ the operator $1-\gamma_k$ is invertible in $X_m=\ker(R_m)$ with inverse of norm bounded by $|\varpi^{-c_2}|$. We can take $c_2\to 0$ as $m\to \infty$. Taking the constants $c_1$ and $c_2$ such that $\epsilon> 2c_1+2c_2$, we fix $n=m$ satisfying the previous conditions. Thus, we found some $n\gg k$ such that   $||R_n||\leq |\varpi^{-c_1}|$ and that $1-\gamma_k$ is invertible on $X_n$ with inverse satisfying $||(1-\gamma_k)^{-1}||\leq |\varpi^{-c_2}|$.  

Then $U(\gamma_k)\in \GL_I(A^0)$ is a matrix satisfying $U\equiv 1 \mod \varpi^{\epsilon}$ and by taking $\delta>0$ such that $\epsilon> 2c_1+2c_2+\delta$, Corollary \ref{LemmaAnaliticityMatrixB} says that there is a matrix $M\in \GL_I(A^0)$ with $M\equiv 1 \varpi^{\epsilon-c_1+c_2}$ such that $M^{-1}U(\gamma_k) \gamma_k(M)\in \GL_I(A^0_n)$.  Let $U':\Gamma\to \GL_I(A^0)$ be the cocycle  $U'(\gamma):= M^{-1}U(\gamma)\gamma(M)$, then $U'$ takes values in $\GL_I(A^0_n)$ for $\gamma\in \Gamma_k$ by construction. We claim that  $U'$ is already defined over $A^0_n$. Indeed, let $\gamma\in \Gamma$, then, as $\Gamma$ is commutative, we have the relation $\gamma\gamma_k=\gamma_k\gamma$ which yields the identity at the level of cocycles
\[
U'(\gamma)\gamma(U'(\gamma_k))=U'(\gamma\gamma_k)=U'(\gamma_k\gamma)= U'(\gamma_k) \gamma_k(U'(\gamma)).
\]
Therefore, 
\[
\gamma_k(U'(\gamma))=U'(\gamma_k)^{-1} U(\gamma) \gamma U'(\gamma_k).
\]
Since $U'(\gamma_k)\equiv  1\mod \varpi^{\epsilon}$ and $\epsilon>c_2$, Lemma \ref{LemmaAnaliticityMatrixB} shows that $U'(\gamma)\in \GL_I(A_n)$ and so that $U'$ is defined over $A_n$.

Translating from cocycles to representations, $U'$ gives rise to an ON $A_n$-Banach representation $S_n(V)$ of $\Gamma$ such that $A\widehat{\otimes}_{A_n}{S_n(V)}\cong V$ as $A$-semilinear $\Gamma$-representation. For $m\geq n$ we simply define $S_m(V):=A_m\widehat{\otimes}_{A_n} S_n(V)$. Since the action of $\Gamma$ is locally analytic on $A_n$, and the action of $\Gamma_k$ is trivial on the modified basis $v_i'=M v_i$ of $S_n(V)$ modulo some power of $\varpi$, Lemma \ref{LemmaBanachLocAn} implies that the action of $\Gamma$ on $S_n(V)$ is locally analytic (and so $h$-analytic for some radius $p^{-h}$).   It is straightforward to check that the representation $S_n(V)$ satisfies the condition (1) of the theorem. 

Let us now show that part (2) follows  formally from part (1). Suppose that $S_n(V)$  and $V_n$ are $h$-analytic for some $h>0$. Then by \eqref{eqkoaemfoawq} we have that 
\[
V^{h-an}=(A\widehat{\otimes}_{A_n} S_n(V))^{h-an}=  A^{h-an}\widehat{\otimes}_{A_n} S_n(V)
\] 
and so $V^{h-an}$ is an ON $A^h$-Banach representation of $\Gamma$. Moreover, taking $m\gg n$ such that $A^h\subset A_m$ we get that 
\[
A_m\widehat{\otimes}_{A^{h-an}} V^{h-an}= A_m\widehat{\otimes}_{A_n} S_n(V)= S_m(V),
\]
proving part (2). 
\end{proof}

\subsection{Group cohomology via Sen theory}
\label{s:SenTheoryCohomology}

We finish the general discussion of Sen theory with some applications to the computation of group cohomology. Let $A$ be a sous-perfectoid ring, $\Gamma\cong \bb{Z}_p^d$ a torsion free abelian compact $p$-adic Lie group of dimension $d$ and $(A,\Gamma, (R_n^i)_{i=1}^d)$ a $d$-dimensional Sen theory on $A$.  We make the following additional assumption that will always hold in practice:

\begin{hypothesis}\label{HypoONBasisSen}
 The maps $A_n\to A_m$ are finite flat and $A$ has an ON-basis over $A_n$ for all $n\gg 0$.
\end{hypothesis}

A first corollary is the computation of group cohomology in terms of Lie algebra and smooth cohomology.

\begin{cor}
\label{CoroGroupCohoAsSen}
Let $V$ be a relative locally analytic ON $A$-Banach representation of $\Gamma$.   Then 
\[
R\Gamma(\Gamma,  V)=   R\Gamma(\Gamma^{sm}, R\Gamma(\Lie \Gamma,  S(V) )). 
\]
In particular, 
\[
H^{i}(\Gamma,V)= H^{i}(\Lie \Gamma, S(V))^{\Gamma}.
\]
\end{cor}
\begin{proof}
By Theorem \ref{TheoSenFunctor} we have that $V^{R\Gamma-la}=S(V)=\varinjlim_n S_n(V)$ sits in degree $0$ and is the colimit of the decompletion by traces. Then, Theorem \ref{TheoResumeLocAn} (4) yields the desired result. 
\end{proof}

We will consider a last hypothesis that holds in the main geometric application of this paper.  In the arithmetic case  over $\bb{Q}_p$,  this hypothesis is key in the proof of the Ax-Sen-Tate theorem.

\begin{itemize}
\item[(AST)]  

\begin{enumerate}

\item  A $1$-dimensional Sen theory  $(A,\Gamma,(R_n))$ satisfies the Ax-Sen-Tate property if the following conditions hold:

\begin{itemize}

\item[i.]  $A_{n}=A^{\Gamma_n}$, where $\Gamma_n=\Gamma^{p^n}$.

\item[ii.] The traces $R_{n}:  A\to A_{n}$ are constructed from normalized traces
\begin{gather*}
R_{n}^{m}:  A_{m}\to A_{n} \\
x \mapsto \frac{1}{p^{m-n}} \sum_{g\in \Gamma_{m}/\Gamma_{n}} g(x). 
\end{gather*}
More precisely, the colimit as $m\to \infty$ of the operators $R_{n}^m$ define an operator  $R_n: \varinjlim_m A_m\to A_n$ that completes to an operator on uniform completions $R_n:A=(\varinjlim_m A_m)^{u}\to A_n$ which agrees with the Sen trace, cf. Lemma \ref{lemUniformCompletion}. 
\end{itemize}

\item Let $(A, \Gamma, (R_n^i)_{i=1}^d)$ be a $d$-dimensional Sen theory. Write $\Gamma=\Gamma_1\times \Gamma_{[2,d]}\cong \bb{Z}_p\times \bb{Z}_p^{d-1}$. We say that  $(A, \Gamma, (R_n^i)_{i=1}^d)$  satisfies the Ax-Sen-Tate property if $(A, \Gamma_1, (R_{n}^1))$ satisfies the Ax-Sen-Tate property and for all $n\gg 0$ the $d-1$-dimensional Sen theories $(A_n^1, \Gamma_{[2,d]}, (R_n^i)_{i=2}^d)$ satisfies the Ax-Sen-Tate property. 

\end{enumerate}
\end{itemize}

\begin{exam}\label{ExamAXT}
\begin{enumerate}

\item Let $K$ be a $p$-adic discretely valued field over $\bb{Q}_p$ with perfect residue field. Then the cyclotomic tower  $K\to K^{\cyc}$ satisfies the axiom (AST). 

\item Let $(C,C^+)$ be a perfectoid field containing $\bb{Q}_p^{\cyc}$ and let  $X$ be an affinoid fs log smooth adic space over $(C,C^+)$  admitting a chart $\psi: X\to \bb{S}^{(e,d-e)}_{C}$ that factors as composite of finite \'etale maps and rational localizations. Let $\bb{S}^{(e,d-e)}_{C,\infty}$ be the perfectoid product of tori and discs obtained by taking $p$-th power roots in the coordinates and let $X_{\infty}=\Spa(A_{\infty}, A_{\infty}^+)$ be the pullback over $X$. Let $\Gamma$ be the Galois group of $X_{\infty}\to X$. Then $A_{\infty}$ is a perfectoid ring in by Proposition \ref{PropSenTheoryAffinoids} the Sen traces $R_n^i$ of Example \ref{ExamDiscTorus} (4) extend to a Sen theory $(A_{\infty}, \Gamma, (R_n^i)_{n})$ satisfying the (AST) axiom.  The same holds for the rings of functions of the boundary divisor $D_{J,\infty}\subset X_{\infty}$ defined by a subset $J\subset \{e+1,\ldots, d\}$. 

\end{enumerate}
\end{exam}

A Sen theory satisfying the Ax-Sen-Tate axiom can be endowed with a Sen operator as follows. 

\begin{definition}
\label{Definition:SenOperator}
Let $(A,\Gamma, (R_n^{i})_{i=1}^d)$ be a $d$-dimensional Sen theory satisfying (AST).   Let $V$ be a relative locally analytic ON $A$-Banach representation  of $\Gamma$.   The Sen operator  of $V$ is the $A$-linear map
\[
\theta_V:  V \to V \otimes_{\bb{Q}_p} (\Lie \Gamma)^{\vee} 
\]
given by the $A$-extension of scalars of the connection 
\[
S(V)\to S(V) \otimes_{\bb{Q}_p} (\Lie \Gamma)^{\vee}. 
\]
Equivalently, we define the Sen operator
\[
\Sen_{V}: \Lie \Gamma \otimes V\to V
\]
to be the extension of scalars of the derivation
\[
\Lie \Gamma\otimes S(V) \to S(V).
\]

The Higgs cohomology $R\Gamma(\theta_V,V)$ of $V$ is defined as the complex
\[
0\to V\xrightarrow{\theta_V} V\otimes (\Lie \Gamma_{H'})^{\vee} \to \cdots \to V\otimes \bigwedge^d (\Lie \Gamma_{H'})^{\vee} \to 0,
\]
which is nothing but the $A$-extension of scalars of the Chevalley-Eilenberg complex of $S(V)$ computing $\Lie \Gamma$-cohomology.  We also denote $H^{i}(\theta_V,V):= H^{i}(R\Gamma(\theta_V,V))$.
\end{definition}

In the rest of the section we suppose that the $d$-dimensional Sen theory $(A,\Gamma,(R_n^i)_{i=1}^d)$ satisfies the Ax-Sen-Tate axiom.     The next result describes some cohomological properties of $A$-semilinear $\Gamma$-representations in terms of their Sen operators. 

\begin{prop}
\label{PropositionSenTheoryCohomology}
Let $V$ and $W$ be a relative locally analytic ON $A$-Banach representations  of $\Gamma$, and $W\to V$ an $A$-linear $\Gamma$-equivariant map.  Consider the exact sequence in solid $\bb{Q}_p$-vector spaces
\[
0 \to K \to W \to V \to Q \to 0
\]    Then the derived Sen modules $RS(-)$ of $K$ and $Q$ (cf. Remark \ref{RemarkSenFunctorSolid}) are in degree $0$, and we have an exact sequence of solid $\bb{Q}_p$-vector spaces
\[
0\to S(K)\to S(W)\to S(V)\to S(Q)\to 0. 
\]
In particular, we have isomorphisms of solid group cohomology
\[
\underline{H}^{i}(\Gamma,V)=\underline{H}^{i}(\theta_V, V)^{\Gamma}.
\]
\end{prop}
\begin{remark}
One of the reasons we need to see $W$ and $V$ as solid $\bb{Q}_p$-vector spaces in the previous proposition is for the Sen module of $Q$ to be well defined. Otherwise $Q$ could be a non-Hausdorff topological space and there would not be a good theory of locally analytic vectors. 
\end{remark}

We thank Lue Pan for the explanation of the following argument.    

\begin{proof}
By an inductive argument in the dimension of the Sen theory we can assume without loss of generality that $\Gamma\cong \bb{Z}_p$. Consider the map $f:W\to V$, by taking $h$-analytic vectors and extending scalars to some $A_n$ such that $A^{h-an}\subset A_n$, Theorem \ref{TheoSenFunctor} (2) gives a decompletion of $f$ to a map $f_n: S_n(W)\to S_n(V)$. By Hypothesis \ref{HypoONBasisSen},  $A$ has an ON basis over $A_n$ and so it is a flat solid $A_n$-module being isomorphic to $V\otimes_{\bb{Q}_p,\sol} A_n$ with $V$  a Banach $\bb{Q}_p$-vector space, cf. \cite[Lemma 3.21]{RRLocallyAnalytic}.  Let $S_n(K):=\ker (f_n)$ and $S_n(Q):= \coker(f_n)$, these are locally analytic representations of $\Gamma$. The flatness of $A_n\to A$ implies that 
\[
K=A\otimes_{A_n,\square}^L S_n(K) \mbox{ and } Q=A\otimes_{A_n,\square}^L S_n(Q).
\]
Passing to locally analytic vectors, we see by the projection formula (Lemma \ref{LemmaProjectionFormulaLocAnRep}) and by Lemma \ref{LemLocAnVectorsRingSenTheory} that 
\[
RS(K)= K^{R\Gamma-la} = (A)^{R\Gamma-la}\otimes_{A_n,\square}^L S_n(K)= \varinjlim_{m} A_m \otimes_{A_n}^L S_n(K)= S(K),
\]
similarly for $Q$, where in the last equivalence we used that the maps $A_n\to A_m$ are finite flat.  We deduce that  $K$ and $Q$ have no higher locally analytic vectors and  thus that we have an exact sequence of locally analytic representations
\[
0\to S(K)\to S(W)\to S(V)\to S(Q)\to 0. 
\]

Let us now show the last statement about the computation of $\Gamma$-cohomology on $V$. By Corollary \ref{CoroGroupCohoAsSen} we know that $R\Gamma(\Gamma, V)= R\Gamma(\Gamma^{\sm}, R\Gamma(\Lie \Gamma, S(V)))$ and so that 
\[
H^i(\Gamma, V)= H^i(\Lie \Gamma, S(V))^{\Gamma}.
\] 
Fix a basis of $\Lie \Gamma$  and consider the exact sequence 
\[
0\to K\to V\xrightarrow{\theta_V} V \to Q\to 0
\]
given by the Sen operator of $V$. Then the previous point shows that we have an exact sequence 
\[
0\to S(K)\to S(V)\xrightarrow{\theta_V} S(V)\to S(Q)\to 0.
\]
We deduce that 
\[
H^i(\theta_V , V)^{R\Gamma-la}= H^i(\Lie \Gamma, S(V))
\]
and by taking invariants  that 
\[
H^i(\theta_V,V)^{\Gamma}= H^i(\Lie \Gamma, S(V))^{\Gamma}= H^i(\Gamma, V)
\]
proving what we wanted. 
\end{proof}

\begin{cor}
\label{CoroProjectionFormula}
Keep the notation of Proposition \ref{PropositionSenTheoryCohomology}. Suppose that $\theta_V=0$, then there is an equivalence
\[
R\Gamma(\Gamma,V)=\bigoplus_{i=0}^{d} V^{\Gamma}\otimes \bigwedge^i \Lie \Gamma[-i].
\] 
Moreover, for $n$ big enough we have that $S_{n}(V)=V^{\Gamma^{p^n}}$ and  so
\[
V= A\widehat{\otimes}_{A_{n}} V^{\Gamma^{p^n}}.
\]
\end{cor}
\begin{proof}

 The first claim of the Corollary follows from Proposition \ref{PropositionSenTheoryCohomology} since the Higgs complex of $V$ is split.

For the second claim,  by hypothesis we have $\theta_{V}=0$, this means that the action of $\Lie \Gamma$ on $S(V)=(V)^{\Gamma-la}$ is zero, so that $S(V)$ is a smooth representation of $\Gamma$. Since $A^{\Gamma^{p^n}-an}=A^{\Gamma^{p^n}}=A_n$ by the (AST) axiom, by Theorem \ref{TheoSenFunctor} (2) there is $n>>0$ such that $S_{n}(V)=V^{\Gamma^{p^n}-an}= V^{\Gamma^{p^n}}$. 

\end{proof}

\begin{remark}
\label{RemarkProjectionFormula}
The splitting of the cohomology  of  Corollary \ref{CoroProjectionFormula} depends on the Lie algebra $\Lie \Gamma$ and so on the group $\Gamma$. In applications we will allow $\Gamma$ to vary, so to guaranty that the splitting is independent of $\Gamma$ we shall need some additional structure (eg. Hodge-Tate weights arising from an arithmetic Galois action).
\end{remark}


\section{Geometric Sen theory}
\label{ch:HTSrigid}

 Let $\bb{Q}_p^{\cyc}$ be the completed cyclotomic extension of $\bb{Q}_p$, $(C,C^+)$  a perfectoid field over $\bb{Q}_p^{\cyc}$, and let $X$ be an fs   log smooth adic space over $(C,C^+)$ with log structure given by reduced normal crossing divisors. For $?\in \{\an,\et, \ket, \proet,\proket\}$ we let $X_{?}$ denote the corresponding site over $X$ (see   \cite[Example 2.3.17 and Sections 4 and 5]{DiaoLogarithmic2019}).  We let $\widehat{\s{O}}^{(+)}_X$  denote the (bounded) complete structural  sheaf over $X_{\proket}$, and for $?\in \{\an,\et, \ket \} $ we let $\s{O}^{+}_{X}$ be the (bounded) structural sheaf on $X_{?}$.  We also let $\nu_X:X_{\proket}\to X_{\ket}$ and $\eta_X:X_{\proket}\to X_{\an}$ be the projection of sites, if $X$ is clear from the context we write $\nu$ and $\eta$ instead.   Suppose that $(C,C^+)$ is the completion of an algebraic extension of a discretely valued field with perfect residue field $(K,K^+)$, and that $X$ has a form $X'$ over $(K,K^+)$, we shall write  $\OBdr^{(+)}$ and $\OC:= \gr^0(\OBdr)$ for the log de Rham and Hodge-Tate period sheaves over $X'_{\proket}$,  see  \cite{DiaoLogarithmicHilbert2018}.

The main goal of this section is to use the abstract Sen theory formalism of Section \ref{ch:SenTheory} to study  the  Hodge-Tate cohomology of $X$,  obtaining Theorems \ref{TheoSenOperatorIntro} and \ref{TheoSenBundleofTorsorIntro} of the introduction.      In Section \ref{ss:LogKummerSequence} we prove that $R\nu_*^{1} \widehat{\s{O}}_X(1)\cong \Omega^1_{X}(\log)$, where $\nu: X_{\proket}\to X_{\ket}$ is the projection of sites following and argument of Scholze.  In Section \ref{s:HTSrigidSenBundle} we  construct the geometric  Sen operator of $X$ locally on  toric coordinates, for which we are essentially reduced to the Sen theory of a product of tori and discs as in Example \ref{ExamDiscTorus} (4).   Then, in Section \ref{ss:GeoSenOpe-Globalization}, we show that these local constructions of the Sen operator glue, following an argument suggested by Lue Pan using the isomorphism $R^1\nu_* \widehat{\s{O}}_X\cong \Omega^1_X(\log)$.  Finally  in  Section \ref{s:HTSrigidProetalecoho},  we apply the previous results  to  study the locally analytic vectors of the completed structural sheaf of pro-Kummer-\'etale torsors  of $p$-adic Lie groups.  We finish  by explaining the relation of  geometric Sen theory with the works of \cite{LiuZhuRiemannHilbert, DiaoLogarithmicHilbert2018, wang2021padic}.

All the fiber products considered in the next sections are as fs log adic spaces in the sense of \cite[Proposition   2.3.27]{DiaoLogarithmic2019}, in particular they  might differ from the fiber products of usual adic spaces (but both agree for trivial log structures).   We shall consider almost mathematics with respect to the ideal $\f{m}_C\subset C^+$ of topologically nilpotent elements of $C$.

\subsection{Log-Kummer exact sequence}
\label{ss:LogKummerSequence}

Let $X$ be an fs log smooth adic space over $(C,C^+)$ with log structure given by normal crossing divisors. Equivalently, locally in the \'etale topology, $X$ admits an \'etale map towards $\bb{S}^{(e,d-e)}_C:=\bb{T}^{e}_C\times \bb{D}^{d-e}_C$ with \[ \bb{T}^{e}_C:= \Spa(C\langle T_1^{\pm 1},\ldots, T_e^{\pm 1}\rangle ,C^+\langle T_1^{\pm 1},\ldots, T_d^{\pm 1}\rangle )\] and \[\bb{D}^{d-e}_C:=\Spa (C\langle S_{e+1},\ldots, S_d  \rangle,  C^+\langle S_{e+1},\ldots, S_d\rangle),\] such that the log structure of $X$ is the pullback of the log structure on $\bb{S}^{(e,d-e)}_C$ defined by $S_{e+1}\cdots S_d=0$. An \'etale map $\psi:X\to \bb{S}^{(e,d-e)}_C$ factoring as composite of finite \'etale maps and rational localizations is called a \textit{toric chart} of $X$ (also called a frame in the literature).

The log-Kummer exact sequence is constructed as follows.

\begin{lem}
Let $\n{M}_X$ be the Kummer-\'etale sheaf of monoids defining the log structure of $X$, and let $\n{M}_X^{gr}$ be  its group of fractions. We have a short exact sequence of pro-Kummer-\'etale sheaves over $X$
\[
0\to \widehat{\bb{Z}}_p(1) \to \varprojlim_p \n{M}_X^{gr} \to \n{M}_X^{gr} \to 0
\]
where the limit is given by multiplication by $p$. 
\end{lem}
\begin{proof}
Consider the usual Kummer short exact sequence
\[
0\to \widehat{\bb{Z}}_p(1)\to \varprojlim_p \s{O}_{X}^{\times} \to \s{O}_{X}^{\times} \to 0
\] 
where $\s{O}_{X}$ is the uncompleted structural sheaf.  To prove the lemma it suffices to see that the quotient 
\[
\underline{\n{M}}_X^{gp}= \n{M}_X^{gp}/\s{O}_{X}^{\times}
\]
is a $\bb{Z}[\frac{1}{p}]$-module. This property can be checked at the level of geometric points $\overline{x}$ of $X_{\ket}$. If $\overline{x}$ is disjoint to the divisor $D$ defining the log structure then $\underline{\n{M}}_{X,\overline{x}}^{gp}=0$ and we are done. Otherwise, we can assume that $X$ has a toric chart $X\to \bb{S}^{(e,d-e)}_C$, which implies that 
\begin{equation}\label{eqpskodjfadw}
\overline{\n{M}}_{X,\overline{x}}\cong \bb{Q}_{\geq 0}^{k}
\end{equation}
with $0\leq k\leq d-e$. Indeed, \'etale locally the log structuree of $X$ around $x$ is modelled by the monoid $\bb{N}^k$ for $0\leq k\leq d-e$, and the equation \eqref{eqpskodjfadw} follows from \cite[Construction 4.4.3]{DiaoLogarithmic2019}. Then, the group of fractions of $\overline{\n{M}}_{X,\overline{x}}$ is a  $\bb{Q}$-vector space proving the claim.
\end{proof}

\begin{prop}[{\cite[Proposition 3.23]{ScholzePerfectoidSurvey} and \cite[Proposition 2.25]{BenLineBundles}}]
\label{PropProjectionO}
Let $X$ be as before, and let $\nu:X_{\proket} \to X_{\ket}$ be the projection of sites. Then there is a natural isomorphism
\[
R^{1}\nu_* \widehat{\s{O}}_X(1) \cong \Omega_{X}^1(\log)
\]
making the following diagram commute
\[
\begin{tikzcd}
\n{M}_X^{\times} \ar[d, "\delta"] \ar{r} & R^1 \nu_* \widehat{\bb{Z}}_p(1)  \ar{d}\\
\Omega_X^1(\log)  \ar[r,"\sim"]& R^1 \nu_*\widehat{\s{O}}_X(1)
\end{tikzcd}
\]
obtained by the log-Kummer-\'etale sequence and the log-differential map $\delta$.
\end{prop}
\begin{remark}
We prefer to keep the Tate twist   even  if $C$ contains $\bb{Q}_P^{\cyc}$; in the case where $X$ has a form $X'$ over $(K,K^+)$,  the isomorphism of Proposition \ref{PropProjectionO} is Galois equivariant.   Later in \S \ref{s:HTSrigidProetalecoho} we will  show that, if $X$ admits the form $X'$,  $R\nu_* \widehat{\s{O}}_X =\bigoplus_{i=1}^d \Omega^i_X(\log)(-i) [-i]$, i.e. that the pro-Kummer-\'etale cohomology of $\widehat{\s{O}}_X$ naturally splits thanks to the Galois action.  
\end{remark}
\begin{proof}
If $X$ as trivial log structure this is \cite[Proposition 3.23]{ScholzePerfectoidSurvey}, let us show that the same argument holds in the log-smooth situation. First, the image of the  map $\delta$ generates $\Omega_X^1(\log)$ as $\s{O}_X$-module, so if such an  equivalence exists it must be unique. We can then argue locally in the \'etale topology of $X$ and assume it has toric coordinates $X\to \bb{S}^{(e,d-e)}_C$. Then, using the approximation argument of \cite[Lemma 3.24]{ScholzePerfectoidSurvey} (which only requires that the map $X\to \bb{S}^{(e,d-e)}_C$ factors as composites of finite \'etale maps and rational localizations), we can assume that $X\to \bb{S}^{(e,d-e)}_C$ arises from base change of an \'etale map $X'\to \bb{S}^{(e,d-e)}_{Y}=Y\times_{\Spa \bb{Q}_p} \bb{S}^{(e,d-e)}_{\bb{Q}_p}$ via a map $\Spa (C,C^+)\to Y$, where $Y$ is smooth of finite type over $\bb{Q}_p$. We endow $X'$ with the log structure arising from the normal crossing divisors of $\bb{S}^{(e,d-e)}_{\bb{Q}_p}$. Finally, the same argument of \textit{loc. cit.} holds by using instead the log-Faltings extension of $X'$:
\[
0\to \widehat{\s{O}}_{X'}(1)\to \gr^1 \s{O}\!\bb{B}^+_{\dR,\log,X'} \to \Omega_{X'}^1(\log)\otimes \widehat{\s{O}}_{X'}\to 0,
\]
we left the details to the reader. 
\end{proof}

\begin{cor}
\label{CoroFunctHTiso}
Let $f:Y\to X$ be a map of fs log smooth adic spaces over $(C,C^+)$ with normal crossing divisors. The following diagram is commutative 
\[
\begin{tikzcd}
f^*\Omega^{1}_X(\log) \ar[d, "f^*"] \ar[r,"\sim"]  & f^* R^1\nu_{X,*} \widehat{\s{O}}_X(1) \ar[d,"f^*"] \\ 
\Omega^{1}_Y(\log) \ar[r,"\sim"] & f^* R^1\nu_{Y,*} \widehat{\s{O}}_Y(1) .
\end{tikzcd}
\]
\end{cor}
\begin{proof}
This follows from Proposition \ref{PropProjectionO}, the commutative diagram 
\[
\begin{tikzcd}
f^*\n{M}_X^{\times} \ar[r, "\delta"] \ar[d] & f^*\Omega_X^1(\log)  \ar[d]\\
\n{M}_Y^{\times} \ar[r, "\delta"] &  \Omega_Y^1(\log),
\end{tikzcd}
\]
and the fact that the image of $\delta$ generates the sheaf of differentials as vector bundles.
\end{proof}

We give a different construction of the isomorphism $R^1\nu_*\widehat{\s{O}}_X(1)\cong \Omega^1_{X}(\log)$ suggested by the referee.  We first need a lemma:

\begin{lem}\label{LemComputationCohomologyToricCoordinates}
Let $X$ be an fs log adic space over $(C,C^+)$ and let $V\in X_{\proket}$ be an object in the pro-Kummer-\'etale site of $X$. Let $V_{\infty}\to V$ be a pro-Kummer-\'etale torsor with Galois group $\Pi$ a profinite group with trivial pro-$p$-Sylow subgroup, and with $V_{\infty}$ a log affinoid perfectoid space (see \cite[Definition 5.3.1]{DiaoLogarithmic2019}).

Let $\s{F}$ be a $p$-torsion free $p$-adically complete $\widehat{\s{O}}_X^+$-module,  pro-Kummer-\'etale sheaf on $X_{\proket}$ which is almost acyclic on log affinoid perfectoid spaces and such that $\s{F}/p$ arises from the fully-faithful map $\widetilde{X}_{\ket}\hookrightarrow \widetilde{X}_{\proket}$ of topoi (see \cite[Proposition 5.1.7]{DiaoLogarithmic2019}). Then there is an almost quasi-isomorphism
\[
R\Gamma_{\proket}(V , \s{F}) =^{ae} \s{F}(V).
\] 
 Moreover, for any $\epsilon>0$ one has 
\[
R\Gamma_{\proket}(V, \s{F}/p^{\epsilon}) =^{ae} \s{F}(V)/p^{\epsilon}. 
\]
\end{lem}
\begin{proof}
 Since $\s{F}$ is $p$-complete and $p$-torsion free we have $\s{F}=R\varprojlim_{n} \s{F}/^{\bb{L}} p^n$ as sheaves in $X_{\proket}$. Moreover, $\s{F}/^{\bb{L}} p^n= \s{F}/p^n$ sits in degree $0$. One formally has that 
 \[
 R\Gamma_{\proket}(V , \s{F}) = R\varprojlim_{n} R\Gamma_{\proket}(V , \s{F}/p^n)= R\varprojlim_{n} R\Gamma_{\ket}(V , \s{F}/p^n)
 \]
 where in the last equivalence we use \cite[Proposition 5.1.7]{DiaoLogarithmic2019}.  Since $V_{\infty}$ is log affinoid perfectoid, $\s{F}$ is almost acyclic on $V_{\infty}$ which implies that $\s{F}/p^{\epsilon}$ is almost acyclic on $V_{\infty}$ and $(\s{F}/p^{\epsilon})(V_{\infty})=^{ae} \s{F}(V_{\infty})/p^{\epsilon}$ for all $\epsilon>0$. Namely, from the short exact sequence 
 \begin{equation}\label{eqpksdpawa}
 0\to \s{F}\xrightarrow{p^{\epsilon}} \s{F} \to \s{F}/p^{\epsilon}\to 0
 \end{equation}
 we get an almost short exact sequence
 \[
 0\to \s{F}(V_{\infty}) \xrightarrow{p^{\epsilon}} \s{F}(V_{\infty}) \to (\s{F}/p^{\epsilon})(V_{\infty})\to 0. 
 \]

 Since $V_{\infty}\to V$ is a $\Pi$-torsor one deduces that 
 \[
 R\Gamma_{\proket}(V , \s{F})=^{ae}R\varprojlim_{n} R\Gamma(\Pi , \s{F}(V_{\infty})/p^n) 
 \]
 and that 
 \[
 R\Gamma_{\ket}(V, \s{F}/p^{\epsilon})=^{ae}R\Gamma(\Pi, \s{F}(V_{\infty})/p^{\epsilon}). 
 \]
 
Notice that the action of $\Pi$ on $\s{F}(V_{\infty})/p^{\epsilon}$ is smooth as $\s{F}/p^{\epsilon}$  arises from the Kummer-\'etale site.  Now, since $\Pi$ has no pro-$p$-sylow subgroup, it admits a Haar measure modulo $p^\epsilon$ for all $\epsilon>0$, and  $\Pi$-cohomology is exact on smooth representations modulo $p^{\epsilon}$. This implies that 
\[
R\Gamma_{\proket}(V, \s{F}/p^{\epsilon})=^{ae} (\s{F}/p^{\epsilon})(V).
\]
 On the other hand, for  $m\geq n$ the map $(\s{F}/p^m)(V)\to (\s{F}/p^n)(V)$ is almost surjective and one gets that 
\[
 R\Gamma_{\proket}(V , \s{F})=^{ae} R\varprojlim_{n} (\s{F}/p^n)(V)= \lim_n (\s{F}/p^n)(V)=\s{F}(V)
\]
proving the first claim.  We deduce the second claim from the almost aciclicity of $\s{F}$ in $V$ and the sort exact sequence \eqref{eqpksdpawa}. 
\end{proof}

\begin{prop}\label{PropProjectionO2}
Let $X$ be an fs log adic space over $(C,C^+)$ and consider the  short exact sequence of pro-Kummer-\'etale sheaves
\[
0\to \widehat{\s{O}}_X(1)\to \bb{B}_{\dR}^+/t^2\to \widehat{\s{O}}_X\to 0
\]
where $t$ is a generator of the kernel $\bb{B}^+_{\dR}\to \widehat{\s{O}}_X$.  Then the connecting map 
\[
d:\s{O}_X=\nu_* \widehat{\s{O}}_X \to R^1 \nu_{*} \widehat{\s{O}}_X(1)
\]
is a $C$-linear continuous derivation inducing a natural map of $\s{O}_X$-modules 
\[
\beta:\Omega^1_X\to R^1 \nu_{*} \widehat{\s{O}}_X(1).
\]
Moreover, the map $\beta$ extends uniquely to an isomorphism 
\[
\Omega^1_{X}(\log)\xrightarrow{\sim} R^1 \nu_{*} \widehat{\s{O}}_X(1).  
\]
\end{prop}
\begin{proof}
All statements are \'etale local on $X$, so we can assume without loss of generality that $X=\Spa(A,A^+)$ is affinoid and admits a chart $\psi: X\to \bb{S}^{(e,d-e)}_{C}$.  Let $T_1,\ldots, T_e$ be the torus coordinates  and $S_{e+1},\ldots, S_{d}$ the disc coordinates of $\bb{S}^{(e,d-e)}_{C}$. Let $\bb{S}^{(e,d-e)}_{C, \infty}$ be the pro-Kummer-\'etale cover over $\bb{S}^{(e,d-e)}_{C}$ obtained by  all $n$-th roots of the variables $T_i$ and $S_j$ for all $n\in \bb{N}$. By \cite[Definition 5.3.1]{DiaoLogarithmic2019},   $\bb{S}^{(e,d-e)}_{C, \infty}$ is a log affinoid perfectoid space and then so is its pulback $X_{\infty}=\Spa(A_{\infty}, A^+_{\infty})$ over $X$. Note that $X_{\infty}\to X$ is a pro-Kummer-\'etale torsor for the group $\widehat{\bb{Z}}(1)^{d}$ obtained as the Tate module of roots of unity. 

Let $\bb{S}^{(e,d-e)}_{C, p^{\infty}}$ be the pro-Kummer-\'etale $\Gamma=\bb{Z}_p(1)^d$-torsor over $\bb{S}^{(e,d-e)}_{C, p^{\infty}}$ obtained by taking only $p$-th power roots of the $T_i$ and $S_i$, let $X_{p^{\infty}}=\Spa (A_{p^{\infty}},A_{p^{\infty}}^+ )$ be its pullback to $X$. The map $X_{\infty}\to X_{p^{\infty}}$ is a $\widehat{\bb{Z}}^{(p),k}(1)$-torsor, where $\widehat{\bb{Z}}^{(p)}=\prod_{\ell\neq p} \widehat{\bb{Z}}_{\ell}$  has trivial pro-$p$-Sylow subgroups. Since $\s{O}^+_X$ is almost acyclic on log affinoid perfectoids (equiv. the topologically nilpotent elements $\s{O}^{\circ\circ}_X\subset \s{O}^+_X$ are acyclic, cf. \cite[Theorem 5.4.3]{DiaoLogarithmic2019}),    Lemma \ref{LemComputationCohomologyToricCoordinates} implies that 
\[
R\Gamma_{\proket}(X_{p^{\infty}}, \s{F}) = \s{F}(X_{p^{\infty}})
\]
for $\s{F}=\widehat{\s{O}}_X$ and $\bb{B}^+_{\dR}/t^2$. 

Then, since $X_{p^{\infty}}\to X$ is a $\Gamma$-pro-Kummer-\'etale torsor, one has an isomorphism  of cohomologies 
\[
R\Gamma_{\proket}(X, \s{F})=R\Gamma(\Gamma, \s{F}(X_{p^{\infty}}))
\]
where the RHS is the continuous cohomology of Banach representations. By Proposition \ref{PropSenTheoryAffinoids} there are Sen traces $R_{n}^i:A_{p^{\infty}}\to A^i_{p^{\infty},n}$ arising from the standard traces of the product of tori and polydiscs $\bb{S}^{(e,d-e)}_C$, and the triple $(A_{p^{\infty}}, \Gamma, (R_{n}^i))$  is a $d$-dimensional Sen theory. Note that as $A$-semilinear $\Gamma$-representation, $A_{p^{\infty}}$ and $A_{p^{\infty}}(1)$ are trivial, so they have trivial Sen operator and by Corollary \ref{CoroProjectionFormula} we have that 
\[
R\Gamma(\Gamma, A_{p^{\infty}})\cong \bigoplus_{i=0}^d A \otimes  \bigwedge^i(\Lie \Gamma)^{\vee}[-i]. 
\]
In particular, $H^1(X, \widehat{\s{O}}_X(1))$ is a free $A$-module of rank $d$. Let us now see that the connecting map 
\[
d:A\to  H^1_{\proket}(X, \widehat{\s{O}}_X(1))\cong A^d
\]
is a derivation. For this, we see $H^1_{\proket}(X, \widehat{\s{O}}_X(1))$ as isomorphic to the group cohomology $H^1(\Gamma, A_{\infty}(1))=H^1(\Gamma, A(1))\cong A^d$ which can be computed via $1$-cocycles of $\Gamma$.   The map $d$ is constructed as follows: let $\gamma_1,\ldots, \gamma_d\in \Gamma$ be the standard  basis obtained by fixing a sequence of $p$-th power roots of unity $\epsilon=(\zeta_{p^n})_n$. Let $a,b\in A\subset A_{p^{\infty}}$ and let $\widetilde{a}, \widetilde{b}\in \bb{B}^{+}_{\dR}/t^2(A_{p^{\infty}})$ be lifts of $a$ and $b$ respectively. The map $d$ sends the element $a$ to the $1$-cocycle of $\Gamma$ on $A$ given by the tuple $d(a)=((1-\gamma_1)(\widetilde{a}), \ldots, (1-\gamma_d)(\widetilde{a}))\in A^d$. One has that 
\[
(1-\gamma_i)(\widetilde{a}\widetilde{b})= (1-\gamma_i)(\widetilde{a}) \gamma_i(\widetilde{b}) + \widetilde{a}(1-\gamma_i)(\widetilde{b}) = d(a) \gamma_i(b) + a d(b)=d(a)b+a d(b)
\]
as element in $A$, showing that  $d$ satisfies the Leibniz rule. The map $d$ is clearly $C$-linear as $\Gamma$ acts trivially on $\bb{B}^+_{\dR}(C)$. This proves that $d$ is a $C$-linear derivation and so that it induces a natural map 
\[
\beta:\Omega^1_A\to H^1_{\proket}(X, \widehat{\s{O}}_X(1)).
\]
We now want to show that the map $\beta$ induces an isomorphism $\Omega^1_A(\log)\xrightarrow{\sim}  H^1_{\proket}(X, \widehat{\s{O}}_X(1))$, note that if this isomorphism exists it must be unique since $\Omega_X^1\to \Omega_X^1(\log)$ is an inclusion of $A$-vector bundles of same rank and $ H^1_{\proket}(X, \widehat{\s{O}}_X(1))$ is a vector bundle itself.

By construction, the space of log differentials of $X$ has basis given by $\frac{dT_1}{T_1},\ldots, \frac{dT_e}{T_e}, \frac{dS_{e+1}}{S_{e+1}},\ldots, \frac{dS_d}{S_d}$. Thus, in order to show that $\beta$ induces the desired isomorphism it suffices to compute the map $d$ on the coordinates $T_i$ and $S_j$. For this, take $[T^{\flat}_i]\in \bb{B}^{+}_{\dR}(A_{p^{\infty}})$ the Teichm\"uller lift of the sequence of $p$-power roots of $T_i$ (resp. $[S^{\flat}_j]$ for the $S_j$). Then
\[
d(T_i)=((1-\gamma_1)([T^{\flat}_i]), \ldots, (1-\gamma_d)([T^{\flat}_i]))
\]
which vanishes in all but the $i$-th entry which is equal to  the class of
\[
(1-[\epsilon]))[T^{\flat}_i]=  \frac{(1-[\epsilon])}{(1-[\epsilon]^{1/p}) } (1-[\epsilon]^{1/p}) [T^{\flat}_i]= t (1-[\epsilon]^{1/p}) [T^{\flat}_i]
\] 
in $A(1)\cong A$ which is nothing but $\theta((1-[\epsilon]^{1/p}) [T^{\flat}_i])= (1-\zeta_p) T_i$ (after trivializing the Tate twist with $t$), with $\theta: \bb{B}^+_{\dR}\to \widehat{\s{O}}_X$ being Fontaine's map. Similarly, one has that $d(S_j)=(1-\zeta_p) S_j$.  Therefore, if $\{v_i\}_{i=1}^d$ is the standard basis of $A^d$,  the map $\beta$ sends $\frac{dT_i}{T_i}$ to the vector $(1-\zeta_p)v_i$ and $dS_j$ to the vector $(1-\zeta_p) S_j v_j$. This implies that $\beta$ extends uniquely to an isomorphism 
\[
\Omega^1_A(\log)\xrightarrow{\sim} H^1_{\proket}(X, \widehat{\s{O}}_X(1))
\] 
proving what we wanted.
\end{proof}

\begin{remark}
One can easily show that the two isomorphisms $\Omega^1_X(\log)\cong R^1\nu_* \widehat{\s{O}}_X$ of Propositions \ref{PropProjectionO} and \ref{PropProjectionO2} are the same.  For this it suffices to see that one has a morphism of short exact sequences
\[
\begin{tikzcd}
0  \ar[r] & \bb{Z}_p(1)  \ar[r] \ar[d]& \varprojlim_p \n{M}_X  \ar[d,"{[\alpha(-)]}"] \ar[r]& \n{M}_X  \ar[r] \ar[d, "\alpha"] & 0 \\ 
0 \ar[r]& \widehat{\s{O}}_X(1)  \ar[r]& \bb{B}^{+}_{\dR}/t^2  \ar[r]& \widehat{\s{O}}_X  \ar[r] & 0 
\end{tikzcd}
\]
where the middle map sends a sequence $(a^{1/p^n})_n$ in $\n{M}_X$ to the Teichm\"uller lift $[\alpha(a^{1/p^n}))_n]$ where $\alpha: \n{M}_X\to \s{O}_X\to \widehat{\s{O}}_X$ is the monoid map of the log structure. 
\end{remark}

\subsection{The geometric Sen operator: local computation}
\label{s:HTSrigidSenBundle}

In this section we prove a local version of Theorems \ref{TheoSenOperatorIntro} and \ref{TheoSenBundleofTorsorIntro} depending on toric charts. Let $(C,C^+)$ be a perfectoid field containing $\bb{Q}_p^{\cyc}$. Let $X$ be an fs log adic space over $(C,C^+)$ with log structure arising from reduced normal crossing divisors. We define the relevant pro-Kummer-\'etale sheaves that will admit decompletions via locally analytic vectors.

\begin{definition}
\label{DefRelativeLocAnSheaf}
 A  pro-Kummer-\'etale $\widehat{\s{O}}_X$-module   $\s{F}$  over $X$  is  a  \textit{relative locally analytic ON Banach $\widehat{\s{O}}_X$-sheaf}\footnote{The abbreviation \textit{ON} comes from orthonormalizable, meaning that locally we have a Banach basis.} if there is a Kummer-\'etale  cover $\{U_i\}_{i\in I}$ of $X$ such that:
 \begin{itemize}
 
 \item[i.]   For all $i$,  the restriction $\s{F}|_{U_i}$ admits a $p$-adically complete  $\widehat{\s{O}}^+_X$-lattice  $\s{F}^0_i$.
 
 \item[ii.]  There is   $\epsilon>0$ (depending on $i$) such that $\s{F}^0_i /p^{\epsilon}\cong^{ae} \bigoplus_{J} \s{O}^+_X/p^{\epsilon} $ as almost $\s{O}^+_X/p^{\epsilon}$-modules\footnote{Recall that $\s{O}_X^+/p^{\epsilon}=\widehat{\s{O}}^+_X/p^{\epsilon}$ as $\widehat{\s{O}}^+_X$ is the $p$-completion of $\s{O}_X^+$.} for some index set $J$.   
 
 \end{itemize}
\end{definition}

\subsubsection{The set-up}

\label{ss:HTSrigidSetup}

 For the rest of the section we will assume that $X=\Spa (A,A^+)$ is affinoid and  has a toric chart $\psi: X\to \bb{S}^{(e,d-e)}_C$ where $\bb{S}^{(e,d-e)}_C=\bb{T}^{e}_{C}\times \bb{D}^{d-e}_C$, and that $\psi$ factors as the composite of finite \'etale maps and rational localizations. We highlight that, by definition of $X$, the charts $\psi$ exist \'etale locally on $X$, see \cite[Example 2.3.17]{DiaoLogarithmic2019}. We let $T_i$ for $i=1,\ldots, e$ and $S_j$ for $j=e+1,\ldots, d$ denote the coordinates of the torus and disc components  of $\bb{S}^{(e,d-e)}_C$ respectively.  Let $\bb{S}^{(e,d-e)}_{C,{\infty}}$ be the pro-Kummer-\'etale torsor over $\bb{S}^{(e,d-e)}_{C}$ obtained by taking $p$-th power roots of the coordinates $T_i$ and $S_j$, and let $X_{{\infty}}=\Spa(A_{{\infty}}, A^+_{{\infty}})$ be its pullback along $\psi$. Let $\Gamma=\bb{Z}_p(1)^d$ denote the Galois group of $X_{{\infty}}\to X$, we let $\gamma_1,\ldots, \gamma_d$ denote the coordinates of $\Gamma$ obtained after fixing  $\epsilon=(\zeta_{p^n})_n$ a compatible sequence of $p$-th power roots of unity. 
 
 \begin{remark} \label{RemarkAciclycitySheafF} 
Note that the underlying adic space of $X_{{\infty}}$ is an affinoid perfectoid space, but  as an object in $X_{\proket}$ it is not a log affinoid  perfectoid space in the sense of \cite[Definition 5.3.1]{DiaoLogarithmic2019}, namely, the sheaf of monoids of $X_{{\infty}}$ is not modelled in a  $n$-divisible  monoid for $n\neq p$. However, thanks to Lemma \ref{LemComputationCohomologyToricCoordinates} the sheaves $\s{F}$ of Definition \ref{DefRelativeLocAnSheaf}  are acyclic on $X_{{\infty}}$ (after passing to a Kummer-\'etale cover such that $\s{F}$ has a lattice $\s{F}^0$ as in  the definition). This yields a   quasi-isomorphism
 \[
 R\Gamma_{\proket}(X, \s{F})\cong  R\Gamma(\Gamma, \s{F}(X_{\infty}))
 \]
 between the pro-Kummer-\'etale cohomology of $\s{F}$ and the continuous $\Gamma$ cohomology of its $X_{{\infty}}$-points, whenever $\s{F}$ admits such lattice $\s{F}^0$.  Furthermore, the lemma also implies that $\s{F}^0(X_{\infty})/p^{\epsilon}\cong^{ae} \bigoplus_{j\in J} A^+_{\infty}/p^{\epsilon} $ as $\Gamma$-representations. Then, after  modifying the lattice $\s{F}^0(X_{\infty})$ if necessary, the $\Gamma$-representation $\s{F}(X_{\infty})$ is a relative locally analytic  ON $A_{\infty}$-Banach representation of $\Gamma$ as in Definition \ref{DefRelLABanach}.   Furthermore, the natural map 
 \[
 \s{F}(X_{\infty})\widehat{\otimes}_{A_{\infty}} \widehat{\s{O}}_{X_{\infty}} \xrightarrow{\sim} \s{F}|_{X_{\infty}}
 \]
 is an equivalence of $\Gamma$-equivariant pro-Kummer-\'etale sheaves on $X_{\infty}$. 
 \end{remark}

\subsubsection{Local version of Theorem \ref{TheoSenOperatorIntro}}

 Proposition \ref{PropSenTheoryAffinoids} implies that the ring $A_{\infty}$ has Sen traces $R_n^i: A_{\infty}\to A_n^i$ for $i=1,\ldots, d$ such that the triple $(A_{\infty},\Gamma, (R_n^i)_n)$ is a $d$-dimensional Sen theory as in Definition \ref{defProdSenTheory}. We deduce the following proposition: 
 
\begin{prop}
\label{PropMainTheo1LocalVersion}
Let $\s{F}$ be a relative locally analytic ON $\widehat{\s{O}}_X$-module over $X$ admitting a lattice $\s{F}^0$ such that $\s{F}^{0}/p^{\epsilon}\cong^{ae} \bigoplus \s{O}^+_X/p^{\epsilon}$ for some $\epsilon>0$. Then there is a $\widehat{\s{O}}_X$-linear local geometric Sen operator functorial on $\s{F}$ (but a priori depending on the chart $\psi$) 
\[
\theta_{\s{F}}:\s{F}\to \s{F} \otimes_{\s{O}_X} \Omega^1_X(\log)(-1)
\]
such that: 
\begin{enumerate}
\item $\theta_{\s{F}}$ is a Higgs field, i.e, $\theta_{\s{F}}\wedge \theta_{\s{F}} =0$. 

\item  Higgs cohomology computes pro-Kummer-\'etale cohomology, namely,  if $\nu:X_{\proket}\to X_{\ket}$ and $\eta:X_{\proket}\to X_{\an}$ are the projection of sites, then
\[
R^{i} \nu_{*} \s{F} = \nu_* H^{i}(\theta_{\s{F}}, \s{F})  \mbox{ and } R^{i} \eta_{*} \s{F} = \eta_* H^{i}(\theta_{\s{F}}, \s{F}) ,
\]
where $H^{i}(\theta_{\s{F}}, \s{F})$ is the cohomology of the Higgs complex
\[
0\to \s{F} \to \s{F}\otimes_{\s{O}_X}\Omega^1_X(\log)(-1) \to\cdots \to \s{F}\otimes_{\s{O}_X}  \Omega^d_X(\log)(-d) \to 0.
\]
\item Suppose that $\theta_{\s{F}}=0$, then there is a natural equivalence
\[
R\nu_* \s{F}\cong  \bigoplus_{i=0}^d \nu_*\s{F}\otimes_{\s{O}_X}\Omega^i_X(\log)(-i)[-i] \mbox{ and } R\eta_* \s{F}= \bigoplus_{i=0}^d \eta_*\s{F}\otimes_{\bb{Q}_p} \otimes_{\s{O}_X}\Omega^i_X(\log)(-i)[-i]
\]
depending on the toric chart $\psi$.  Moreover, $\nu_*\s{F}$ is an ON $\s{O}_X$-Banach sheaf locally finite Kummer-\'etale on $X$, and we have 
\[
\s{F}=\widehat{\s{O}}_X\widehat{\otimes}_{\s{O}_X} \nu_*\s{F}.
\]
Conversely, if $\s{G}$ is a locally ON Banach $\s{O}_X$-module in the Kummer-\'etale topology, then the geometric Sen operator of $\widehat{\s{O}}_X\widehat{\otimes}_{\s{O}_X} \s{G}$ vanishes. 
\end{enumerate}

We write $\Sen_{\s{F}}:  \Omega^1_X(\log)^{\vee}(1)\otimes_{\s{O}_X} \s{F}\to \s{F}$ for the adjoint of $\theta_{\s{F}}$.

\end{prop}
\begin{proof}
By Remark \ref{RemarkAciclycitySheafF} the sheaf $\s{F}$ is acyclic on $X_{\infty}$ and $\s{F}(X_{\infty})$ is a relative locally analytic ON $A_{\infty}$-Banach representation of $\Gamma$. Since $(A_{\infty},\Gamma)$ admits Sen traces $R_{n}^i$ and $(A_{\infty}, \Gamma, (R_n^i)_n)$ is a Sen theory by  Proposition \ref{PropSenTheoryAffinoids}, we have an $A_{\infty}$-linear $\Gamma$-equivariant  Sen operator 
\begin{equation}\label{eqSenOperatorMap1oqejwfq}
\theta_{\s{F}}:\s{F}(X_{\infty}) \to  \s{F}(X_{\infty}) \otimes (\Lie \Gamma)^{\vee}. 
\end{equation}
By Proposition \ref{CoroProjectionFormula} we can  identify  (depending on the chart $\psi$)
\[
H^1_{\proket}(X, \widehat{\s{O}}_X)\cong A\otimes_{\bb{Q}_p} (\Lie\Gamma)^{\vee}.
\]
From Propositions \ref{PropProjectionO} or \ref{PropProjectionO2} we can also naturally identify  (independently of the chart $\psi$)
\[
H^1_{\proket}(X, \widehat{\s{O}}_X(1)) = \Omega^1_A(\log).
\]
Combining these two isomorphisms the map \eqref{eqSenOperatorMap1oqejwfq} becomes
\[
\theta_{\s{F}}: \s{F}(X_{\infty})\to \s{F}(X_{\infty})\otimes_{A} \Omega^1_A(\log)(-1). 
\]
Taking the completed base change to $\widehat{\s{O}}_{X_{\infty}}$ and keeping track of the $\Gamma$-equivariance we have constructed an $\widehat{\s{O}}_X$-linear map of pro-Kummer-\'etale sheaves on $X_{\proket}$ 
\[
\theta_{\s{F}}: \s{F}\to  \s{F}(X_{\infty})\otimes_{A} \Omega^1_A(\log)(-1). 
\]
The map $\theta_{\s{F}}$ is clearly functorial on $\s{F}$ (though a priori it depends on $\psi$) being constructed by taking the derivation of $\Lie \Gamma$ on the locally analytic vectors of $\s{F}(X_{\infty})$. From the construction it is also clear that $\theta_{\s{F}}$ is a Higgs field, obtaining (1). The comparison between invariants of Higgs cohomology and pro-\'etale cohomology is a consequence of Proposition \ref{PropositionSenTheoryCohomology}, this gives (2). The cohomology computation of (3) follows from Proposition \ref{CoroProjectionFormula}, similarly for the statements about $\nu_* \s{F}$ being ON $\s{O}_X$-Banach module locally finite Kummer-\'etale and $\widehat{\s{O}}_X\widehat{\otimes}_{\s{O}_X} \nu_* \s{F}=\s{F}$. 
\end{proof}

\subsubsection{Local version of Theorem \ref{TheoSenBundleofTorsorIntro}}

Next, we construct the geometric Sen operator for a pro-Kummer-\'etale torsor with group given by a $p$-adic Lie group. We let $G$ denote a compact $p$-adic Lie group and let  $\widetilde{X}\to X$ be a pro-Kummer-\'etale torsor over $X$ with Galois group $G$.

\begin{prop}
\label{PropSenOperatorsTorsorLocal}
Let $V$ be a  locally analytic Banach representation of $G$, and let $V_{\ket}$ be the pro-Kummer-\'etale sheaf over $X$ constructed by $V$ via the $G$-torsor $\widetilde{X}\to X$. Then there is a  geometric Sen operator functorial on $V$ (but a priori depending on the chart $\psi$)
\begin{equation}\label{eqLocalMapEquivSenAjosfa}
\theta_{\widetilde{X}}:  \widehat{\s{O}}_X \otimes_{\bb{Q}_p} (\Lie G)_{\ket}^{\vee}\to  \widehat{\s{O}}_X \otimes_{\s{O}_X} \Omega^1_X(\log)(-1), 
\end{equation}
or dually a map
\[
\Sen_{\widetilde{X}}:  \Omega^1_X(\log)^{\vee}(1) \otimes_{\s{O}_X} \widehat{\s{O}}_X \to (\Lie G)_{\ket}\otimes_{\bb{Q}_p} \widehat{\s{O}}_X,
\]
where $\Lie G$ is endowed with the adjoint action of $G$, and such that we have a commutative diagram of pro-Kummer-\'etale sheaves:
\[
\begin{tikzcd}
V_{\ket}\widehat{\otimes }_{\bb{Q}_p} \widehat{\s{O}}_X \ar[r, "d_V\otimes \id_{\widehat{\s{O}}_X}"] \ar[rd,"\theta_{V}"'] & (V_{\ket}\widehat{\otimes}_{\bb{Q}_p} \widehat{\s{O}}_X )  \otimes_{\bb{Q}_p} (\Lie G)_{\ket}^{\vee}  \ar[d,"\id_{V}\otimes \theta_{\widetilde{X}}"] \\
& (V_{\ket} \widehat{\otimes}_{\bb{Q}_p} \widehat{\s{O}}_X)  \otimes_{\s{O}_X} \Omega^1_X(\log)(-1)
\end{tikzcd}
\]
such that $d_V:V\to  V\otimes_{\bb{Q}_p} (\Lie G)^{\vee}$ is induced by the derivation, and $\theta_V$ is the geometric Sen operator of $V_{\ket}\widehat{\otimes}_{\bb{Q}_p}\widehat{\s{O}}_X$ of Proposition \ref{PropMainTheo1LocalVersion}. 
\end{prop}
\begin{proof}
First, we highlight that by Lemma \ref{LemmaBanachLocAn}  $V_{\ket}\widehat{\otimes}_{\bb{Q}_p} \widehat{\s{O}}_X$ is a  relative locally analytic ON-$\widehat{\s{O}}_X$-sheaf on $X_{\proket}$, and so by Proposition \ref{PropMainTheo1LocalVersion} it admits a geometric Sen operator $\theta_V$.  We want to show that $\theta_V$ actually only depends on an universal map as in \eqref{eqLocalMapEquivSenAjosfa}. Let $W=C^h(G, \bb{Q}_p)_{\star_2}$ be the space of $h$-analytic functions of $G$ (for a fixed local chart $G_0\cong \bb{Z}_p^k$ and radius of analyticity $p^{-h}$) endowed with the right regular action of $G$. Let $\widetilde{X}_{\infty}=\widetilde{X}\times_X X_{\infty}$ be the pullback in the pro-Kummer-\'etale site of $X$ with ring of functions $\widetilde{A}_{\infty}$. The map $\widetilde{X}_{\infty}\to X$ is a $G\times \Gamma$ pro-Kummer-\'etale torsor. 

By construction, the action of the Sen operators of $\Lie \Gamma$ on the algebra $W\widehat{\otimes}_{\bb{Q}_p} \widetilde{A}_{\infty}$ is by left $G$-invariant $\widetilde{A}_{\infty}$-linear derivations. Indeed, set $\s{F}= W_{\ket}\widehat{\otimes}_{\bb{Q}_p}\widehat{\s{O}}_X$, the sheaf $\s{F}$ is an $\widehat{\s{O}}_X$-algebra, so its evaluation at $X_{\infty}$ is an $A_{\infty}$-algebra. By construction, the Sen operators are given by the derivations of $\Lie \Gamma$ on the locally analytic vectors $S(\s{F}(X_{\infty}))=\s{F}(X_{\infty})^{\Gamma-la}$. Since $\s{F}(X_{\infty})^{\Gamma-la}$ is an algebra, this action by derivations satisfies the Leibniz rule. In addition, $\Lie \Gamma$ acts $A_{\infty}^{\Gamma-la}=A_{\infty}^{\Gamma-sm}$-linearly on $\s{F}(X_{\infty})^{\Gamma-la}$. Since $\s{F}$ is nothing but the base change of $S(\s{F}(X_{\infty}))$ from $A_{\infty}^{\Gamma-sm}$ to $\widehat{\s{O}}_X$, one deduces that $\Lie \Gamma$ acts on $\s{F}$ by $\widehat{\s{O}}_X$-linear derivations. Furthermore, the action of $\Lie \Gamma$ is $G$-equivariant for the left regular action of $\s{F}$, hence the Sen operators act by left $G$-invariant derivations.  

 The previous shows that the action of $\Lie \Gamma$ on  $W\widehat{\otimes}_{\bb{Q}_p} \widetilde{A}_{\infty}$  must factor through an $\widetilde{A}_{\infty}$-linear  $G\times \Gamma$-equivariant map 
\[
\Sen_{\widetilde{X}}: \Lie \Gamma\otimes_{\bb{Q}_p} \widetilde{A}_{\infty} \to \Lie G \otimes_{\bb{Q}_p} \widetilde{A}_{\infty},
\]
where $\Lie G$ acts on $W$ via right derivations (i.e.,  left $G$-invariant derivations).  Dually, we have  a map 
\[
\theta_{\widetilde{X}}: \widetilde{A}_{\infty}  \otimes_{\bb{Q}_p}  (\Lie G)^{\vee}  \to  \widetilde{A}_{\infty} \otimes_{\bb{Q}_p}  (\Lie \Gamma)^{\vee}\cong  \widetilde{A}_{\infty}\otimes_{A}\Omega_A^1(\log) (-1).
\]
In particular, we have a $G\times \Gamma$-equivariant commutative diagram 
\[
\begin{tikzcd}
W\widehat{\otimes}_{\bb{Q}_p} \widetilde{A}_{\infty} \ar[r, "d_W\otimes \id"] \ar[rd, "\theta_V"'] &  W\widehat{\otimes}_{\bb{Q}_p} \widetilde{A}_{\infty} \otimes_{\bb{Q}_p} (\Lie G)^{\vee} \ar[d,"\id_V\otimes \theta_{\widetilde{X}}"] \\ 
& W\widehat{\otimes}_{\bb{Q}_p} \widetilde{A}_{\infty} \otimes_A \Omega^1_{A}(\log)(-1).
\end{tikzcd}
\]
Taking base change to $\widehat{\s{O}}_{\widetilde{X}_{\infty}}$ and keeping track of the $G\times \Gamma$-equivariance, we have a map of pro-Kummer-\'etale sheaves over $X_{\proket}$ as in \eqref{eqLocalMapEquivSenAjosfa} which is compatible with the geometric Sen operator of $W$. Let now $V$ be a general Banach locally analytic representation of $V$, then by using the orbit map we get a $G$-equivariant inclusion 
\[
\mathcal{O}_V:V\hookrightarrow C^{h}(G, \bb{Q}_p)_{\star_2}\widehat{\otimes}_{\bb{Q}_p} V_0
\]
for some $h>0$ where $V_0$ has the trivial action of $G$. Passing to pro-Kummer-\'etale sheaves and since the formation of $\theta_V$ is functorial on $V$, we deduce that $\theta_{\widetilde{X}}$ also computes the geometric Sen operator of $V$ finishing the proof of the proposition. 
\end{proof}

\begin{remark}
Proposition \ref{PropSenOperatorsTorsorLocal} shows that  in order to compute the geometric Sen operator of a torsor it suffices to compute the geometric Sen  operator of a faithful representation of $\Lie G$. 
\end{remark}

 A direct consequence of the previous proposition is the vanishing of the action of the geometric Sen operators at infinite level. Let $\widetilde{A}=\widehat{\s{O}}_X(\widetilde{X})$ be the alebra of completed functions of $\widetilde{X}$.

\begin{cor}\label{corVanishingSenOpaofjaowd}
Let $V=C^{la}(G, \bb{Q}_p)_{\star_1}$ be the left regular locally analytic representation of $G$ and let $V_{\ket}$ be the pro-Kummer-\'etale sheaf over $X$ obtained from $V$ via the $G$-torsor $\widetilde{X}\to X$. The following holds: 
\begin{enumerate}

\item $H^0_{\proket}(X, V_{\ket}\widehat{\otimes}_{\bb{Q}_p} \widehat{\s{O}}_X) = \widetilde{A}^{G-la}$.

\item Consider the action of $\widetilde{A}^{G-la}\otimes_{\bb{Q}_p} \Lie G$ on $\widetilde{A}^{G-la}$ by derivations. Then the restriction of this  action to the geometric Sen operators 
\[
\Sen_{\widetilde{X}}: \widetilde{A}^{G-la}\otimes_A \Omega^1_A(\log)(-1)\to \widetilde{A}^{G-la}\otimes_{\bb{Q}_p} \Lie G
\]
vanishes. 

\end{enumerate}
\end{cor}
\begin{proof}
For the first claim note that 
\[
H^0_{\proket}(X, V_{\ket}\widehat{\otimes}_{\bb{Q}_p} \widehat{\s{O}}_X) = H^0(G, C^{la}(G,\bb{Q}_p)\widehat{\otimes}_{\bb{Q}_p} \widetilde{A})= \widetilde{A}^{G-la}.
\]
For the second statement, by Proposition \ref{PropMainTheo1LocalVersion} we  have that 
\[
H^0_{\proket}(X, V_{\ket}\widehat{\otimes}_{\bb{Q}_p} \widehat{\s{O}}_X)= H^0_{\proket}(X, (V_{\ket}\widehat{\otimes}_{\bb{Q}_p} \widehat{\s{O}}_X)^{\theta_V=0})= H^0(G,( C^{la}(G,\bb{Q}_p)\widehat{\otimes}_{\bb{Q}_p} \widetilde{A})^{\theta_{V,\star_1}=0} ).
\]
In other words, if $\mathrm{Orb}: \widetilde{A}^{G-la}\to  C^{la}(G, \bb{Q}_p)\widehat{\otimes}_{\bb{Q}_p}  \widetilde{A}^{G-la}$ denotes  the orbit map $a\mapsto (g\mapsto ga)$, it factors through the subspace $ (C^{la}(G, \bb{Q}_p)\widehat{\otimes}_{\bb{Q}_p}  \widetilde{A}^{G-la} )^{\theta_{V,\star_1}=0}$ where $\theta_V$ acts by zero via $\widetilde{A}^{G-la}$-linear left derivations. 

Let $g\in G$ and $f: G\to \widetilde{A}^{G-la}$ a locally analytic function. Let us write $\star_2$ for the right regular action. We have that 
\[
(g\star_1 f)(1)= f(g^{-1}) = (g^{-1}\star_2 f)(1). 
\]
This implies that for $\f{X}\in \widetilde{A}^{G-la}\otimes_{\bb{Q}_p} \Lie G$ and $f$ as before, 
\[
(\f{X} \star_1 f) (1)= - (\f{X}\star_2 f)(1).
\]

On the other hand, a retract of the orbit map $\mathrm{Orb}$ is given by evaluation at $1$, additionally, the orbit map is $G$-equivariant for the right regular action on the locally analytic functions. Thus, for $\f{X}\in \mathrm{Im}( \Sen_{\widetilde{X}})\subset \widetilde{A}^{G-la}\otimes_{\bb{Q}_p} \Lie G$  and $a\in \widetilde{A}^{G-la}$ we get that 
\[
\f{X}\cdot a = (\f{X}\star_2 \mathrm{Orb}(a)) (1) = - (\f{X}\star_1 \mathrm{Orb}(a)) (1) = 0
\]
proving the vanishing of Sen operators on $\widetilde{A}^{G-la}$ as wanted. 
\end{proof}

\subsubsection{Vanishing of higher locally analytic vectors}

As a first application of Proposition \ref{PropSenOperatorsTorsorLocal} let us prove a local vanishing of higher locally analytic vectors at infinite level. Let $\psi:X\to \bb{S}^{(e,d-e)}_C$  be as before. Let $G$ be a compact $p$-adic Lie group and let $\widetilde{X}\to X$ be a pro-Kummer-\'etale torsor. Let us write $X=\Spa(A,A^+)$ and let $\widetilde{A}:=\widehat{\s{O}}_X(\widetilde{X})$ be the completed functions of $\widetilde{X}$.

\begin{prop}\label{PropLocalVanishingHigherLocAn}
Suppose that the geometric Sen operator
\[
\theta_{\widetilde{X}}: \widehat{\s{O}}_X \otimes_{\bb{Q}_p} (\Lie G)^{\vee}_{\ket} \to \widehat{\s{O}}_X \otimes_{\s{O}_X} \Omega^1_X(\log)(-1)
\]
is surjective.  Let $V=C^{la}(G, \bb{Q}_p)$ be the space of locally analytic functions of $G$ endowed with the left regular action, and let $V_{\ket}$ be the pro-Kummer-\'etale sheaf over $X$ attached to $V$ via the $G$ torsor $\widetilde{X}\to X$. Then the natural map 
\[
\widetilde{A}^{G-la} \xrightarrow{\sim} R\Gamma_{\proket}(X,V_{\ket}\widehat{\otimes}_{\bb{Q}_p} \widehat{\s{O}}_X)
\]
is a quasi-isomorphism. In particular, the RHS is concentrated in degree $0$. 
\end{prop}
\begin{proof}
By Corollary \ref{corVanishingSenOpaofjaowd} we know that  $H^0_{\proket}(X, V_{\ket}\widehat{\otimes}_{\bb{Q}_p} \widehat{\s{O}}_X)= \widetilde{A}^{G-la}$. Therefore, it suffices to prove the vanishing of higher cohomology groups. 

Write $V=\varinjlim_h V_h$ with $V_h=C^{h}(G,\bb{Q}_p)$ spaces of $h$-analytic representations for $h\to \infty$. Since $X$ is qcqs we have that 
\[
 R\Gamma_{\proket}(X,V_{\ket}\widehat{\otimes}_{\bb{Q}_p} \widehat{\s{O}}_X=\varinjlim_h  R\Gamma_{\proket}(X,V_{h,\ket}\widehat{\otimes}_{\bb{Q}_p} \widehat{\s{O}}_X).
\]
By Proposition \ref{PropSenOperatorsTorsorLocal}, the Sen operator of $V_h$ arises from the left derivation of $\Lie G$ and the dual of the map $\theta_{\widetilde{X}}$. On the other hand, Proposition \ref{PropMainTheo1LocalVersion} (2) we have that 
\[
H^i_{\proket}(X, V_h\widehat{\otimes}_{\bb{Q}_p} \widehat{\s{O}}_X) = H^0_{\proket}(X, H^i(\theta_{V_h},V_h\widehat{\otimes}_{\bb{Q}_p} \widehat{\s{O}}_X )).
\] 
Therefore, in order to prove the vanishing of higher cohomology groups it suffices to show that 
\begin{equation}\label{eqVanishingLocanarjojwnqoe}
H^i(\theta_V, V\widehat{\otimes}_{\bb{Q}_p} \widehat{\s{O}}_X ):=\varinjlim_{h} H^i(\theta_{V_h}, V_h\widehat{\otimes}_{\bb{Q}_p} \widehat{\s{O}}_X )=0
\end{equation}
for $i\geq 1$.  

The vanishing will follow essentially from the Poincar\'e-Birkhoff-Witt theorem by using a complementary basis of the Sen operators in $\Lie G$  to construct local coordinates of the group $G$. To make this idea precise, it is more convenient to take colimits along compact open subgroups $G_0\subset G$.  For $G_0\subset G$ normal compact open  let $X_{G_0}\to X$ be the finite Kummer-\'etale extension obtained by the quotient $\widetilde{X}/G_0$ in the pro-Kummer-\'etale site. One has that 
\[
C^{la}(G, \bb{Q}_p)= \mathrm{Ind}_{G_0}^{G}( C^{la}(G_0,\bb{Q}_p)),
\]
and by Shapiro's lemma we have that 
\[
R\Gamma_{\proket}(X, (C^{la}(G,\bb{Q}_p))_{\ket}\widehat{\s{O}}_{\bb{Q}_p} \widehat{\s{O}}_X)=R\Gamma_{\proket}(X_{G_0}, (C^{la}(G_0,\bb{Q}_p))_{\ket}\widehat{\s{O}}_{\bb{Q}_p} \widehat{\s{O}}_X).
\]
Therefore, in order to show that \eqref{eqVanishingLocanarjojwnqoe} vanishes it will suffice to work locally on $G$, and even to work as $G_0\to 1$.  We fix $\f{X}_1,\ldots, \f{X}_g$ a basis of the Lie algebra $\Lie G$ and for $h>>0$ we let $\mathring{\bb{G}}_h$ be the Stein analytic group induced by the exponential of the basis $p^{h}\f{X}_1,\ldots, p^{h}\f{X}_g$. Then $\mathring{\bb{G}}_h$ is isomorphic to the open polydisc $\mathring{\bb{D}}^g_{\bb{Q}_p}(p^{-h})$ of radius $p^{-h}$ and dimension $g$. Given $h\gg 0$ we  let $G_h=\mathring{\bb{G}}_h(\bb{Q}_p)$, for $h\gg 0$ it is an open compact subgroup of $G$. We can write the space of locally analytic functions as the colimit 
\[
C^{la}(G,\bb{Q}_p)=\varinjlim_{h\to \infty}  \mathrm{Ind}_{G_h}^G\s{O}(\mathring{\bb{G}}_h).
\]
Therefore, to prove the vanishing of \eqref{eqVanishingLocanarjojwnqoe} it suffices to show that the following colimit vanishes 
\[
\varinjlim_h H^0(X_{G_h},H^i(\theta_V, \s{O}(\mathring{\bb{G}}_h)_{\ket}\widehat{\otimes}_{\bb{Q}_p}\widehat{\s{O}}_X) )= 0,
\]
where $\s{O}(\mathring{\bb{G}}_h)_{\ket}$ is the pro-Kummer-\'etale sheaf  over $X_{G_h}$ associated to $\s{O}(\mathring{\bb{G}}_h)$. Let us take an arbitrary  map $U=\Spa(R,R^+)\to \widetilde{X}$ with $U$ log affinoid perfectoid, it suffices to see that 
\[
\varinjlim_h H^i(\theta_V, \s{O}(\mathring{\bb{G}}_h)\widehat{\otimes}_{\bb{Q}_p} R) = 0. 
\]

By hypothesis, we have a direct sum decomposition 
\[
\Lie G \otimes_{\bb{Q}_p} R \cong \Omega^1_X(\log)(1)\otimes_{A} R \oplus W.
\]
Fix $R^0$-lattices  $W^0\subset W$ and $\s{L}\subset \Omega^1_X(\log)(1)\otimes_{A} R $. For $r\gg 0$ the $R^0$-lattice $p^{r}W^0\oplus p^r\s{L}$ of $\Lie G \otimes_{\bb{Q}_p} R $ is stable under the Lie bracket and it gives rise to Stein analytic groups $\mathring{\bb{H}}_r$ over $\Spa(R,R^+)$. Given $h$ there is  some $r$ such that $p^{r}W^0\oplus p^r\s{L}$ contains the lattice $R^0  p^h \f{X}_1\oplus \cdots R^0 p^h \f{X}_g $ and conversely.  We deduce that 
\begin{equation}\label{eqpkawpmaoef}
\varinjlim_{h} \s{O}(\mathring{\bb{G}}_h)\widehat{\otimes}_{\bb{Q}_p} R =  \varinjlim_r \s{O}(\mathring{\bb{H}}_r)
\end{equation}
as modules over the $R$-linear Lie algebra $R\otimes \Lie G$.  Finally, we can write 
\[
\s{O}(\mathring{\bb{H}_r})= \s{O}(\mathring{\exp(p^r \s{L})})  \widehat{\otimes}_{R} \s{O}(\mathring{\exp(p^r W^0)}) 
\] 
 as modules over the Lie algebra $R\otimes_{\bb{Q}_p}\Lie \Gamma$ (acting by left derivations on the left term), where the exponentials of the lattices are simply given by open polydiscs over $\Spa(R,R^+)$ after picking a basis. Finally, the action of the Sen operators $\Lie \Gamma$ is induced by the left regular action,  and so it only acts in the term $\s{O}(\mathring{\exp(p^r \s{L})}) $ of the tensor product. Since $\Lie \Gamma$ is abelian,  this action is induced by the linear action of $p^r\s{L}$ on the adic space, and by picking coordinates the cohomology 
 \[
 R\Gamma(\Lie \Gamma, \s{O}(\mathring{\exp(p^r \s{L})}) )
 \]
 is identified with the de Rham cohomology of $\s{O}(\mathring{\exp(p^r \s{L})})$. One has 
 \[
  R\Gamma(\Lie \Gamma, \s{O}(\mathring{\exp(p^r \s{L})}) )= R
 \]
 by the Poincar\'e lemma for open polydiscs \cite[Lemma 26]{Tamme}. We deduce that 
 \[
 R\Gamma(\theta_{V}, \s{O}(\mathring{\bb{H}}_r))\cong \s{O}(\mathring{\exp(p^r W^0)}) 
 \]
 is concentrated in degree $0$. This implies that the $\theta_V$-cohomology of \eqref{eqpkawpmaoef} is in degree $0$ and therefore that \eqref{eqVanishingLocanarjojwnqoe} holds. This finishes the proof of the proposition. 
\end{proof}

\begin{remark}\label{remVanishingProductasfos}
Let $G$ be a compact $p$-adic Lie group and let $\widetilde{X}\to X$ be a pro-Kummer-\'etale torsor. Let $X_{\infty}\to X$ be the $\Gamma$-torsor that arises from the coordinates, then the $G\times \Gamma$-torsor $\widetilde{X}_{\infty}=\widetilde{X}\times_X X_{\infty}\to X$ satisfies the hypothesis of Proposition \ref{PropLocalVanishingHigherLocAn}. 
\end{remark}

\subsection{The geometric Sen operator: globalization}
\label{ss:GeoSenOpe-Globalization}

In Propositions \ref{PropMainTheo1LocalVersion}  and \ref{PropSenOperatorsTorsorLocal} we proved  local versions of Theorems \ref{TheoSenOperatorIntro} and \ref{TheoSenBundleofTorsorIntro} respectively. In order to obtain the global version we need to show that the Sen operators glue in toric charts $\psi$ according  to $\Omega^{1}_X(\log)$.

\subsubsection{Key case}

Let $X$ and $Y$ be affinoid fs log smooth adic spaces over $(C,C^{+})$, and suppose we have two toric charts $\psi_X:X\to \bb{S}_{C}^{(e,d-e)}$ and $\psi_Y: Y \to \bb{S}_C^{(g,h-g)}$ respectively. Let $X_{\infty}$ and $Y_{\infty}$ be the pro-Kummer-\'etale torsors over $X$ and $Y$ obtained by taking $p$-th power roots of the coordinates of the charts, and let $\Gamma_X$ and $\Gamma_Y$ denote the Galois groups of $X_{\infty}\to X$ and $Y_{\infty}\to Y$ respectively.  Let $f:Y\to X$ be a morphism over $(C,C^+)$, and let $f^{*}X_{\infty}= Y\times_{X} X_{\infty}$ seen as an object in $X_{\proket}$. The following is a direct generalization of \cite[Lemma 3.4.3]{LuePan}, we thank Lue Pan for the simplifications of a previous proof.

\begin{prop}
\label{PropGluingKeyCase}
Let $\nu_X:X_{\proket}\to X_{\ket}$ and $\nu_Y:Y_{\proket}\to Y_{\ket}$ be the projection of sites.  We have a commutative diagram of $\widehat{\s{O}}_Y$-linear pro-Kummer-\'etale sheaves over $Y$
\[
\begin{tikzcd}
(\Lie \Gamma_X)^{\vee}\otimes_{\bb{Q}_p} \widehat{\s{O}}_Y \ar[r, "\theta_{f^* X_{\infty}}"] & \Omega^1_{Y}(\log)(-1) \otimes_{\s{O}_Y}\widehat{\s{O}}_Y \\ 
f^*\Omega^1_X(\log)(-1)\otimes_{\s{O}_Y}\widehat{\s{O}}_Y \ar[ur, "f^*"'] \ar[u, "\widetilde{\beta}"] & 
\end{tikzcd}
\]
where the upper horizontal map is the Sen operator of the $\Gamma_X$-torsor $f^*X_{\infty}\to Y$ of Proposition \ref{PropSenOperatorsTorsorLocal}, the diagonal map is the pullback of differentials, and the left vertical map is the composite 
\[
\Omega^1_X(-1)\xrightarrow{\beta} R^1\nu_{X,*} \widehat{\s{O}}_X \cong (\Lie \Gamma_X)^{\vee} \otimes_{\bb{Q}_p } \s{O}_X
\]
with the first map being the natural isomorphism of Propositions \ref{PropProjectionO} and \ref{PropProjectionO2}, and the second map is obtained from Corollary \ref{CoroProjectionFormula} by computing $R^1\nu_{X,*}\widehat{\s{O}}_X$ as $\Gamma_X$-cohomology. 

\end{prop}
\begin{proof}
For proving the proposition, it suffices to compute the geometric Sen operators over $Y$ for a faithful representation of $\Gamma_X$. Let $V$ be the unique unipotent algebraic representation of $\Gamma_X$ fitting in a short exact sequence
\begin{equation}\label{eqpakefpawaw}
0\to \bb{Q}_pe \to V \to (\Lie \Gamma_X)^{\vee}\to 0
\end{equation}
such that: 

\begin{itemize}

\item The action of $\Lie \Gamma_X$ is trivial on the left and right terms of \eqref{eqpakefpawaw}.

\item Let $\f{Y}^{\vee }\in (\Lie \Gamma_X)^{\vee}$ and $\f{X}\in \Lie \Gamma_X$. Then, for any lift $\widetilde{\f{Y}}^{\vee}\in V$ of $\f{Y}$ we have  $\f{X}\cdot \widetilde{\f{Y}}^{\vee}= \f{Y}^{\vee}(\f{X}) \cdot e$. 

\end{itemize}

It is clear that $V$ is a faithful representation of $\Gamma_X$. Let $V_{Y,\ket}$ be the pro-Kummer-\'etale sheaf over $Y$ that arises from $V$ via the $\Gamma_X$-torsor $f^*X_{\infty}\to Y$ and let 
\[
\Sen_{V,Y}: (\Omega_Y^1(\log))^{\vee}(1)\otimes_{\s{O}_Y}\widehat{\s{O}}_Y \otimes_{\bb{Q}_p}  V_{Y,\ket} \to V_{Y,\ket}\otimes_{\bb{Q}_p}\widehat{\s{O}}_Y  
\] 
be the local Sen operator of $V_{\ket}\otimes_{\bb{Q}_p}\widehat{\s{O}}_Y$ of Proposition \ref{PropMainTheo1LocalVersion}.   Similarly, we let $V_{X,\ket}$ be the pro-Kummer-\'etale sheaf over $X$ associated to $V$ via the $\Gamma_X$-torsor $X_{\infty}\to X$. We also have a geometric Sen operator over $X$
\[
\Sen_{V,X}: (\Omega_X^1(\log))^{\vee}(1)\otimes_{\s{O}_X}\widehat{\s{O}}_X \otimes_{\bb{Q}_p}  V_{X,\ket} \to V_{X,\ket}\otimes_{\bb{Q}_p}\widehat{\s{O}}_X. 
\]
Note that $f^*V_{X,\proket}=V_{Y,\proket}$ as pro-Kummer-\'etale sheaves. We want to show that the following diagram is commutative 
\begin{equation}\label{eqkpasojoaenfa}
\begin{tikzcd}
(\Omega_Y^1(\log))^{\vee}(1)\otimes_{\s{O}_Y}\widehat{\s{O}}_Y \otimes_{\bb{Q}_p}  V_{Y,\ket} \ar[rd, "\Sen_{V,Y}"] \ar[d, "f_*"'] &    \\
f^*(\Omega_X^1(\log))^{\vee}(1)\otimes_{\s{O}_Y}\widehat{\s{O}}_Y \otimes_{\bb{Q}_p}  V_{Y,\ket} \ar[r, "f^* \Sen_{V,X}"']  &  V_{Y,\ket}\otimes_{\bb{Q}_p}\widehat{\s{O}}_Y,
\end{tikzcd}
\end{equation}
where $f_*: (\Omega_Y^1(\log))^{\vee}\to f^*(\Omega_X^1(\log))^{\vee}$ is dual to the pullback map of log differentials. 

For this we need to do a computation. Write $X=\Spa (A,A^+)$ and $Y=\Spa (B,B^+)$ as affinoid spaces. Let $X_{\infty}=\Spa(A_{\infty}, A^+_{\infty})$ and $Y_{\infty}=\Spa(B_{\infty}, B_{\infty}^+)$ be the perfectoid torsors arising from the coordinates of $X$ and $Y$ respectively, and let us write $f^*X_{\infty}\times_X Y_{\infty}= \Spa(B_{\infty,\infty},B_{\infty,\infty}^+)$ for the $\Gamma_{X}\times\Gamma_Y$-torsor over $Y$. We have maps of rings 
\[
\begin{tikzcd}
B \ar[r] & B_{\infty} \ar[r] & B_{\infty,\infty} \\
A \ar[u]\ar[r]  & A_{\infty}  \ar[ur] &  
\end{tikzcd}
\]
By Remark \ref{remVanishingProductasfos} we have that  $B_{\infty,\infty}^{R(\Gamma_X\times\Gamma_Y)-la}=B_{\infty,\infty}^{\Gamma_X\times\Gamma_Y-la}$ sits in degree zero. By Lemma \ref{LemComputationCohomologyToricCoordinates} we can compute 
\[
H^1_{\proket}(Y, \widehat{\s{O}}_Y) \cong H^1(\Gamma_X\times \Gamma_Y, B_{\infty,\infty}).
\] 
By Theorem \ref{TheoResumeLocAn} (4) and the vanishing of higher locally analytic vectors of $B_{\infty,\infty}$  we have that 
\[
H^1(\Gamma_X\times \Gamma_Y, B_{\infty,\infty})= H^{1}(\Lie \Gamma_X\times\Lie \Gamma_Y, B_{\infty,\infty}^{\Gamma_X\times \Gamma_Y-la})^{\Gamma_X\times \Gamma_Y}. 
\]
The map $\beta$ induces an isomorphism 
\begin{equation}\label{eqpakpfaow}
\Omega^1_{B}(\log)(-1)= H^1_{\proet}(Y, \widehat{\s{O}}_Y) \cong H^{1}(\Lie \Gamma_X\times\Lie \Gamma_Y, B_{\infty,\infty}^{\Gamma_X\times \Gamma_Y-la})^{\Gamma_X\times \Gamma_Y}.  
\end{equation}

Let $(\gamma_{i,X})_{i=1}^d$ and $(\gamma_{j,Y})_{j=1}^{h}$ be the standard  basis of the groups $\Gamma_X=\bb{Z}_p(1)^d$ and $\Gamma_Y=\bb{Z}_p(1)^{h}$ obtained after fixing a compatible sequence of $p$-th power roots of unit $\epsilon=(\zeta_{p^n})_n$. Let $(\log \gamma_{i,X})$ and $(\log \gamma_{j,Y})$ be the  corresponding basis of the Lie algebras.

By Corollary \ref{CoroFunctHTiso} we have a commutative diagram
\[
\begin{tikzcd}
\Omega^1_B(\log)(-1)  \ar[r, "\beta_{Y}"]& H^1_{\proet}(Y, \widehat{\s{O}}_Y)  \ar[r,"\sim"] & H^1(\Gamma_Y, B) \cong (\Lie  \Gamma_Y)^{\vee}\otimes_{\bb{Q}_p} B \\ 
\Omega^1_A(\log)(-1) \ar[u, "f^*"] \ar[r, "\beta_X"] & H^1_{\proet}(Y, \widehat{\s{O}}_Y) \ar[r,"\sim"] \ar[u, "f^*"] & H^1(\Gamma_X, A) \cong (\Lie  \Gamma_X)^{\vee} \otimes_{\bb{Q}_p} A \ar[u,"A"]
\end{tikzcd}
\]
where we have identified group cohomology with Lie algebra cohomology of the trivial representation, and where $A=(a_{i,j})$ is the induced matrix from the basis $((\log \gamma_{i,X}^{\vee}))_i$ to $((\log \gamma_{j,Y})^{\vee})_j$.   

Thanks to \eqref{eqpakpfaow}, and the interpretation of Lie algebra cohomology as cocycles, we can find locally analytic functions $(z_{j})_{j=1}^h$ in $B_{\infty,\infty}^{\Gamma_X\times\Gamma_Y-la}$ such that 
\begin{equation}\label{awprkpwfawaef}
\log \gamma_{l,X}\cdot z_k = \delta_{k,l} \mbox{ and } \log \gamma_{l, Y} \cdot z_k = -a_{k,l}.  
\end{equation}
Indeed, the $1$-cocycle $(\log \gamma_{k,X})^{\vee}-\sum_{j} a_{k,j} (\log \gamma_{j,Y})^{\vee}$ in the Koszul complex of $B^{\Gamma_X\times\Gamma_Y-la}$  for $\Lie \Gamma_X\times \Lie \Gamma_Y$ is cohomologically trivial and so there exists an element $z_k\in B^{\Gamma_X\times\Gamma_Y-la}$ such that 
\[
\sum_i (\log \gamma_{i,X}\cdot z_k) (\log \gamma_{i,X})^{\vee} + \sum_j (\log \gamma_{j,Y}\cdot z_k) (\log \gamma_{j,Y})^{\vee}=  (\log \gamma_{k,X})^{\vee}-\sum_{j} a_{k,j} (\log \gamma_{j,Y})^{\vee}
\]
which yields the equations \eqref{awprkpwfawaef}.  

Consider a splitting as $\bb{Q}_p$-vector space $V\cong \bb{Q}_pe\oplus (\Lie \Gamma_X)^{\vee}$ and write $\widetilde{(\log \gamma_{i,X})}^{\vee}$ for the lift of $(\log \gamma_{i,X})^{\vee}$.  Using the elements $v_k$, and after shrinking $\Gamma_X$ and $\Gamma_Y$ if necessary for them to act analytically on the $z_k$, we find a $\Gamma_X\times\Gamma_Y$-equivariant isomorphism 
\[
(V\otimes_{\bb{Q}_p} B_{\infty,\infty})\cong B_{\infty,\infty}e \oplus \bigoplus_{i=1}^{h} v_k B_{\infty,\infty}
\]
such that $v_k= \widetilde{(\log \gamma_{k,X})}^{\vee}-z_k\cdot e$ has a trivial action of $\Gamma_X$. Taking invariants under $\Gamma_{X}$ we get that 
\[
(V\otimes_{\bb{Q}_p} B_{\infty,\infty})^{\Gamma_X}\cong B_{\infty}e \oplus \bigoplus_{i=1}^{h} v_k B_{\infty}
\] 
as $B_{\infty}$-semilinear $\Gamma_Y$-representation.  The basis $(e, v_1,\ldots, v_h)$ is $\Gamma_Y$-locally analytic, $e$ is already $\Gamma_{Y}$-invariant, and the other vectors satisfy 
\[
(\log \gamma_{j,Y}) \cdot v_k = - ((\log \gamma_{j,Y}) \cdot z_k) e = a_{k,j} e. 
\]
This shows the commutativity of \eqref{eqkpasojoaenfa} proving the proposition
\end{proof}

\begin{theo}
\label{TheoGluingCaseGamma}
Let $\s{F}$ be an relative locally analytic ON Banach $\widehat{\s{O}}_X$-sheaf over $X$. Then the local geometric Sen operators of Proposition \ref{PropMainTheo1LocalVersion} glue to an $\widehat{\s{O}}_X$-linear global geometric Sen operator
\[
\theta_{\s{F}}:\s{F}\to \s{F}\otimes_{\s{O}_X} \Omega^{1}_X(\log) (-1).
\]
Furthermore, the following properties hold:

\begin{enumerate}

\item The formation of $\theta_{\s{F}}$ is functorial in $\s{F}$ and $\theta_{\s{F}}\wedge \theta_{\s{F}}=0$.

\item Let $\nu:X_{\proket}\to X_{\ket}$ be the projection from the pro-Kummer-\'etale site to the Kummer-\'etale site, then there is a natural equivalence
\[
R^{i}\nu_* \s{F}= \nu_* H^{i}(\theta_{\s{F}},\s{F}),
\]
where $H^{i}(\theta_{\s{F}},\s{F})$ is the cohomology of the Higgs complex
\[
0\to \s{F}\xrightarrow{\theta_{\s{F}}} \s{F}\otimes_{\s{O}_X} \Omega^{1}_X(\log)(-1) \to\cdots \to \s{F}\otimes_{\s{O}_X} \Omega_{X}^{d}(\log)(-d)\to 0.
\]

\item Suppose that $\theta_{\s{F}}=0$, then  $\nu_* \s{F}$ is locally on the Kummer-\'etale topology of $X$ an ON  $\s{O}_X$-Banach module and  $\s{F}= \widehat{\s{O}}_X \widehat{\otimes}_{\s{O}_{X}}\nu_*\s{F}$. Conversely, for any locally ON $\s{O}_X$-Banach module $\s{G}$ the pullback $\widehat{\s{O}}_X\otimes_{\s{O}_X} \s{G}$ has trivial Sen operator.

\item  If $X$ has a form $X'$ over a discretely valued field with perfect residue field $(K,K^+)$, then $\theta_{\s{F}}$ is Galois equivariant. In particular, we recover the natural splitting 
\[
R\nu_* \widehat{\s{O}}_X = \bigoplus_{i=0}^d \Omega_X^{i}(\log)(-i)[-i]. 
\] 

\item Let $Y$ be another fs log smooth adic space over $(C,C^{+})$, and let $f:Y\to X$ be a morphism. Then there is a commutative diagram of geometric Sen operators
\[
\begin{tikzcd}
f^{*}\s{F} \ar[r, "f^* \theta_{\s{F}}"] \ar[rd, "\theta_{f^*\s{F}}"']& f^{*}\s{F} \otimes_{\s{O}_Y} f^{*}\Omega_X^{1}(\log)(-1) \ar[d, "\id\otimes f^*"] \\ 
			& f^{*}\s{F}\otimes_{\s{O}_Y} \Omega^{1}_Y(\log)(-1).
\end{tikzcd}
\]

\end{enumerate}
\end{theo}
\begin{proof}

The gluing of the local Sen operators of Proposition \ref{PropMainTheo1LocalVersion} follows from Proposition \ref{PropGluingKeyCase}. Indeed, we can suppose that $X$ is affinoid and that has two charts $\psi_1,\psi_2:X\to \bb{S}_C^{(e,d-e)}$ (these exists by definition of $X$ having log structure arising from reduced normal crossing divisors, see \cite[Example 2.3.17]{DiaoLogarithmic2019}). Let $X_{\infty,0}$ and $X_{0,\infty}$ be the pro-Kummer-\'etale $\Gamma$-torsors obtained from the perfectoid toric charts $\psi_1$ and $\psi_2$-respectively, let $X_{\infty,\infty}=X_{\infty,0}\times_X X_{0,\infty}$. Let us write $\Gamma_1$ and $\Gamma_2$ for the Galois groups of $X_{\infty,0}$ and $X_{0,\infty}$ respectively.  Then,  by Proposition \ref{PropSenTheoryAffinoids} and   Theorem \ref{TheoSenFunctor} there is $n>>0$ such that  we have a $\Gamma_1\times\Gamma_2$-equivariant isomorphism
\[
\widehat{\s{O}}_X(X_{\infty,\infty}) \widehat{\otimes}_{\s{O}_X(X_{\infty,0})^{p^{n}\Gamma_1}}\s{F}(X_{\infty,0})^{p^{n}\Gamma_1-an} = \s{F}(X_{\infty,\infty}) =\widehat{\s{O}}_X(X_{\infty,\infty}) \widehat{\otimes}_{\s{O}_X(X_{0,\infty})^{p^{n}\Gamma_2}}\s{F}(X_{0,\infty})^{p^{n}\Gamma_2-an}.
\]
Therefore, by Proposition \ref{PropSenOperatorsTorsorLocal}, the action of the Sen operators  $\Lie \Gamma_1\otimes_{\bb{Q}_p} \widehat{\s{O}}_X \cong (\Omega^1_X(\log))^{\vee}(1)\otimes_{\s{O}_X}\widehat{\s{O}}_X$  on $\s{F}$  via the chart $\psi_1$ can be either computed as the extension of the scalars via the natural derivations on $\s{F}(X_{\infty,0})^{p^n\Gamma_1-an}$, or as the extension of the scalars of the action of $\Lie \Gamma_2$ on $\s{F}(X_{0,\infty})^{p^n\Gamma_2-an}$ by precomposing with the Sen map 
\[
\Sen: (\Lie \Gamma_1)\otimes_{\bb{Q}_p}\widehat{\s{O}}_X \to (\Lie \Gamma_2)\otimes_{\bb{Q}_p}\widehat{\s{O}}_X\cong (\Omega^1_X(\log))^{\vee}(1)\otimes_{\s{O}_X}\widehat{\s{O}}_X.
\]
These two actions agree thanks to Proposition \ref{PropGluingKeyCase}. 

Then, part (1) follows from the gluing and the functoriality of Proposition \ref{PropMainTheo1LocalVersion}. Parts (2) and (3) are   consequences of parts (2) and (3)  of \textit{loc. cit.} respectively. For part (4), note that the Galois action of $\Gal_K$ on the Lie algebra $\Lie \Gamma$ induced by some local toric coordinate $\psi:X\to \bb{S}_{C}^{(e,d-e)}$ is given by the cyclotomic character $\chi$. Finally, part (5) follows from Proposition \ref{PropGluingKeyCase} by using an analogue argument as in the beginning of the proof. 
\end{proof}

Under further assumptions on the space $X$ and the sheaf $\s{F}$ we can even compute the projection to the analytic site:

\begin{cor}
\label{CoroProjectionAnalyticSite}
Let $\eta:X_{\proket}\to X_{\an}$ be the projection of sites.  Suppose that  that $\s{F}$ admits a lattice $\s{F}^+$ such that $\s{F}^{+}/p^{\epsilon}\cong^{ae}\bigoplus_I \s{O}^+_X/p^{\epsilon}$ for some $\epsilon>0$ and some index set $I$, and that $X$ admits toric charts locally in the analytic topology.  Then there are natural isomorphisms
\[
R^{i}\eta_* \s{F} = \eta_* H^{i}(\theta_{\s{F}},\s{F}). 
\]
\end{cor}
\begin{proof}
This follows from Theorems \ref{TheoSenFunctor}, \ref{TheoGluingCaseGamma}  and Corollary \ref{CoroProjectionFormula} by applying Sen theory to toric coordinates arising locally in the analytic topology of $X$. 
\end{proof}

\subsubsection{Gluing for general $G$}

We have made all the preparations to show Theorem \ref{TheoSenBundleofTorsorIntro}.  

\begin{theo}
\label{TheoSenBundle}
Let $X$ be an fs log smooth adic space over  $\Spa(C,C^+)$ with log structure given by normal crossing divisors. Let   $G$ a $p$-adic Lie group and $\widetilde{X}\to X$ a  pro-Kummer-\'etale $G$-torsor.     Then the geometric Sen operators of Proposition \ref{PropSenOperatorsTorsorLocal} given by local charts of $X$ glue to a morphism of  $\widehat{\s{O}}_X$-vector bundles over $X_{\proket}$
\[
\theta_{\widetilde{X}}:  \widehat{\s{O}}_X  \otimes_{\bb{Q}_p}  (\Lie G)^{\vee}_{ \ket} \to \widehat{\s{O}}_X(-1)     \otimes_{\s{O}_{X}} \Omega_X^1(\log) 
\]
such that $\theta_{\widetilde{X}}\bigwedge \theta_{\widetilde{X}}=0$. In particular, for any locally analytic Banach representation $V$ of $G$, we have a commutative diagram
\[
\begin{tikzcd}
V_{\ket}\widehat{\otimes }_{\bb{Q}_p} \widehat{\s{O}}_X \ar[r, "d_V\otimes \id_{\widehat{\s{O}}}"] \ar[rd,"\theta_{V}"'] & (V_{\ket}\widehat{\otimes}_{\bb{Q}_p} \widehat{\s{O}}_X )  \otimes_{\bb{Q}_p} (\Lie G)_{\ket}^{\vee}  \ar[d,"\id_{V}\otimes \theta_{\widetilde{X}}"] \\
& (V_{\ket} \widehat{\otimes}_{\bb{Q}_p} \widehat{\s{O}}_X(-1)) \otimes_{\s{O}_X}  \Omega^{1}_{X}(\log)
\end{tikzcd}
\]
such that $d_V:V\to  V\otimes (\Lie G)^{\vee}$ is induced by  derivations, and $\theta_V$ is the geometric Sen operator of $V_{\ket}\widehat{\otimes}_{\bb{Q}_p}\widehat{\s{O}}_X$.

Moreover, let  $ H\to G$ be a morphism of $p$-adic Lie groups,    let  $Y$ an fs log smooth adic space over $(C,C^+)$  and  let $\widetilde{Y}\to Y'$ be an $H$-torsor.  Suppose we are given with a commutative diagram compatible with the group actions
\begin{equation}
\label{eqMaoTorsorsSen}
\begin{tikzcd}
\widetilde{Y}  \ar[r] \ar[d] & \widetilde{X} \ar[d]  \\ 
Y  \ar[r,"f"]& X
\end{tikzcd}.
\end{equation}
Then the following square is commutative 
\begin{equation}
\label{eqFunctSenBundle}
\begin{tikzcd}
 f^*(\Lie G)^{\vee}_{\ket}\otimes_{\bb{Q}_p} \widehat{\s{O}}_Y \ar[d] \ar[r, "f^{*} \theta_{\widetilde{X}}"] & f^* \Omega_X^1(\log) \otimes_{\s{O}_Y} \widehat{\s{O}}_Y(-1)   \ar[d] \\
(\Lie H)^{\vee}_{\ket} \otimes_{\widehat{\bb{Q}}_p} \widehat{\s{O}}_Y \ar[r,"\theta_{\widetilde{Y}}"] &   \Omega_Y^1(\log) \otimes_{\s{O}_Y} \widehat{\s{O}}_Y(-1) 
\end{tikzcd}
\end{equation}

\end{theo}
\begin{proof}
We first construct the geometric Sen operator $\theta_{\widetilde{X}}$ for the $G$-torsor $\widetilde{X}\to X$.  Let $V$ be a locally analytic representation of $G$ on a Banach space, and consider the inclusion of the orbit map 
\[
V\hookrightarrow V \widehat{\otimes}_{\bb{Q}_p} C^{la}(G,\bb{Q}_p)_{\star_2},
\] 
where $G$ acts via the right regular action on the right hand side. By writing $C^{la}(G,\bb{Q}_p)=\varinjlim_h C^{h}(G,\bb{Q}_p)$ as colimit of $h$-analytic  functions of $G$ (for some fixed  local coordinates), we can just assume without loss of generality that $V=C^{h}(G,\bb{Q}_p)_{\star_2}$ endowed with the right regular action. Let $\s{F}=C^{h}(G,\bb{Q}_p)_{\star_2,\ket}\widehat{\otimes}_{\bb{Q}_p} \widehat{\s{O}}_X$ be the relative locally analytic ON $\widehat{\s{O}}_X$-Banach module over $X$ defined by $V$, and let $\theta_{\s{F}}:\s{F}\to \s{F}\otimes_{\s{O}_X} \Omega^{1}_{X}(\log)(-1)$ be the geometric Sen operator of $\s{F}$ of Theorem \ref{TheoGluingCaseGamma}. Consider the adjoint map
\[
\Sen_{\s{F}}: \Omega^{1,\vee}_X(\log)(1)\otimes_{\s{O}_X} \s{F}\to \s{F}.
\]
By Proposition \ref{PropSenOperatorsTorsorLocal}, locally in the Kummer-\'etale topology of $X$,  $\Sen_{\s{F}}$ acts by $\widehat{\s{O}}_X$-linear derivations on $\s{F}$ that are in addition left $G$-equivariant. Then, the local maps of  Proposition \ref{PropSenOperatorsTorsorLocal} glue to an $\widehat{\s{O}}_X$-linear morphism
\[
\Sen_{\widetilde{X}}: \Omega^{1,\vee}_X(\log)\otimes_{\s{O}_X} \widehat{\s{O}}_X(1)\to (\Lie G)_{\ket}\otimes_{\bb{Q}_p}\widehat{\s{O}}_X,
\]
or equivalently, by taking adjoints to a map 
\[
\theta_{\widetilde{X}}: (\Lie G)_{\ket}^{\vee}\otimes_{\bb{Q}_p} \widehat{\s{O}}_X \to  \Omega^{1}_X(\log)\otimes_{\s{O}_X} \widehat{\s{O}}_X(-1)
\]
satisfying the conclusion of the theorem.

Next, we prove functoriality for the commutative square   \eqref{eqMaoTorsorsSen}.  Consider the space $C^{la}(G,\bb{Q}_p)=\varinjlim_h C^{h}(G,\bb{Q}_p)$ of locally analytic functions of $G$, written as colimit of $h$-analytic functions.  By restriction, we can also see $C^{h}(G,\bb{Q}_p)$ as a $H$-representation via the map $H\to G$.   Let 
\[
\s{F}=C^{h}(G,\bb{Q}_p)_{\star_2,\ket}\widehat{\otimes}_{\bb{Q}_p} \widehat{\s{O}}_X
\] 
be the pro-Kummer-\'etale sheaf over $X$ associated to $C^h(G,\bb{Q}_p)$. The sheaf $f^*\s{F}$ over $Y$ is also the  $\widehat{\s{O}}_Y$-extension of scalars of the pro-Kummer-\'etale sheaf associated to the $H$-representation $C^h(G,\bb{Q}_p)|_{H}$. Let $\theta_{\s{F}}$ and $\theta_{f^*\s{F}}$ be the geometric Sen operators of $\s{F}$ and $f^*\s{F}$ respectively.    By Theorem \ref{TheoGluingCaseGamma} (5) we have a commutative diagram 
\[
\begin{tikzcd}
f^{*}\s{F} \ar[r, "f^* \theta_{\s{F}}"] \ar[rd, "\theta_{f^*\s{F}}"']& f^{*}\s{F} \otimes_{\s{O}_Y} f^{*}\Omega_X^{1}(\log)(-1) \ar[d, "\id\otimes f^*"] \\ 
			& f^{*}\s{F}\otimes_{\s{O}_Y} \Omega^{1}_Y(\log)(-1).
\end{tikzcd}
\]
The action of $f^*\theta_{\s{F}}$ and $\theta_{f^*\s{F}}$ is also  given by left $G$-invariant $\widehat{\s{O}}_Y$-linear derivations, this implies that they must arise from the Lie algebra action of $\Lie G$ of  the  square  \eqref{eqFunctSenBundle} proving that it commutes.  This finishes the proof of the theorem. 
\end{proof}

The Sen morphism should encode the directions of perfectoidness of $X$.    We have the following conjecture, which is a generalization of a theorem of Sen saying that a $p$-adic Galois representation of $p$-adic field has vanishing Sen operator if and only if it is potentially  unramified,  see Corollary 3.32 of \cite{FontaineOuyang}.  

\begin{conjecture}
\label{Conjecture1}
Let $X$ be an fs log smooth adic space over $(C,C^+)$ with normal crossing divisors, let $G$ be a compact $p$-adic Lie group and $\widetilde{X}\to X$ a pro-Kummer-\'etale $G$-torsor. Then the  geometric Sen operator $\theta_{\widetilde{X}}:  \widehat{\s{O}}_X \otimes_{\widehat{\bb{Q}}_p} (\Lie G)^{\vee}_{\ket} \to   \widehat{\s{O}}_X \otimes_{\s{O}_X} \Omega_X^1(\log)^{\vee}(-1) $  is surjective if and only if $\widetilde{X}$ is a perfectoid space.  
\end{conjecture}

\begin{remark}
In the work \cite{RCLocAnACoho}, we show that  the pro(-Kummer-)\'etale torsors defining the infinite level  Shimura varieties satisfy the hypothesis of the Conjecture \ref{Conjecture1}.    The proof of this fact never uses the perfectoidness of the  Shimura variety,   only  the $p$-adic Riemann-Hilbert correspondence of  \cite{DiaoLogarithmicHilbert2018}. Moreover, in \cite{RCLocAnACoho} we use the explicit construction of the geometric Sen operators for Shimura varieties to prove the vanishing part of the rational Calegari-Emerton conjectures, see \cite{CalegariEmerton}.
\end{remark}

\begin{remark}
In his  original work  \cite{MR319949,MR314853}, Sen shows that an algebraic extension $K$ of $\bb{Q}_p$ of Galois group a $p$-adic Lie group $G$  is deeply ramified (equivalently, that its completion is perfectoid) if and only if the Sen operator seen as an object $\theta\in \Lie G\otimes_{\bb{Q}_p}\widehat{K}$   is non zero.  Recently in \cite{he2024perfectoidness}, He proved that the Conjecture \ref{Conjecture1}  holds for  residue fields of  rigid varieties. This pointwise perfectoidness criteria is  sufficient for proving the vanishing of the Calegari-Emerton conjectures integrally.  
\end{remark}

\subsection{Locally analytic vectors of pro-Kummer-\'etale towers}
\label{s:HTSrigidProetalecoho}

We keep the previous notations,  i.e., $(C,C^+)$ is a perfectoid field containing $\bb{Q}_p^{\cyc}$,   $X$ is an fs log smooth adic space over  $(C,C^+)$ with log structure given by normal crossing divisors,  $G$ a compact $p$-adic Lie group,  and $\widetilde{X}$ a pro-Kummer-\'etale $G$-torsor over $X$.  In this last section we apply Theorem \ref{TheoGluingCaseGamma} to study the locally analytic vectors of $\widehat{\s{O}}_{\widetilde{X}}$ for the action of $G$, we also extend the cohomology computations of the theorem to log adic spaces arising from the boundary of $X$.

\subsubsection{Pro-Kummer-\'etale cohomology of the boundary}

Let $X$ and $\widetilde{X}\to X$ be as before.  Let $D\subset X$ be the boundary divisor. By definition,  \'etale locally on $X$,  $D$ can be written as a disjoint union of irreducible components $D=\bigcup_{a\in I} D_a$,  where the finite intersections of the $D_a$'s are smooth, cf. \cite[Examples 2.3.17 and 2.3.18]{DiaoLogarithmic2019}. For the rest of the section we work with a boundary divisor $D$ decomposed as an union $D=\bigcup_{a\in I}= D_a$ as before.     Given $J\subset I$ a finite subset  we set $D_J=\bigcap_{a\in I} D_a $ endowed with the log structure pulled back from $X$,  and  write $\iota_{J}:D_{J}\subset X$ for the inclusion map.  We  denote by $\widehat{\s{O}}_{D_J}^{(+)}$  the sheaf $\iota_{J,*} \widehat{\s{O}}_{D_J}^{(+)}$ over $X_{\proket}$.    

In this section shall write $\s{O}_{X,\ket}$ and $\s{O}_{X,\an}$ for the structural sheaves on the Kummer-\'etale site and the analytic site respectively. We also denote by  $\s{O}_{D_J,\ket}$ and $\s{O}_{D_J,\an}$ for the sheaves defined by the boundary divisor $D_J$. 

We start with  a partial extension of Theorem \ref{TheoGluingCaseGamma} to  the boundary.

\begin{prop}
\label{TheoCohomologyLASheavesBoundary}
Let $\s{F}$ be a relative locally analytic $\widehat{\s{O}}_X$-module over $X_{\proket}$ and let $\theta_{\s{F}}:\s{F}\to \s{F}\otimes_{\s{O}_X}\Omega^{1}_{X}(\log)(-1)$ be the geometric Sen operator of $\s{F}$. Let $\nu_*:X_{\proket}\to X_{\ket}$ be the projection of sites, and  let $J\subset I$ be a finite subset. Then there is a natural isomorphism 
\begin{equation}
\label{eqCohoGroupsBoundary}
R^{i}\nu_*  \iota_{J,*}\iota_{J}^{*} \s{F}= \nu_*( H^{i}(\theta_{\s{F}}, \iota_*\iota_{J}^{*}\s{F})). 
\end{equation}
Moreover, let $\eta:X_{\proket}\to X_{\an}$ be the projection of sites and suppose that $\s{F}$ has a lattice $\s{F}^{+}$ such that $\s{F}^{+}/p^{\epsilon}=^{ae}\bigoplus_I \s{O}^+_{X}/p^{\epsilon}$ as Kummer-\'etale $\s{O}^+_X/p^{\epsilon}$-modules for some $\epsilon>0$. Then 
\begin{equation}
\label{eqCohoGroupsBoundary2}
R^{i}\eta_*  \iota_{J,*}\iota_{J}^{*} \s{F}= \eta_*( H^{i}(\theta_{\s{F}}, \iota_*\iota_{J}^{*}\s{F})) . 
\end{equation}
\end{prop}

\begin{proof}
All the statements are local on $X$ for the Kummer-\'etale or analytic topology, so we can assume that we have a toric chart $\psi:X\to \bb{S}_{C}^{(e,d-e)}$, and take $D=D_J$ to be the vanishing locus of $S_{e+1}=\cdots= S_{d}=0$. We also write $\iota\colon D\to X$ for the inclusion map. We can also assume that $\s{F}$ admits an almost $\widehat{\s{O}}^+_X$-lattice $\s{F}^+$  as in Definition \ref{DefRelativeLocAnSheaf}. In particular, notice that 
\[
\iota_*\iota^*( \s{F}^+/p^{\epsilon}) \cong^{ae} \bigoplus_{I} \s{O}^+_{D}/p^{\epsilon}. 
\]

 Let $X_{\infty}$ be the $\Gamma\cong \bb{Z}_p(1)^{d}$-torsor over $X$ obtained via $\psi$ by taking $p$-th power roots of the  coordinates $T_i$ and $S_j$, set  $D_{\infty}=X_{\infty}\times_X D$. Let $A_{\infty}=\widehat{\s{O}}_X(X_{\infty})$ be the sheaf of functions of $X_{\infty}$ and $A_{\infty, D}= \widehat{\s{O}}_{D}(D_{\infty})$. By Example \ref{ExampleSenTheoryProduct} (5), the boundary of $\bb{S}_{\infty}^{(e,d-e)}=\Spa(C\langle T_i^{\frac{1}{p^{\infty}}},S_j^{\frac{1}{p^{\infty}}} \rangle, C^+\langle T_i^{\frac{1}{p^{\infty}}},S_j^{\frac{1}{p^{\infty}}} \rangle)$ admits Sen traces for the action of $\Gamma$, and by Proposition \ref{PropSenTheoryAffinoids}   these Sen traces extend to a $d$-dimensional Sen theory $(A_{D,\infty}, \Gamma, (R_n^i)_{n,i})$ to the boundary of $X_{\infty}$. 

By  Lemma \ref{LemComputationCohomologyToricCoordinates} applied to  $\iota_{*}\iota^*\s{F}$, we have  that 
\[
R\Gamma(X_{\infty}, \iota_{*}\iota^*\s{F}) =\s{F}(D_{\infty}).
\]
The condition on $\s{F}^+/p^{\epsilon}$ together with derived Nakayama's lemma yield
\[
\s{F}(D_{\infty})=^{ae} A_{D,\infty}\widehat{\otimes}^L_{A_{\infty}} \s{F}(X_{\infty}) 
\]
 as $\Gamma$-equivariant $A_{\infty}$-modules. 
 
The Sen operators of $\s{F}(D_{\infty})$ are constructed from the action of $\Lie \Gamma$ on its locally analytic vectors. By the projection formula of locally analytic vectors of Lemma \ref{LemmaProjectionFormulaLocAnRep} we get that 
\[
\begin{aligned}
\s{F}(D_{\infty})^{\Gamma-la}& =(A_{D,\infty}\widehat{\otimes}^L_{A_{\infty}} \s{F}(X_{\infty}))^{\Gamma-la}  \\
	& = (A_{D,\infty}\widehat{\otimes}^L_{A_{\infty}^{\Gamma-la}} \s{F}(X_{\infty})^{\Gamma-la})^{\Gamma-la} \\
	& = A_{D,\infty}^{\Gamma-la}\widehat{\otimes}^L_{A_{\infty}^{\Gamma-la}} \s{F}(X_{\infty})^{\Gamma-la}.
\end{aligned}
\] 
where in the second equivalence we have used Theorem \ref{TheoSenFunctor} to decomplete $\s{F}(X_{\infty})$ via locally analytic vectors.

This shows that the geometric Sen operators of $\iota_*\iota^*\s{F}$ are just the base change along $\widehat{\s{O}}_X\to \widehat{\s{O}}_{D}$ of the geometric Sen operators of $\s{F}$.  Since $(A_{D,\infty}, \Gamma, (R_n^i)_{n,i})$ is a Sen theory satisfying (AST), the equations \eqref{eqCohoGroupsBoundary} and  \eqref{eqCohoGroupsBoundary} follow from Proposition  \ref{PropositionSenTheoryCohomology} as the left hand sides are computed via $\Gamma$-group cohomology. 
\end{proof}

\subsubsection{Locally analytic vectors of \texorpdfstring{$\widehat{\s{O}}_X$}{Lg}}
\label{ss:HTSrigidLocAnO} 

Let $(C,C^{+})$  be perfectoid field containing $\bb{Q}_p^{\cyc}$, $X$  an fs log smooth adic space over $(C,C^+)$ with 
normal crossing divisors, and $\widetilde{X}\to X$ a pro-Kummer-\'etale torsor for a compact $p$-adic Lie group $G$.  For an open subgroup $G_0\subset G$ we let $X_{G_0}$ be the finite Kummer-\'etale extension over $X$ given by the quotient $\widetilde{X}/G_0$ in $X_{\proket}$. The space $\widetilde{X}$ has an underlying topological space given by $|\widetilde{X}|=\varprojlim_{G_0\subset G} |X_{G_0}|$, where $|X_{G_0}|$ is the underlying adic space of $X_{G_0}$. The analytic site of $\widetilde{X}$ is the site of disjoint unions of open subspaces of  $|\widetilde{X}|$.

\begin{definition}
Let $\widehat{\s{O}}_{\widetilde{X}}$ be the restriction of the completed structural sheaf of $X_{\proket}$ to the analytic site of $\widetilde{X}$. We let $\s{O}^{la}_{\widetilde{X}}\subset \widehat{\s{O}}_{\widetilde{X}}$ be the presheaf mapping a qcqs open subspace $\widetilde{U}\subset \widetilde{X}$ to the Ind-Banach space
\[
\s{O}^{la}_{\widetilde{X}}(\widetilde{U}):=\widehat{\s{O}}_{\widetilde{X}}(\widetilde{U})^{G_{\widetilde{U}}-la}
\]
of $G_{\widetilde{U}}$-locally analytic sections of $\widehat{\s{O}}_{\widetilde{X}}(\widetilde{U})$, where $G_{\widetilde{U}}$ is the stabilizer of $\widetilde{U}$.
\end{definition}

\begin{lem}
The presheaf $\s{O}^{G-la}_{\widetilde{X}}$ is a sheaf for the analytic topology of $\widetilde{X}$. 
\end{lem}
\begin{proof}
Note that for any qcqs open subspace $\widetilde{U}\subset \widetilde{X}$, the ring $\widehat{\s{O}}_{\widetilde{X}}(\widetilde{U})$ is a $\bb{Q}_p$-Banach space. Let $\{\widetilde{V}_i\}_{i=1}^{s}$ be an open cover of $\widetilde{U}$ by qcqs open subspaces. We have a short exact sequence
\begin{equation}
\label{eqShortExactOhat}
0\to \widehat{\s{O}}_{\widetilde{X}}(\widetilde{U}) \to \prod_{i} \widehat{\s{O}}_{\widetilde{X}}(\widetilde{V}_i) \to \prod_{i,j}\widehat{\s{O}}_{\widetilde{X}}(\widetilde{V}_i \cap \widetilde{V}_j).
\end{equation}
Let $G_0\subset G$ be an open compact subgroup stabilizing all the $\widetilde{V}_i$'s, then $G_0$ stabilizes $\widetilde{U}$ and the intersections $\widetilde{V}_i \cap \widetilde{V}_j$, and \eqref{eqShortExactOhat} is a $G_0$-equivariant exact sequence of Banach spaces. Then, tensoring with $C^{la}(G_0, \bb{Q}_p)_{\star_1}$ and taking $G_0$-invariant vectors we get an exact sequence
\[
0\to \s{O}^{la}_{\widetilde{X}}(\widetilde{U}) \to \prod_{i} \s{O}^{la}_{\widetilde{X}}(\widetilde{V}_i) \to \prod_{i,j}\s{O}^{la}_{\widetilde{X}}(\widetilde{V}_i \cap \widetilde{V}_j).
\]
proving what we wanted. 
\end{proof}

The following definition is useful to construct sheaves over $\widetilde{X}$.

\begin{definition}
\label{DefinitionTowerSheaf}
Let $G_0\subset G$ be an open compact subgroup, denote  by $\eta_{G_0}: X_{G_0,\proket}\to X_{G_0,\an}$ the projection of sites. Let $\s{F}:=(\s{F}_{G_0})_{G_0\subset G}$ be a  compatible sequence of pro-Kummer-\'etale sheaves on $X_{G_0}$ in the sense that if $G_0'\subset G_0$ we have a map $\psi_{G_0}^{G_0'}:\s{F}_{G_0}|_{X_{G_0'}}\to \s{F}_{G_0'}$, and that for $G_0''\subset G_0'\subset G_0$ we have $\psi_{G_0'}^{G_0''} \circ \psi_{G_0}^{G_0'}=\psi_{G_0}^{G_0''}$. We define 
\[
R\eta_{\infty,*} \s{F}:= \varinjlim_{G_0} R\eta_{G_0,*} \s{F}_{G_0}
\]
seen as sheaf  over $|\widetilde{X}|$.
\end{definition}

Let $D\subset X$ be the boundary divisor and  suppose that $D=\bigcup_{i\in I} D_i$ is written as a disjoint union of irreducible divisors with smooth finite intersections. Given $J\subset I$ a finite subset we let $D_{J}=\cap_{i\in J} D_i$, we declare $D_{\emptyset}= X$. 

\begin{theo}
\label{TheoProjectionLocallyAnalyticVectors}
Let $\eta:  X_{ \proket}\to X_{\an}$ be the projection of sites, and let 
\[
C^{la}(\Lie G, \bb{Q}_p)_{\star_1,\ket}:=(C^{la}( G_0, \bb{Q}_p)_{\star_1,\ket} )_{G_0\subset G}
\]
 be the ind-sequence of pro-Kummer-\'etale sheaves obtained from left regular representations over the tower $(X_{G_0})_{G_0\subset G}$ as in Definition \ref{DefinitionTowerSheaf}. Suppose that the following axiom holds: 
 
 \begin{itemize}
 
\item[(BUN)]  The geometric Sen operator of $\widetilde{X}$
\[
\theta_{\widetilde{X}}: \widehat{\s{O}}_X \otimes_{\bb{Q}_p} (\Lie G)^{\vee}_{\ket} \to \widehat{\s{O}}_X \otimes_{\s{O}_X} \Omega^1_X(\log)(-1)
\]
is surjective.
 \end{itemize}

Then, for any finite subset $J\subset I$, we have that 

\begin{gather*}
R\eta_{\infty,*} (C^{la}(\Lie G, \bb{Q}_p)_{\star_1,\ket}\widehat{\otimes}_{\bb{Q}_p}  \widehat{\s{O}}_{D_J})= \s{O}_{D_J}^{la},
\end{gather*}
where  $\s{O}_{D_J}^{la}$ is the sheaf of locally analytic sections of $\widehat{\s{O}}_{D_J}$ restricted to $|\widetilde{X}|$.   Furthermore, we have an exact sequence of sheaves in $|\widetilde{X}|$
\begin{equation}
\label{eqpafpasmfa}
\s{O}^{la}_{\widetilde{X}}\to \bigoplus_{a\in I} \s{O}^{la}_{D_a}\to \bigoplus_{\substack{J\subset I \\ |J|=2}}  \s{O}^{la}_{D_J}\to\cdots \to \s{O}_{D_I}^{la}\to 0
\end{equation}
induced by the boundary divisors. 

\end{theo}
\begin{proof}
This follows formally from Propositions \ref{PropLocalVanishingHigherLocAn} and  \ref{TheoCohomologyLASheavesBoundary}. Indeed, we can work locally in the analytic topology of $X$, and even after a finite Kummer-\'etale cover of the form $X_{G_0}$. Thus, we can assume without loss of generality that $X=\Spa(A,A^+)$ is affinoind and that it admits a chart to $\bb{S}^{(e,d-e)}_{C}$ with coordinates $T_i, S_j$ such that $D=D_J$ is the vanishing locus of   $S_{e+1}=\cdots=S_d=0$.  Let us write $V_{G_0}= C^{la}(G_0,\bb{Q}_p)_{\star_1}$ for the left regular representation of $G_0$, we want to see that 
\[
\varinjlim_{G_0} R\Gamma(X_{G_0,\proket}, V_{G_0,\ket}\widehat{\otimes}_{\bb{Q}_p} \widehat{\s{O}}_{D})=\s{O}^{la}_{D}
\]
The vanishing of higher cohomology groups follows from Proposition \ref{PropLocalVanishingHigherLocAn}. The identification of the degree $0$-cohomology is also clear. 

It is left to show that \eqref{eqpafpasmfa} is an exact sequence.  By Theorem \ref{TheoGluingCaseGamma} (2) we know that 
\[
\s{O}^{la}_{D}(\widetilde{X})=  \varinjlim_{G_0} H^0_{\proket}(X_{G_0}, (V_{G_0}\widehat{\otimes}_{\bb{Q}_p} \widehat{\s{O}}_{X} )^{\theta_{V}=0}). 
\]
By the proof of Proposition \ref{PropLocalVanishingHigherLocAn}, the filtered system of pro-\'eatle sheaves $( V_{G_0}\widehat{\otimes}_{\bb{Q}_p} \widehat{\s{O}}_{X} )^{\theta_{V}=0})$ on the tower $X_{G_0}$ is ind-equivalent to a system of pro-\'etale sheaves $(\s{F}_{G_0})_{G_0}$ where each $\s{F}_{G_0}$ is a relative locally analytic sheaf on $X_{G_0}$ admitting a lattice $\s{F}_{G_0}^+$ as in Definition \ref{DefRelativeLocAnSheaf}. Indeed, in the notation of the proposition, these are the sheaves corresponding to 
\[
\s{O}(\exp(p^rW^0))=\s{O}(\bb{H}_r)^{\Lie \Gamma=0}. 
\]
Hence, we can write 
\[
\s{O}^{la}_{D}(\widetilde{X})=\varinjlim_{G_0}  \s{F}_{G_0}(X_{G_0}). 
\]

Let $X_{\infty}$ be the $\Gamma$-torsor over $X$ obtained by adding $p$-th power roots of the coordinates $S_i$ and $T_j$, set $\widetilde{X}_{\infty}=\widetilde{X}_{\infty}$. For $G_0\subset G$ an open compact subgroup denote $X_{G_0,\infty}=X_{G_0}\times_X X_{\infty}$.  Then, we can write 
\[
H^0(X_{G_0}, \s{F}_{G_0})= H^0(G_0\times \Gamma, \s{F}_{G_0}(\widetilde{X}_{\infty}))= H^0(\Gamma, \s{F}(X_{G_0,\infty})). 
\]
By Proposition \ref{PropSenTheoryAffinoids} we have a Sen theory on $\s{O}(X_{G_0,\infty})$ for the action of $\Gamma$. Since the Sen operators of $\s{F}_{G_0}$ are trivial, by Corollary \ref{CoroProjectionFormula}   there is an open subgroup $\Gamma_{G_0}\subset \Gamma$ (depending on $G_0$) such that 
\[
\s{F}_{G_0}= \s{F}(X_{G_0,\infty})^{\Gamma_{G_0}} \widehat{\otimes}_{\s{O}(X_{G_0,\infty})^{\Gamma_{G_0}}} \widehat{\s{O}}_{X_{G_0}}
\]
with $\s{F}(X_{G_0,\infty})^{\Gamma_{G_0}}$ an ON $\s{O}(X_{G_0,\infty})^{\Gamma_{G_0}}$-Banach module.

We have an exact sequence of pro-Kummer-\'etale $\widehat{\s{O}}_{X}$-modules on $X_{\proket}$
\[
\widehat{\s{O}}_X\to \bigoplus_{a\in I} \widehat{\s{O}}_{D_a}\to \bigoplus_{\substack{J\subset I \\ |J|=2}}  \widehat{\s{O}}_{D_J}\to\cdots \to \widehat{\s{O}}_{D_I}\to 0.
\]
tensoring with $\s{F}_{G_0}$, evaluating at $X_{G_0,\infty}$ and taking $\Gamma_{G_0}$-invariant vectors, we get an exact sequence 
\[
\begin{gathered}
\s{F}(X_{G_0,\infty})^{\Gamma_{G_0}}\to \bigoplus_{a\in I } \s{F}(X_{G_0,\infty})^{\Gamma_{G_0}}\widehat{\otimes}_{\s{O}(X_{G_0,\infty})^{\Gamma_{G_0}}} \s{O}_{D_a}(X_{G_0,\infty})^{\Gamma_{G_0}} \to \\ 
\to \cdots   \s{F}(X_{G_0,\infty})^{\Gamma_{G_0}} \widehat{\otimes}_{\s{O}(X_{G_0,\infty})^{\Gamma_{G_0}}} \s{O}_{D_I}(X_{G_0,\infty})^{\Gamma_{G_0}} \to 0
\end{gathered}
\]
(this is the same as tensoring $\s{F}(X_{G_0,\infty})^{\Gamma_{G_0}}$ with the exact sequence induced by  divisors of the rigid space $X_{G_0,\Gamma_0}:= X_{G_0,\infty}/\Gamma_{G_0}$; it remains exact since $\s{F}(X_{G_0,\infty})^{\Gamma_{G_0}}$ admits an ON basis).  Taking group cohomology for the finite group action $\Gamma/\Gamma_{G_0}$, we obtain an exact sequence 
\[
\s{F}_{G_0}(X_{G_0})\to \bigoplus_{a\in I } (\s{F}_{G_0}\widehat{\otimes}_{\widehat{\s{O}}_X} \widehat{\s{O}}_{D_a})(X_{G_0}) \to \cdots \to (\s{F}_{G_0}\widehat{\otimes}_{\widehat{\s{O}}_X} \widehat{\s{O}}_{D_I})(X_{G_0})\to 0.
\]
Taking colimits as $G_0\to 1$ one obtains the exact sequence \eqref{eqpafpasmfa} as wanted. 
\end{proof}

\subsection{Relation with the $p$-adic Simpson correspondence}
\label{ss:HTSrigidpAdicSimpson}

We finish this section with the relationship between the  geometric Sen operator and the $p$-adic Simpson correspondence of \cite{LiuZhuRiemannHilbert,  DiaoLogarithmicHilbert2018, wang2021padic}. We shall write  $(K,K^+)$  for a complete discretely valued non-archimedean extension of $\bb{Q}_p$ with perfect residue field, and let $(C,C^+)$ be the completion of an algebraic closure of $K$. We let $t\in \bb{B}_{\dR}^+(C,C^+)$ be a generator of the kernel of the map of Fontaine $\theta: \bb{B}_{\dR}^+(C,C^+)\to C$, eg. $t=\log ([\epsilon])$ where $\epsilon=(\zeta_{p^n})_{n}$ is a sequence of compatible  $p$-th power roots of unity.

    Let us recall the following result

\begin{theo}[{\cite[Theorem  2.1]{LiuZhuRiemannHilbert} and  \cite[Theorem  3.2.4]{DiaoLogarithmicHilbert2018}} ]
\label{TheoLiuZhuRH}
Let $X$ be an fs log smooth adic space over $\Spa(K,K^+)$ and let $\bb{L}^0$ be a pro-Kummer-\'etale $\bb{Z}_p$-local system, write $\bb{L}=\bb{L}^0[\frac{1}{p}]$.   Let $\OC =\gr^0 \OBdr$ be the Hodge-Tate period sheaf,  and $\nu: X_{C,\proket}\to X_{C,\ket}  $ be the projection of sites.   Then 
\[
\n{H}(\bb{L}):= R\nu_*(\bb{L}\otimes_{\bb{Q}_p} \OC )
\]
is a $\Gal_K$-equivariant  $\log$ Higgs bundle concentrated in degree $0$.   Let $\theta$ denote the Higgs field of $\n{H}(\bb{L})$.  Then one has 
\begin{equation}\label{qwekpafma}
R \nu_* \bb{L}\otimes_{\bb{Q}_p}\widehat{\s{O}}_X = R\Gamma(\theta,  \n{H}(\bb{L})).  
\end{equation}
\end{theo}

To reprove this theorem with our theory we first need to compute the geometric Sen operator of $\OC$. 

\begin{prop}
\label{PropSenOperatorOC}
The geometric Sen operator of   \[\OC = \gr^0(\OBdr) =\varinjlim_n  \Sym^n(\gr^1 \OBdr^+)\cdot  t^{-n} \] is  given by 
\[
\OC \xrightarrow{-\overline{\nabla}} \OC(-1) \otimes_{\s{O}_X} \Omega^1(\log),
\]
where $\overline{\nabla}$ is the residual connection of $\nabla: \OBdr\to \OBdr \otimes_{\s{O}_X} \Omega^1(\log)$.  
\end{prop}
\begin{proof}
It suffices to identify $\theta_{\OC}$ with $-\nabla$ locally on $X$, we can then assume that $X$ has toric coordinates $\psi:X\to \bb{S}^{(e,d-e)}_K$. Let $X_{C,\infty}= X\times_{ \bb{S}_{C}^{(e,d-e)}} \bb{S}_{C,\infty}^{(e,d-e)}$ the $\Gamma$-torsor over $X$. Then, by  \cite[Proposition 2.3.15]{DiaoLogarithmicHilbert2018} we have a presentation 
\[
\OC|_{X_{C,\infty}} = \widehat{\s{O}}_{X_{C,\infty}}\big[ \frac{\log (\underline{T}^{-1}[\underline{T}]^{b})}{t}, \frac{\log(\underline{S}^{-1}[\underline{S}^{\flat}]) }{t}  \big],
\]
where $T^{\flat}=(T^{\frac{1}{p^{n}}})_{n\in \bb{N}}$,  $S^{\flat}=(S^{1/p^{n}})_{n\in \bb{N}}$, and $t=\log([\epsilon
])$.  We see from this presentation that 
\[
\overline{\nabla}(\log( \underline{T}^{-1} [T^{\flat}] ))= -\frac{d\underline{T}}{\underline{T}} \mbox{ and } \overline{\nabla}(\log (\underline{S}^{-1}[\underline{S}^{\flat}])) = - \frac{d\underline{S}}{\underline{S}}.
\]

Note that the products of the variables $\frac{\log (T^{-1}_i [T_i]^{\flat})}{t}$ and $\frac{\log(S_j^{-1}[S_j]^{\flat})}{t}$ form a locally analytic  basis of the $\widehat{\s{O}}_X$-module $\OC$ over $X_{C,\infty}$. Hence, to compute the Sen operator of $\OC$ it suffices to know the derivative of the action of $\Gamma$ on them. Let $\gamma_i$ be the standard basis of $\Gamma$, since 
\[
\gamma_{i}^a(\log(T^{-1}_k [T_k]^{\flat}))= \log(T^{-1}_k [T_k]^{\flat}) + \delta_{i,k} a t
\]
(resp. for the $S_j$), we deduce  that 
\[
\theta_{\OC}(\log \underline{T}^{-1} [T^{\flat}])= \frac{d\underline{T}}{\underline{T}}   \mbox{ and }  \theta_{\OC}(\log \underline{S}^{-1}[\underline{S}^{\flat}])=\frac{d\underline{S}}{\underline{S}},
\]
proving that $\theta_{\OC}=-\overline{\nabla}$ as wanted.
\end{proof}

\begin{proof}[Proof of  Theorem \ref{TheoLiuZhuRH}]
We first need to make a construction.  Let us suppose without loss of generality  that $\bb{L}^0$ is of rank $n$.  Define the $\GL_n(\bb{Z}_p)$-torsor 
\[
\widetilde{X}:= \mathrm{Isom}(\bb{Z}_p^{n},  \bb{L}^0).  
\]
Thus,  $\bb{L}$ is constructed from the standard representation of $\GL_n$ via the torsor $\widetilde{X}$.  In particular,  by Theorem \ref{TheoSenBundle},  $\bb{L}\otimes_{\bb{Q}_p} \widehat{\s{O}}_X$ has a Sen operator $\theta_{\bb{L}}$ arising from a map of pro-Kummer-\'etale sheaves
\[
\theta_{\widetilde{X}}: (\f{gl}_n)_{\ket}^{\vee}\otimes_{\bb{Q}_p} \widehat{\s{O}}_X \to \Omega^1_X(\log)\otimes_{\s{O}_X} \widehat{\s{O}}_X(-1),
\] 
or equivalently a map
\[
\Sen_{\widetilde{X}}: \Omega^{1}_X(\log)^{\vee}\otimes_{\s{O}_X} \widehat{\s{O}}_X(1)\to \f{gl}_{n,\ket}\otimes_{\bb{Q}_p} \widehat{\s{O}}_X
\] On the other hand,   by Proposition \ref{PropSenOperatorOC},  the Sen operator of $\OC$ is $-\overline{\nabla}$.    Thus,  by Theorem \ref{TheoGluingCaseGamma} (2) one gets that 
\[
R^i \nu_* (\bb{L}\otimes_{\bb{Q}_p} \OC)=  \nu_* H^i(-\theta_{\bb{L}\otimes \OC}, \bb{L}\otimes_{\bb{Q}_p} \OC), 
\]
where $\theta_{\bb{L}\otimes \OC}= \theta_{\bb{L}}\otimes \id_{\OC}- \id_{\bb{L}}\otimes_{\bb{Q}_p} \overline{\nabla}$. We want to show that $R^i\nu_* (\bb{L}\otimes_{\bb{Q}_p} \OC)=0$ for $i>0$ and that it is a vector bundle of rank $n$ for $i=0$.    We need the following lemma,  which is the Lie algebra incarnation of  \cite[Lemma 2.15]{LiuZhuRiemannHilbert}.  

\begin{lem}
\label{LemmaNilpotenceSenOperators}
The image of $\Sen_{\widetilde{X}}:   \Omega^1_X(\log)^{\vee} \otimes_{\s{O}_X} \widehat{\s{O}}_X(1) \to  \f{gl}_{n,\ket} \otimes_{\bb{Q}_p} \widehat{\s{O}}_X$ is contained in a nilpotent  subalgebra.  
\proof
Since $\bb{L}$ is obtained by the standard representation of $\f{gl}_n$,  it is enough to prove that the action of $\Sen_{\widetilde{X}}$ on  $\bb{L}\otimes_{\bb{Q}_p} \widehat{\s{O}}_X$ is nilpotent.   The coefficients of the characteristic polynomial of $\Sen_{\widetilde{X}}$  are given by Galois equivariant maps
\[
\sigma_i:   \Omega^1_X(\log)^{\vee} \otimes_{\s{O}_X} \widehat{\s{O}}_X(1) \to  \End_{\widehat{\s{O}}_X}(\bigwedge^i \bb{L}\otimes_{\bb{Q}_p} \widehat{\s{O}}_X) \xrightarrow{\Tr} \widehat{\s{O}}_X.
\]
 Since we are working over a discrete valuation field, Tate's theorem on the vanishing of Galois cohomology \cite[Proposition  8]{Tatepdivisible}  forces $\sigma_i=0$ for all $i=1,\ldots,n$,  proving that $\Sen_{\widetilde{X}}$ is nilpotent.  
\endproof
\end{lem}

We have proven that the Sen operator $\theta_{\bb{L}}$ of $\bb{L}\otimes_{\bb{Q}_p}\widehat{\s{O}}_X$ is nilpotent. We claim that $\n{H}(\bb{L}):= \nu_* (\bb{L}\otimes_{\bb{Q}_p} \OC)$ is a vector bundle of rank $n$ and that the natural map 
\[
\n{H}(\bb{L})\otimes_{\s{O}_X} \OC \xrightarrow{\sim}  \bb{L}\otimes_{\bb{Q}_p} \OC
\] 
is an isomorphism.  If this holds, then 
\[
R\nu_*(\bb{L}\otimes_{\bb{Q}_p} \OC)= R\nu_* (\n{H}(\bb{L})\otimes_{\s{O}_X} \OC ) = \n{H}(\bb{L})\otimes_{\s{O}_X} R\nu_* \OC = \n{H}(\bb{L}), 
\]
where the vanishing of $R\nu_* \OC$ follows from the computation of its Sen operator in Proposition \ref{PropSenOperatorOC} and Theorem \ref{TheoGluingCaseGamma} (2). The quasi-isomorphism \eqref{qwekpafma} then follows by taking cohomology for the Higgs field of $\OC$ and $\n{H}(\bb{L})$.

We are left to show that $\n{H}(\bb{L})$ is an $\s{O}_X$-vector bundle in the Kummer-\'etale site with the aforementioned properties.   This property is local in the Kummer-\'etale site, so we can assume that $X$ is affinoid and that has toric coordinates $\psi\colon X\to \bb{S}^{(e,d-e)}_K$ with coordinates $T_i,S_j$. Let $X_{\infty,C}$ be the $\Gamma$-torsor over $X_C$ obtained by taking $p$-th power roots of $T_i$ and $S_j$, let $A_{\infty}=\s{O}(X_{\infty,C})$ be the perfectoid algebra of functions of $X_{\infty,C}$. Write $\s{F}=\bb{L}\otimes_{\bb{Q}_p}\widehat{\s{O}}_X$. 

By shrinking $X$ if necessary, we can assume that $\s{F}(X_{\infty,C})$ is a free module  over $A_{\infty}$ of rank $n$. Let $v\in \s{F}(X_{\infty,C})$. Let us write $\Gamma\cong \bb{Z}_p^d$ with coordinates $\gamma_i$, and let $\theta_i$ be the Sen operator  arising from the derivation along the direction of $\gamma_i$.  By Lemma \ref{LemmaNilpotenceSenOperators} the Sen operators $\theta_i$ are nilpotent on  $\s{F}$. On the other hand, we have an explicit description as a polynomial algebra
\[
\OC(X_{\infty,C})=A_{\infty}[Z_1,\ldots, Z_d]
\]
where the $Z_i$ satisfy $\theta_i (Z_j)=\delta_{i,j}$, see the proof of Proposition \ref{PropSenOperatorOC}. Then, in the tensor product $(\s{F}\otimes_{\widehat{\s{O}}_X } \OC )(X_{\infty,C})\cong  \s{F}(X_{\infty,C})[\underline{Z}]$ we can consider the element 
\[
v^{(1)}= \sum_{k=0}^{n-1} (-1)^k \theta_1^k(v) \frac{Z_1^{k}}{k!}.  
\]
One has that $\theta_1(v^{(1)})=0$. By repeating this construction for all the operators $\theta_i$,  we find an element $\widetilde{v}\in  \s{F}(X_{\infty,C})[\underline{Z}]$ with constant term given by $v$ such that $\theta_i(\widetilde{v})=0$ for all $i$. By picking a locally analytic basis  $\{v_1,\ldots, v_n\}$ of $\s{F}(X_{\infty,C})$, the previous algorithm constructs locally analytic elements   $\widetilde{v}_1,\ldots, \widetilde{v}_n\in \s{F}(X_{\infty,C})[\underline{Z}]$ with $\theta_j(\widetilde{v}_i)=0$ for all $i,j$. Thus, the elements $\widetilde{v}_1$ are fixed by the action of an open compact subgroup $\Gamma_0\subset \Gamma$ and, after localizing in the Kummer-\'etale topology,  they give rise to a vector bundle $\n{H}(\bb{L})$ over $X_{C}$  such that 
\[
\n{H}(\bb{L})\otimes_{\s{O}_X} \OC = \s{F}\otimes_{\widehat{\s{O}}_X} \OC.
\]
Thus, we must have 
\[
\n{H}(\bb{L})=\nu_*(\s{F}\otimes_{\widehat{\s{O}}_X} \OC)
\]
proving what we wanted. 
\end{proof}

\begin{remark}
The previous argument can be extended to any $\widehat{\s{O}}_X$-vector bundle $\s{F}$.  Indeed,  we can always find locally Kummer-\'etale on $X$ a lattice $\s{F}^+ \subset \s{F}$ such that $\s{F}^+/p =^{ae} \bigoplus_{i=1}^n \s{O}^+_X/p$ as pro-Kummer-\'etale sheaves.  By  Theorem \ref{TheoGluingCaseGamma},  the sheaf $\s{F}$ admits a geometric Sen operator $\theta_{\s{F}}:  \s{F}\to \s{F}(-1)\otimes \Omega^1_X(\log)$ computing its cohomology.   Then,  since $X$ is defined over a discretely valued field,  the proof of Lemma \ref{LemmaNilpotenceSenOperators} shows that $\theta_{\s{F}}$ is actually a nilpotent operator.   This would prove in particular that Theorem \ref{TheoLiuZhuRH} holds not just for local systems but for any $\widehat{\s{O}}_X$-vector bundle in the case when $X$ is defined over a discretely valued field.  We thank Ben Heuer for pointing this out.  
\end{remark}

Finally,  let us mention the relation with the work of  Wang  \cite{wang2021padic}.   Let $X$ be a rigid analytic space over $\bb{C}_p$ admitting a liftable good reduction $\f{X}$ over $\n{O}_{\bb{C}_p}$ (this means that $\f{X}$ admits a lifting over $A_{\inf}/ \xi^2$ where $\xi=([\epsilon]-1)/([\epsilon^{\frac{1}{p}}]-1)$).   Recall the following theorem
\begin{theo}[{\cite[Theorem  5.3]{wang2021padic}}]
Let $\OC^{\dagger}$ denote the overconvergent Hodge-Tate period  sheaf  of Wang.  Let $a\geq 1/(p-1)$ and $\nu:  X_{\proket} \to  X_{\ket}$ be the projection of sites.   Then the functor 
\[
\n{H}(\n{L}):= \nu_*( \n{L }\otimes_{\widehat{\s{O}}_X} \OC^\dagger)
\]
induces an equivalence from the category of $a$-small generalized representations to the category of $a$-small  Higgs bundles.  
\end{theo}

We do not pretend to give a new proof of this statement, instead let us translate some of the main players in terms of the language used in this paper. An $a$-small generalized representation of rank $l$  is a locally free $\widehat{\s{O}}_X$-module $\n{L}$ admitting a lattice $\n{L}^{0}$ such that there is $b> a + \mathrm{val}(\rho_k)$  with $\n{L}^0/p^{b} =^{ae} (\s{O}^+_X/p^{b})^{l} $ ($\rho_K$ being an element in $\f{m}_{\bb{C}_p}$ depending on the ramification of a discretely valued subfield).   In particular, this is a relative locally analytic $\widehat{\s{O}}_X$-sheaf as in Definition \ref{DefRelativeLocAnSheaf}. The way how Wang constructs the sheaf $\OC^{\dagger}$ is by considering a particular lattice of the Faltings extension provided by the lifting of $\f{X}$ to $A_2$,  cf.  \cite[Corollary  2.19]{wang2021padic}.   Locally on coordinates,  the ring $\OC^{\dagger}$ is nothing but a suitable completion of a polynomial algebra into an overconvergent polydisc of radius $|\rho_k|$ (cf.  \cite[Theorem  2.27]{wang2021padic}).  The $a$-smallness condition is a finite rank version of the relative locally analytic condition,  where one imposes a fixed radius of analyticity.  Finally,  the decompletion  used by Wang in  \cite[\S 3.1]{wang2021padic} is the integral version of the decompletion provided by Berger-Colmez axiomatic Sen theory  \cite{BC1}.


\bibliographystyle{alpha}
\bibliography{GeoSen}


\begin{center}
•
\end{center}

\end{document}